\documentclass[11pt]{amsart}
\usepackage{amsmath, amssymb, extpfeil}

\usepackage{amsfonts}
\usepackage{mathrsfs}
\usepackage[arrow,matrix,curve,cmtip,ps]{xy}
\usepackage{tikz-cd}
\tikzset{rot90/.style={anchor=south, rotate=90, inner sep=.5mm}}

\usepackage{amsthm}
\usepackage{mathtools} 
\usepackage{array}
\usepackage{mathdots}
\usepackage{setspace}
\usepackage{mathabx}
\usepackage{multirow}
\usepackage[normalem]{ulem}

\usepackage{enumitem} 

\usepackage[top= 2.25cm, bottom=2cm, left = 2.8cm, right= 2.8cm]{geometry}
\usepackage{hyperref}

\allowdisplaybreaks

\numberwithin{equation}{section}

\numberwithin{equation}{section}

\newtheorem{theorem}[equation]{Theorem}
\newtheorem{lemma}[equation]{Lemma}
\newtheorem{proposition}[equation]{Proposition}
\newtheorem{corollary}[equation]{Corollary}

\newtheorem{conjecture}[equation]{Conjecture}

\newtheorem*{theorem*}{Theorem}
\newtheorem{hypothesis}[]{Hypothesis}

\theoremstyle{remark}
\newtheorem{remark}[equation]{Remark}
\newtheorem{definition}[equation]{Definition}

\theoremstyle{definition}
\newtheorem{clause}[equation]{}

\newtheorem{assumption}[equation]{Assumption}

\newcommand{\q}{\psi}

\renewcommand{\L}[1]{{}^L#1}

\newcommand{\D}[1]{\widehat{#1}}

\renewcommand{\a}{\alpha}

\renewcommand{\Im}{\text{Im}\,}

\newcommand{\Ind}{\operatorname{Ind}}

\newcommand{\C}{\mathbb{C}}

\newcommand{\R}{\mathbb{R}}

\newcommand{\N}[1]{\frak{N}_{#1}}

\newcommand{\Hom}{\operatorname{Hom}}

\newcommand{\Rep}{\operatorname{Rep}}

\newcommand{\ul}{\underline{l}}

\newcommand{\ueta}{\underline{\eta}}

\renewcommand{\N}{\mathbb N}

\newcommand{\la}{\left\langle}
\newcommand{\ra}{\right\rangle}
\newcommand{\gl}{{\mathfrak {gl}}}
\newcommand{\GL}{GL}

\newcommand{\FDR}{{S}}

\renewcommand{\Re}{{\mathrm{Re}}}
\newcommand{\sgn}{{\operatorname{sgn}}}

\newcommand{\ES}{{\mathcal {E}}}
\newcommand{\MRV}{{\mathcal D}}
\newcommand{\AV}{{\mathrm{pure}}}
\newcommand{\bp}{\mathrm{bp}}
\newcommand{\Std}{{\mathrm {Std}}}
\newcommand{\g}{{\mathfrak g}}
\renewcommand{\k}{{\mathfrak k}}
\newcommand{\h}{{\mathfrak h}}
\renewcommand{\t}{{\mathfrak t}}
\renewcommand{\a}{{\mathfrak a}}
\newcommand{\LAq}{{\mathfrak q}}

\renewcommand{\l}{{\mathfrak l}}
\newcommand{\m}{{\mathfrak m}}
\newcommand{\n}{{\mathfrak n}}
\renewcommand{\u}{{\mathfrak u}}
\renewcommand{\b}{{\mathfrak b}}

\newcommand{\bk}{d}
\newcommand{\bh}{c}

\newcommand{\LI}{{\mathrm B}}
\newcommand{\MI}{{\mathrm D}}
\newcommand{\UI}{{\mathrm C}}
\newcommand{\PI}{{\mathrm P}}
\newcommand{\QI}{{\mathrm Q}}

\newcommand{\inv}{^{-1}}
\renewcommand{\le}{\leqslant}
\renewcommand{\ge}{\geqslant}
\newcommand{\widesim}[2][1.5]{\mathrel{\overset{#2}{\scalebox{#1}[1]{$\sim$}}}}
\renewcommand{\forany}{\,\forall\,} 

\newcommand{\surj}{\twoheadrightarrow}

\newcommand{\surjects}{\xtwoheadrightarrow{}}

\newcommand{\inj}{\hookrightarrow}

\newcommand{\injects}{\xhookrightarrow{\;\;\;\;\;}}

\newcommand{\bij}{\xrightarrow{\raisebox{-1.2ex}[0ex][1ex]{$\sim$}}}

\newcommand{\bijects}{\xrightarrow{\,\raisebox{-1.2ex}[0ex][1ex]{$\widesim[1]{}$\,}}}

\newcommand{\aro}{\xrightarrow{\;\;\;\;\;}}

\newcommand{\acts}{\mathrel{\reflectbox{$\righttoleftarrow$}}}
\newcommand{\racts}{\righttoleftarrow}

\DeclareMathOperator{\diag}{diag}
\DeclareMathOperator{\Ad}{Ad}
\DeclareMathOperator{\Fr}{Fr}
\DeclareMathOperator{\supp}{supp}

\newcommand{\BC}{\mathbb{C}}

\newcommand{\BQ}{\mathbb{Q}}
\newcommand{\BR}{\mathbb{R}}

\newcommand{\BZ}{\mathbb{Z}}

\newcommand{\cD}{\mathcal{D}}

\newcommand{\cI}{\mathcal{I}}

\newcommand{\cR}{\mathcal{R}}

\newcommand{\cV}{\mathcal{V}}

\newcommand{\cX}{\mathcal{X}}

\newcommand{\fg}{\mathfrak{g}}

\newcommand{\fh}{\mathfrak{h}}

\newcommand{\bm}{\mathbf{m}}
\newcommand{\fm}{\mathfrak{m}}
\newcommand{\fp}{\mathfrak{p}}

\newcommand{\ft}{\mathfrak{t}}

\newcommand{\SL}{SL}

\newcommand{\Sp}{Sp}
\newcommand{\SO}{SO}

\newcommand{\Span}{\mathbf{Span}}

\newcommand{\fgl}{\mathfrak{gl}}

\begin{document}

\title[Relating real and $p$-adic A-packets]{Relating Arthur packets of real unitary groups and $p$-adic symplectic and orthogonal groups}

\author{Taiwang DENG}
\address{Beijing Institute of Mathematical Sciences and Applications, Huairou District Beijing 101408, China}
\email{dengtaiw@bimsa.cn}

\author{Chang HUANG}
\address{Department of Mathematics \\ Tsinghua University, Haidian District Beijing 100084, China}
\email{hc21@mails.tsinghua.edu.cn}

\author{Bin XU}
\address{Yau Mathematical Sciences Center and Department of Mathematics \\  Tsinghua University, Haidian District Beijing 100084, China}
\email{binxu@tsinghua.edu.cn}

\author{Qixian ZHAO}
\address{Yau Mathematical Sciences Center \\ Tsinghua University, Haidian District Beijing 100084, China}
\email{zhao{\_}qixian@tsinghua.edu.cn}

\begin{abstract}
We establish an explicit correspondence of certain Arthur packets between real unitary groups and $p$-adic symplectic or orthogonal groups. This allows one to compute Arthur packets of real unitary groups by translating results from the $p$-adic side. A main ingredient in our proof is an explicit relation between Zuckerman's translation functor on the real side and the Jacquet functor on the $p$-adic side. To achieve this, we construct a correspondence of stacks of Langlands parameters with fixed infinitesimal characters between the relevant real and $p$-adic groups. Our approach also allows one to relate the Kazhdan-Lusztig polynomials and the microlocal geometry between real and $p$-adic sides. 
\end{abstract}

\maketitle

\tableofcontents

\section{Introduction}
\label{sec: introduction}

A fundamental problem in the representation theory of real or $p$-adic reductive groups is the classification of their irreducible unitary representations, i.e. their unitary dual. Since the tempered representations are all unitary, the main difficulty is to classify the nontempered unitary representations. There is an important class of such representations coming from automorphic forms, or more precisely the local components of automorphic representations. As suggested by the trace formula, Arthur \cite{Arthur:1984} conjectured that these local components can be grouped into packets enjoying nice analytic properties, namely the endoscopic character relation. These packets are nowadays called Arthur packets or $A$-packets, and they are parameterized by the so-called Arthur parameters. In fact, this additional packet structure is useful for their classification, and thus one would like to understand the packets first. 

There have been several different but conjecturally compatible constructions of Arthur packets, each with its own advantages and limitations. 
For real reductive groups, Adams-Barbasch-Vogan \cite{ABV:1992} gave a construction of Arthur packets through the microlocal geometry of Langlands parameters space (abbreviated as ABV-space), and we will refer to them as the ABV-packets for distinction, denoted by $\Pi_{\psi^{\mathbb{R}}}^{{\rm ABV}}$, where $\psi^\BR$ denotes a real Arthur parameter. The unitarity of these representations is not clear from this definition. In the case of $Sp(2n, \mathbb{R})$, $SO(p,q)$, $U(p,q)$, Arthur \cite{Arthur:2013}, Mok \cite{Mok:2014} in the quasisplit case and Moeglin-Renard \cite{MoeglinRenard:2020} \cite{MoeglinRenard:2018} \cite{MR:2019} in the nonquasisplit case constructed the Arthur packets using harmonic analysis, more precisely the (twisted) endoscopy theory. We denote these simply by $\Pi_{\psi^{\mathbb{R}}}$. The unitarity of these packets follows from lifting them to the automorphic representations. Again for these groups, Moeglin-Renard \cite{MoeglinRenard:2020} gave an equivalent algebraic construction using cohomological induction and parabolic induction starting from the unipotent case \cite{Moeglin:2017}. Lastly, dropping the restriction on groups but instead requiring the representations are cohomological, Adams-Johnson \cite{AJ:1987} proposed a construction using cohomological induction from one dimensional characters of certain twisted Levi subgroups, resulting in the so-called Adams-Johnson packets, denoted by $\Pi_{\psi^{\mathbb{R}}}^{{\rm AJ}}$. These constructions are compatible in the sense that $\Pi_{\psi^{\mathbb{R}}}^{{\rm ABV}} = \Pi_{\psi^{\mathbb{R}}}$ \cite{AAM:2024} \cite{AM:2024} and $\Pi_{\psi^{\mathbb{R}}} = \Pi_{\psi^{\mathbb{R}}}^{{\rm AJ}}$ \cite{AMR:2018} whenever both sides are defined.

For $p$-adic reductive groups, there is a parallel story. First there is also a construction of A-packets based on the local Langlands correspondence together with the microlocal geometry of Langlands parameters space (abbreviated as Vogan variety) given by Vogan \cite{Vogan:1993} and reinterpreted in \cite{CFMMX:2022} using vanishing cycles. By the analogy with ABV's construction for real groups, they are again referred to as ABV-packets, denoted by $\Pi_{\psi^{\mathbb{Q}_p}}^{{\rm ABV}}$, where $\psi^{\BQ_p}$ is a $p$-adic A-parameter. In the case of symplectic or special orthogonal groups, Arthur \cite{Arthur:2013} in the quasisplit case and Moeglin \cite{Moeglin:2009} based on \cite{MoeglinRenard:2018} in the nonquasisplit case constructed the Arthur packets using the (twisted) endoscopy theory, denoted by $\Pi_{\psi^{\mathbb{Q}_p}}$. Again for these groups, Moeglin \cite{Moeglin:2009} gave an algebraic construction using certain recursive formulas involving parabolic induction and Jacquet restriction. It is expected that Moeglin's work can also be extended to $p$-adic unitary groups based on \cite{Mok:2014} \cite{K-M-S-W:2014}. While it is unclear what should be the analogue of cohomological induction for $p$-adic groups, Moeglin distinguished a special type of Arthur packets in her construction, which will be shown to correspond to the Adams-Johnson packets in our present work. Finally, there is an ongoing project of Cunningham-Hazeltine-Liu-Lo-Ray-Xu aiming to show that $\Pi_{\psi^{\mathbb{Q}_p}}^{{\rm ABV}} = \Pi_{\psi^{\mathbb{Q}_p}}$ when both sides are defined.

Our motivation mostly lies in the explicit computation of A-packets. This is difficult from the geometric (i.e. ABV) point of view (Some examples are given on the $p$-adic side in \cite{CFMMX:2022}). Through the approach of endoscopy theory, it is possible to compute some examples using the computer program \textsf{Atlas}. In order to obtain general results, the algebraic approach seems more promising. For example, the Adams-Johnson packets can be computed explicitly by the theory of cohomological induction in the good range, cf. \cite{VZ:1984} \cite{KV:1995}. For more general A-packets of real groups, one also needs to consider the cohomological induction in the weakly fair range, cf. \cite{MoeglinRenard:2020} \cite{Trapa:2001}, which can be very difficult to compute. On the other hand, for $p$-adic symplectic or special orthogonal groups, Moeglin's construction leads to a complete understanding of the packets structure cf. \cite{Moeglin1:2011} \cite{Xu:Comb} \cite{Atobe:2022} \cite{Atobe:2023} \cite{HLL:2022}. Our motivation is to translate these results in the $p$-adic world to real classical groups. In this paper, we focus on real unitary groups, where the translation is most clean.

\subsection{Main results}
\label{subsec: main}

For a positive integer $n$, let $G$ be $U(\frac{n}{2}, \frac{n}{2})$ if $n$ is even or $U(\frac{n-1}{2}, \frac{n+1}{2})$ if $n$ is odd. Let $\D{G} = GL_n(\BC)$, the Langlands dual group of $G$. Let $W_{\mathbb{R}}$ (resp. $W_{\mathbb{C}}$) be the Weil group of $\mathbb{R}$ (resp. $\mathbb{C}$), i.e. $W_{\mathbb{C}} = \mathbb{C}^{\times}$ and 
\(
W_{\mathbb{R}} = \langle \mathbb{C}^{\times}, j \rangle
\)
subject to the relations $j^{2} = -1$ and $jzj^{-1} = \bar{z}$ for any $z \in \mathbb{C}^{\times}$. There is a short exact sequence 
\[
1 \rightarrow W_{\mathbb{C}} \rightarrow W_{\mathbb{R}} \rightarrow {\rm Gal}(\mathbb{C}/\mathbb{R}) \rightarrow 1.
\]
An Arthur parameter of $G$ is a  continuous group homomorphism
\[
\psi^{\mathbb{R}}: W_{\mathbb{R}} \times SL_{2}(\mathbb{C}) \rightarrow \L{G} := \D G \rtimes {\rm Gal}(\mathbb{C}/\mathbb{R})
\]
such that it is compatible with projections to ${\rm Gal}(\mathbb{C}/\mathbb{R})$ on both sides and $\psi^{\mathbb{R}}(\mathbb{C}^{\times})$ is bounded and consists of semisimple elements. Under base change to $\mathbb{C}$, we get a representation 
\[
BC(\psi^{\mathbb{R}}): W_{\mathbb{C}} \times SL_{2}(\mathbb{C}) \rightarrow \D G
\]
which is conjugate self-dual of  parity $(-1)^{n + 1}$ (cf. \cite[Lemma 2.2.1]{Mok:2014}). The isomorphism class o $BC(\psi^{\BR})$ determines $\psi^{\mathbb{R}}$ up to $\D{G}$-conjugation. We say $\psi^{\mathbb{R}}$ has good parity if the irreducible constituents of $BC(\psi^{\mathbb{R}})$ are all conjugate self-dual of the same parity as $BC(\psi^{\BR})$. Under the good parity assumption, we can write
\begin{align}
\label{eq: real A-parameter}
BC(\psi^{\mathbb{R}}) = \bigoplus_{i = 1}^{r} \, (z/\bar{z})^{\frac{k_i}{2}} \boxtimes S_{m_i}, \quad z \in \mathbb{C}^{\times}
\end{align}
for some $k_i \in \mathbb{Z}, m_i \in \mathbb{Z}_{> 0}$ such that $n \equiv k_i + m_i$ mod $2$, where $S_{m_i}$ denotes the $m_i$-dimensional irreducible representation of $SL_{2}(\mathbb{C})$. We define the \textit{pure Arthur packet} of $G$ to be
\[
\Pi_{\psi^{\mathbb{R}}}^{\rm pure}(G) := \bigsqcup_{(p, q): p + q = n} \Pi_{\psi^{\mathbb{R}}}(U(p,q)).
\]
We would like to compare it with certain pure Arthur packet for a split classical group $H$ over $\mathbb{Q}_p$. To do so, we let $\delta \in \mathbb{Z}$ be any number such that $k_i + \delta > m_i$ for all $i$. 
Let 
\[
N = \sum_{i = 1}^{r} (k_i + \delta)m_i.
\]
\begin{itemize}
\item
If $\delta$ is odd, then $N \equiv n$ mod $2$. We take $H$ (and its Langlands dual group $\D H$) to be
\begin{center}
\begin{tabular}{| c | c | c |}
     \hline
      $n$         &  $H$                     & $\D{H}$               \\
     \hline
      odd     &  $Sp(N-1)$           &      $SO(N, \mathbb{C})$  \\
     \hline
      even   &   $SO(N+1)$         & $Sp(N, \mathbb{C})$ \\
      \hline
\end{tabular}
\end{center}
\item
If $\delta$ is even, then $N$ is even. We take $H$ to be
\begin{center}
\begin{tabular}{| c | c | c |}
     \hline
      $n$     &  $H$                       &      $\D{H}$               \\
     \hline
      odd     &  $SO(N+1)$           &      $Sp(N, \mathbb{C})$  \\
     \hline
      even   &   $O(N)$                &       $O(N, \mathbb{C})$ \\
      \hline
\end{tabular}
\end{center}
(Note only in Subsection~\ref{subsec: LLC p-adic} ~\ref{subsec: comparison} we will denote the split $p$-adic special even orthogonal group $SO(N)$ by $H$ and the full even orthogonal group $O(N)$ by $H^{+}$.)
\end{itemize}
Let $W_{\mathbb{Q}_p}$ be the Weil group of $\mathbb{Q}_p$. Since $H$ is split, an Arthur parameter of $H$ is equivalent to a continuous group homomorphism
\[
\psi^{\mathbb{Q}_p}: W_{\mathbb{Q}_p} \times SL_{2}(\mathbb{C}) \times SL_{2}(\mathbb{C}) \rightarrow \D{H},
\]
such that $\psi^{\mathbb{Q}_p}(W_{\mathbb{Q}_p})$ is bounded and consists of semisimple elements of the identity component $\D{H}^0$. Composing it with the standard representation ${\rm std}_{\D{H}}$ of $\D{H}$, we get a representation
\[
{\rm std}_{\D{H}} \circ \psi^{\mathbb{Q}_p}: W_{\mathbb{Q}_p} \times SL_{2}(\mathbb{C}) \times SL_{2}(\mathbb{C}) \rightarrow GL_{N}(\mathbb{C}),
\]
which is self-dual of orthogonal type (paity $+1$) or
symplectic type (parity $-1$) depending on the type of $\D{H}$. The isomorphism class of ${\rm std}_{\D{H}} \circ \psi^{\mathbb{Q}_p}$ determines $\psi^{\mathbb{Q}_p}$ up to $\D{H}$-conjugation. We say $\psi^{\mathbb{Q}_p}$ has good parity if the irreducible constituents of ${\rm std}_{\D{H}} \circ \psi^{\mathbb{Q}_p}$ are all self-dual of the same parity as ${\rm std}_{\D{H}} \circ \psi^{\mathbb{Q}_p}$. In this paper, we will always assume that $\psi^{\mathbb{Q}_p}$ is trivial on $W_{\mathbb{Q}_p}$. Then under the good parity assumption, we can write
\begin{align}
\label{eq: p-adic A-parameter}
{\rm std}_{\D{H}} \circ \psi^{\mathbb{Q}_p} = \bigoplus_{i = 1}^{r'} 1_{W_{\mathbb{Q}_p}} \boxtimes S_{a_i} \boxtimes S_{b_i}
\end{align}
where $a_i + b_i$ is even if $\D{H}$ is orthogonal, or is odd if $\D{H}$ is symplectic. We define the \textit{pure Arthur packet} to be 
\[
\Pi_{\psi^{\mathbb{Q}_p}}^{\rm pure}(H) := \begin{cases} \Pi_{\psi^{\mathbb{Q}_p}}(H) & \text{ if $H$ is symplectic, } \\
\Pi_{\psi^{\mathbb{Q}_p}}(H) \sqcup \Pi_{\psi^{\mathbb{Q}_p}}(H') & \text{ if $H$ is orthogonal. }
\end{cases} 
\]
Here $H'$ is the non-split inner form of $H$ defined by some quadratic form of discriminant one. For the comparison with $\psi^{\mathbb{R}}$ as in \eqref{eq: real A-parameter}, we take $r = r'$, $a_i = k_i + \delta$ and $b_i = m_i$. Note the condition $k_i + \delta > m_i$ implies that $a_i > b_i$ for all $i$.

In order to describe the relation between $\Pi_{\psi^{\mathbb{R}}}^{\rm pure}(G)$ and $\Pi_{\psi^{\mathbb{Q}_p}}^{\rm pure}(H)$, we first relate two larger sets, namely the sets of isomorphism classes of irreducible admissible representations with the same infinitesimal characters as those of $\psi^{\mathbb{R}}$ and $\psi^{\mathbb{Q}_p}$, respectively. The infinitesimal character of $\psi^{\mathbb{R}}$ is represented by $\lambda^{\mathbb{R}} = (\lambda_1, \cdots, \lambda_n) \in \mathbb{C}^{n}$, satisfying
\begin{align*}
\{\lambda_1, \cdots, \lambda_n \} & = \bigsqcup_{i = 1}^{r} \{ -\frac{m_i - 1}{2} + \frac{k_i}{2}, \cdots, \frac{m_i - 1}{2} + \frac{k_i}{2} \}.
\end{align*}
as multisets. Under the condition $k_i + \delta > m_i$, we have $\lambda_{i} > (1 - \delta)/2$ for all $i$. The good parity assumption implies the good parity assumption on $\lambda^{\mathbb{R}}$, i.e., $\lambda_i \in \mathbb{Z}$ if $n$ is odd and $\lambda_i \in \frac{1}{2}\mathbb{Z} \backslash \mathbb{Z}$ if $n$ is even. In terms of $\lambda$, we have 
\[
N = \sum_{i=1}^{n} (2\lambda_i + \delta).
\]
Denote the infinitesimal character of $\psi^{\mathbb{Q}_p}$ by $\lambda^{\mathbb{Q}_p}$. Let $\Phi(G)_{\lambda^{\mathbb{R}}}$ (resp. $\Phi(H)_{\lambda^{\mathbb{Q}_p}}$) be the set of $\D{G}$ (resp. $\D{H}$) -conjugacy classes of Langlands parameters of $G$ (resp. $H$) with infinitesimal character $\lambda^{\mathbb{R}}$ (resp. $\lambda^{\mathbb{Q}_p}$). We define an injection
\[
\iota: \Phi(G)_{\lambda^{\mathbb{R}}} \rightarrow \Phi(H)_{\lambda^{\mathbb{Q}_p}}
\]  
as follows. For any $\phi^{\mathbb{R}} \in \Phi(G)_{\lambda^{\mathbb{R}}}$, it can be characterized by its base change to $\mathbb{C}$, a conjugate self-dual representation of $W_{\mathbb{C}}$ of parity $(-1)^{n+1}$:
\begin{align}
\label{eq: real L-parameter}
BC(\phi^{\mathbb{R}}) = \bigoplus_{i = 1}^{n} (z/\bar{z})^{t_i} (z \bar{z})^{s_i}, \quad z \in \mathbb{C}^{\times}
\end{align}
where $t_i \in \frac{1}{2}\mathbb{Z}$, $s_{i} \in \mathbb{C}$ are such that 
\[
\{\lambda_1, \cdots, \lambda_n\} = \{t_1 + s_1, \cdots, t_n + s_n\}
\]
as multisets. Note if $(z/\bar{z})^{t} (z \bar{z})^{s}$ appears in \eqref{eq: real L-parameter}, then so does its conjugate dual $(z/\bar{z})^{t} (z \bar{z})^{-s}$. Hence $s_{i} \in \frac{1}{2}\mathbb{Z}$ and $t_i > (1 - \delta)/2$. Similarly, any $\phi^{\mathbb{Q}_p} \in \Phi(H)_{\lambda^{\mathbb{Q}_p}}$ can be characterized by its composition with ${\rm std}_{\D{H}}$, which is a representation of $W_{\mathbb{Q}_p} \times SL_2(\mathbb{C})$. We set $\iota(\phi^\BR) = \phi^{\BQ_p}$ where 
\[
{\rm std}_{\D{H}} \circ \phi^{\mathbb{Q}_p} = \bigoplus_{i=1}^{n} |\cdot|^{s_i}_{W_{\mathbb{Q}_p}} \boxtimes S_{2t_i + \delta}. 
\]

The map $\iota$ can be extended to a map between the sets of complete Langlands parameters. These are pairs consisting of a Langlands parameter and an irreducible character of the component group of its centralizer. Let $S_{\phi^{\mathbb{R}}} = Z_{\D{G}}(\phi^{\mathbb{R}})$ and $S_{\phi^{\mathbb{Q}_p}} = Z_{\D{H}}(\phi^{\mathbb{Q}_p})$ be the centralizers of Langlands parameters. When $H$ is symplectic, we also consider $S^{+}_{\phi^{\mathbb{Q}_p}} = Z_{O(N, \mathbb{C})}(\phi^{\mathbb{Q}_p})$. To write down these centralizer groups, we group the isomorphic irreducible components in $BC(\phi^{\mathbb{R}})$ and ${\rm std}_{\D{H}} \circ \phi^{\mathbb{Q}_p}$ as follows
\[
BC(\phi^{\mathbb{R}}) = \bigoplus_{i \in K} \Big((z/\bar{z})^{t_i} (z \bar{z})^{s_i}\Big)^{\oplus l_i}
\]
and
\[
{\rm std}_{\D{H}} \circ \phi^{\mathbb{Q}_p} = \bigoplus_{i \in K} \Big( |\cdot|^{s_i}_{W_{\mathbb{Q}_p}} \boxtimes S_{2t_i + \delta} \Big)^{\oplus l_i}
\]
where $K$ is the index set of distinct irreducible components and the $l_i$'s denote the multiplicities. Since $BC(\phi^{\mathbb{R}})$ is conjugate self-dual of parity $(-1)^{n+1}$, we can decompose 
\begin{align}
\label{eq: index set}
K = I^{+} \sqcup I^{-} \sqcup J \sqcup J^{\vee}
\end{align}
where $I^{+}$ (resp. $I^{-}$) indexes conjugate self-dual irreducible components of the same parity as (resp. opposite parity to) $BC(\phi^{\mathbb{R}})$ and $J$ indexes a subset of non-conjugate self-dual irreducible components, whose conjugate dual are indexed by $J^{\vee}$. It is easy to check that the same decomposition applies to $\phi^{\mathbb{Q}_p}$ concerning the parities. Under the good parity assumption on $\lambda^{\mathbb{R}}$, we have $I^{-} = \emptyset$,
\[
I^{+} = \{i \in K \, | \, s_i = 0\}
\]
and we can choose
\[
J = \{i \in K \, | \, s_i > 0\}.
\]
Then as in \cite[(2.4.13)]{Mok:2014},
\[
S_{\phi^{\mathbb{R}}} \cong \prod_{i \in I^{+}} O(l_i, \mathbb{C}) \times \prod_{j \in J} GL(l_j).
\] 
On the $p$-adic side, we can get similar results (cf. \cite[(1.4.8)]{Arthur:2013}):
\begin{itemize}

\item 
when $H$ is orthogonal, 
\[
S_{\phi^{\mathbb{Q}_p}} \cong \prod_{i \in I^{+}} O(l_i, \mathbb{C}) \times  \prod_{j \in J} GL(l_j).
\] 

\item
when $H$ is symplectic, 
\[
S_{\phi^{\mathbb{Q}_p}} \cong (\prod_{i \in I^{+}} O(l_i, \mathbb{C}))^{+} \times  \prod_{j \in J} GL(l_j),
\] 
where $(\cdot)^{+}$ means that we are taking $(g_{i})_{i \in I^{+}}$ such that 
\[
\prod_{i \in I^{+}} {\rm det} (g_{i})^{2t_i + \delta} = 1
\] 
and
\[
S^{+}_{\phi^{\mathbb{Q}_p}} \cong \prod_{i \in I^{+}} O(l_i, \mathbb{C}) \times  \prod_{j \in J} GL(l_j)
\] 

\end{itemize}
Since these isomorphisms are unique up to inner automorphisms of the centralizer groups, we get an isomorphism $S_{\phi^{\mathbb{R}}} \cong S_{\phi^{\mathbb{Q}_p}}$ when $H$ is orthogonal (resp. $S_{\phi^{\mathbb{R}}} \cong S^{+}_{\phi^{\mathbb{Q}_p}}$ when $H$ is symplectic), unique up to inner automorphisms. Let $A_{\phi^{\mathbb{R}}} = \pi_0(S_{\phi^{\mathbb{R}}})$ and $A_{\phi^{\mathbb{Q}_p}} = \pi_0(S_{\phi^{\mathbb{Q}_p}})$, which are both abelian. When $H$ is orthogonal, we obtain a canonical isomorphism $A_{\phi^{\mathbb{R}}} \cong A_{\phi^{\mathbb{Q}_p}}$. When $H$ is symplectic, let $A^{+}_{\phi^{\mathbb{Q}_p}} = \pi_0(S^{+}_{\phi^{\mathbb{Q}_p}})$ (also abelian), and we get a canonical isomorphism $A_{\phi^{\mathbb{R}}} \cong A^{+}_{\phi^{\mathbb{Q}_p}}$. Note $A^{+}_{\phi^{\mathbb{Q}_p}} \cong A_{\phi^{\mathbb{Q}_p}} \times Z(O(N, \mathbb{C}))$. In either case, there is a canonical map of characters $\D{A}_{\phi^{\mathbb{R}}} \rightarrow \D{A}_{\phi^{\mathbb{Q}_p}}$ allowing one to  extend $\iota$ to a map between of complete Langlands parameters. By the local Langlands correspondence (cf. Subsection~\ref{subsec: hypothesis LLC}), this extneded map induces a map between isomorphism classes of irreducible representations with fixed infinitesimal characters  
\[
\tilde{\iota}: \Pi_{{\rm pure}}(\lambda^{\mathbb{R}}, G) \rightarrow \Pi_{{\rm pure}}(\lambda^{\mathbb{Q}_p}, H),
\]
where
\[
\Pi_{{\rm pure}}(\lambda^{\mathbb{R}}, G) := \bigsqcup_{(p,q) : p+q = n}\Pi(U(p,q))_{\lambda^{\mathbb{R}}} 
\]
and 
\[
\Pi_{{\rm pure}}(\lambda^{\mathbb{Q}_p}, H) := \begin{cases} \Pi(H)_{\lambda^{\mathbb{Q}_p}} \quad & \text{ if $H$ is symplectic, } \\ 
\Pi(H)_{\lambda^{\mathbb{Q}_p}} \bigsqcup \Pi(H')_{\lambda^{\mathbb{Q}_p}} \quad & \text{ if $H$ is orthogonal. }
\end{cases}
\]
When $H$ is orthogonal, $\tilde{\iota}$ is an injection. When $H$ is symplectic, $\tilde{\iota}$ induces an injection on the quotient that identifies $\Pi(U(p,q))$ with $\Pi(U(q, p))$. Note that $\Pi_{\psi^\BR}^{pure}(G)$ and $\Pi_{\psi^{\BQ_p}}^{pure}(H)$ are subsets of $\Pi_{{\rm pure}}(\lambda^{\mathbb{R}}, G)$ and $\Pi_{{\rm pure}}(\lambda^{\mathbb{Q}_p}, H)$, respectively. Our first main theorem is:

\begin{theorem}
\label{thm: comparison of A-packets}
When $H$ is orthogonal, $\tilde{\iota}$ restricts to a bijection $\Pi_{\psi^{\mathbb{R}}}^{\rm pure}(G) \xrightarrow{\simeq} \Pi_{\psi^{\mathbb{Q}_p}}^{\rm pure}(H)$. When $H$ is symplectic, $\tilde{\iota}$ restricts to a bijection
\[
\bigsqcup_{(p, q): p + q = n,\, p \equiv \frac{n-1}{2} \, {\rm mod} \, 2 } \Pi_{\psi^{\mathbb{R}}}(U(p,q)) \xrightarrow{\simeq} \Pi_{\psi^{\mathbb{Q}_p}}^{\rm pure}(H).
\]
\end{theorem}
A special case of this theorem was conjectured in \cite[Conjecture 1.2.1]{DHXZ:GLn}.

In endoscopy theory, Arthur packets are also equipped with maps 
\begin{align}
\label{eq: endoscopy real}
\Pi_{\psi^{\mathbb{R}}}^{\rm pure}(G) \xrightarrow{\epsilon^{\mathbb{R}}} {\rm Rep}(A_{\psi^{\mathbb{R}}})
\end{align}
and
\begin{align}
\label{eq: endoscopy p-adic}
\Pi_{\psi^{\mathbb{Q}_p}}^{\rm pure}(H) \xrightarrow{\epsilon^{\mathbb{Q}_p}} {\rm Rep}(A_{\psi^{\mathbb{Q}_p}}),
\end{align}
where $A_{\psi^{\mathbb{R}}} = \pi_0(Z_{\D{G}}(\psi^{\mathbb{R}}))$ and $A_{\psi^{\mathbb{Q}_p}} = \pi_0(Z_{\D{H}}(\psi^{\mathbb{Q}_p}))$ (cf. \cite{Arthur:2013} \cite{Moeglin:2009}\cite{Moeglin1:2011} \cite{MoeglinRenard:2018} \cite{Mok:2014} \cite{MR:2019}). The two component groups here can be related in the same way as in the case of Langlands parameters. Let $S_{\psi^{\mathbb{R}}} = Z_{\D{G}}(\psi^{\mathbb{R}})$ and $S_{\psi^{\mathbb{Q}_p}} = Z_{\D{H}}(\psi^{\mathbb{Q}_p})$. When $H$ is symplectic, consider also $S^{+}_{\psi^{\mathbb{Q}_p}} = Z_{O(N, \mathbb{C})}(\psi^{\mathbb{Q}_p})$. Group the isomorphic irreducible components in $BC(\psi^{\mathbb{R}})$ and ${\rm std}_{\D{H}} \circ \psi^{\mathbb{Q}_p}$ as 
\[
BC(\psi^{\mathbb{R}}) = \bigoplus_{i \in K} \Big( (z/\bar{z})^{\frac{k_{i}}{2}} \boxtimes S_{m_i} \Big)^{\oplus l_i}
\]
and
\[
{\rm std}_{\D{H}} \circ \psi^{\mathbb{Q}_p} = \bigoplus_{i \in K} \Big(1_{W_{\mathbb{Q}_p}} \boxtimes S_{a_i} \boxtimes S_{b_i} \Big)^{\oplus l_i}.
\]
The index set $K$ of distinct irreducible components can be decomposed according to parities as in \eqref{eq: index set}. By
the good parity assumption on $\psi^{\mathbb{R}}$, we have $K = I^{+}$, and the same is true for $\psi^{\mathbb{Q}_p}$. Then we have
\begin{align}
\label{eq: real centralizer}
S_{\psi^{\mathbb{R}}} \cong \prod_{i \in I^{+}} O(l_i, \mathbb{C}) 
\end{align}
(cf. \cite[(2.4.13)]{Mok:2014}). On the $p$-adic side, we have (cf. \cite[(1.4.8)]{Arthur:2013}):
\begin{itemize}
\item
when $H$ is orthogonal, 
\[
S_{\psi^{\mathbb{Q}_p}} \cong \prod_{i \in I^{+}} O(l_i, \mathbb{C}).
\] 

\item
when $H$ is symplectic, 
\[
S_{\psi^{\mathbb{Q}_p}} \cong (\prod_{i \in I^{+}} O(l_i, \mathbb{C}))^{+}
\] 
where $(\cdot)^{+}$ means that we are taking $(g_{i})_{i \in I^{+}}$ such that 
\[
\prod_{i \in I^{+}} {\rm det} (g_{i})^{a_i b_i} = 1
\] 
and we also have
\[
S^{+}_{\psi^{\mathbb{Q}_p}} \cong \prod_{i \in I^{+}} O(l_i, \mathbb{C}).
\] 
\end{itemize}
Again these isomorphisms are unique up to inner automorphisms of the centralizer groups, and we have isomorphisms $S_{\psi^{\mathbb{R}}} \cong S_{\psi^{\mathbb{Q}_p}}$ and $A_{\psi^{\mathbb{R}}} \cong A_{\psi^{\mathbb{Q}_p}}$ when $H$ is orthogonal (resp. $S_{\psi^{\mathbb{R}}} \cong S^{+}_{\psi^{\mathbb{Q}_p}}$ and $A_{\psi^{\mathbb{R}}} \cong A^{+}_{\psi^{\mathbb{Q}_p}}$ when $H$ is symplectic), unique up to inner automorphism. Note that $A^{+}_{\psi^{\mathbb{Q}_p}} \cong A_{\psi^{\mathbb{Q}_p}} \times Z(O(N, \mathbb{C}))$. In either case, we have projections of finite dimensional representations ${\rm Rep}(A_{\psi^{\mathbb{R}}}) \rightarrow {\rm Rep}(A_{\psi^{\mathbb{Q}_p}})$ and characters $\D{A}_{\psi^{\mathbb{R}}} \rightarrow \D{A}_{\psi^{\mathbb{Q}_p}}$. 

\begin{theorem}
\label{thm: comparison of endoscopic maps}
The following diagram commutes
\[
\xymatrix{
\Pi_{\psi^{\mathbb{R}}}^{\rm pure}(G) \ar[d]_{\epsilon^{\mathbb{R}}} \ar[r]^{\tilde{\iota}} &  \Pi_{\psi^{\mathbb{Q}_p}}^{\rm pure}(H) \ar[d]^{\epsilon^{\mathbb{Q}_p}} \\
{\rm Rep}(A_{\psi^{\mathbb{R}}}) \ar[r] & {\rm Rep}(A_{\psi^{\mathbb{Q}_p}}).
}
\]
\end{theorem}

\begin{remark}
In the setting of this theorem, it has been proved by Moeglin \cite{Moeglin1:2011} (resp. Moeglin-Renard \cite{MR:2019}) that the image of $\epsilon^{\mathbb{Q}_p}$ (resp. $\epsilon^{\mathbb{R}}$) lies in $\D{A}_{\psi^{\mathbb{Q}_p}})$ (resp. $\D{A}_{\psi^{\mathbb{R}}})$).
\end{remark}


For a general Arthur parameter $\psi^{\mathbb{R}}$ of $G$, we have 
\[
BC(\psi^{\mathbb{R}}) = \bigoplus_{i \in K} \Big( \chi_i \boxtimes S_{m_i} \Big)^{\oplus l_i}
\]
where each $\chi_i$ is a character of $\mathbb{C}^{\times}$. If we decompose $K$ according to the parity as in \eqref{eq: index set} and take $\psi^{\mathbb{R}}_{good}$ to be an A-parameter of a smaller unitary group $G_{-}$ given by
\[
BC(\psi^{\mathbb{R}}_{good}) = \bigoplus_{i \in I^{+}} \Big( \chi_i \boxtimes S_{m_i} \Big)^{\oplus l_i},
\]
then $A_{\psi^{\mathbb{R}}} \cong A_{\psi^{\mathbb{R}}_{good}}$ (cf. \eqref{eq: real centralizer}) and there is a bijection
\[
\Pi_{\psi^{\mathbb{R}}_{good}}^{\rm pure}(G_{-}) \xrightarrow{\simeq} \Pi_{\psi^{\mathbb{R}}}^{\rm pure}(G), \quad \pi_{-} \mapsto \pi
\]
where $\pi_{-}$ is an irreducible representation of $U(p - |J|, q- |J|)$ and
\[
\pi := {\rm Ind}^{U(p,q)}_{P} \, \Big(\boxtimes_{i \in J} (\chi_i \circ {\rm det}_{GL_{m_i}})\Big) \boxtimes \pi_{-}
\]
for a parabolic subgroup $P$ whose Levi factor is isomorphic to $\prod_{i \in J} GL_{m_i}(\mathbb{C}) \times U(p - |J|, q- |J|)$. Moreover, the following diagram commutes:
\[
\xymatrix{
\Pi_{\psi^{\mathbb{R}}_{good}}^{\rm pure}(G) \ar[d]_{\epsilon^{\mathbb{R}}} \ar[r]^{\simeq} &  \Pi_{\psi^{\mathbb{R}}}^{\rm pure}(G) \ar[d]^{\epsilon^{\mathbb{R}}} \\
{\rm Rep}(A_{\psi^{\mathbb{R}}_{good}}) \ar[r]^{\simeq} & {\rm Rep}(A_{\psi^{\mathbb{R}}}).
}
\]
Hence, one can always reduce the study of $\Pi_{\psi^{\mathbb{R}}}^{\rm pure}(G)$ to the good parity case (cf. \cite[Section 4.2]{MoeglinRenard:2020}). Nevertheless, it is still an interesting problem to extend Theorem~\ref{thm: comparison of A-packets} and Theorem~\ref{thm: comparison of endoscopic maps} beyond the good parity case. Indeed, we also consider the bad parity case in this paper, which can be viewed as the other extreme. 

In fact, in the course of our proof we obtain a map $\tilde \iota$ between representations of $G$ and $H$ for \textit{arbitrary} real integral infinitesimal characters $\lambda^\BR$ that do not necessarily come from A-parameters, and it enjoys remarkable properties that may be of independent interest. To state them, let $\lambda^\BR = (\lambda_1,\ldots,\lambda_n)$ be integral so that $\Phi(G)_{\lambda^{\mathbb{R}}} \neq \emptyset$. Then $\lambda_i - \lambda_j \in \mathbb{Z}$. Moreover, $\lambda_i \in \mathbb{Z}$ if $n$ is odd, and $\lambda_i \in \frac{1}{2} \mathbb{Z}$ if $n$ is even (cf. Lemma~\ref{lemma: Upq param space}). We say $\lambda^{\mathbb{R}}$ is of good parity if for all $i$,
\[
\begin{cases}
\lambda_i \in \mathbb{Z}, & \text{ $n$ is odd,} \\
\lambda_i \in \frac{1}{2}\mathbb{Z} \backslash \mathbb{Z}, & \text{ $n$ is even.} 
\end{cases}
\]
Otherwise, we say it is of bad parity. We take $\delta \in \mathbb{Z}$ such that $\lambda_i > (1 - \delta)/2$ for all $i$, and let
\[
N = \sum_{i=1}^{n} (2\lambda_i + \delta).
\]
In the good parity case, the split $p$-adic group $H$ can be defined as before. In the bad parity case, both $n$ and $N$ are even, and we let 
\[
H = \begin{cases} SO(N+1) \quad & \text{$\delta$ is odd} \\
O(N) \quad & \text{$\delta$ is even}
\end{cases}
\]
split over $\mathbb{Q}_p$. In both good and bad parity cases, we take $\lambda^{\mathbb{Q}_p}$ to be the infinitesimal character of the Langlands parameter $\phi^{\mathbb{Q}_p}$ of $H$ such that
\[
{\rm std}_{\D{H}} \circ \phi^{\mathbb{Q}_p} = \bigoplus_{i = 1}^{n} |\cdot|_{W_{\mathbb{Q}_p}}^{\frac{\lambda_i - \lambda_{n+1 - i}}{2}} \boxtimes S_{\lambda_i + \lambda_{n+1 - i} + \delta}
\]
We may define
\[
\iota: \Phi(G)_{\lambda^{\mathbb{R}}} \rightarrow \Phi(H)_{\lambda^{\mathbb{Q}_p}}
\]  
and
\[
\tilde{\iota}: \Pi_{{\rm pure}}(\lambda^{\mathbb{R}}, G) \rightarrow \Pi_{{\rm pure}}(\lambda^{\mathbb{Q}_p}, H)
\]
as before. The following result follows from Corollary~\ref{cor: geom comparison} and Theorem~\ref{thm: ABV-packet}.

\begin{theorem}
\label{thm: comparison of ABV-packets}
\begin{enumerate}

\item If we extend $\tilde{\iota}$ to virtual representations by linearity, then $\tilde{\iota}$ preserves standard representations. 

\item For $\phi^{\mathbb{R}} \in \Phi(G)_{\lambda^{\mathbb{R}}}$, let $\phi^{\mathbb{Q}_p} = \iota(\phi^{\mathbb{R}})$. When $H$ is orthogonal, $\tilde{\iota}$ induces a bijection $\Pi_{\phi^{\mathbb{R}}}^{\rm ABV, pure}(G) \rightarrow \Pi_{\phi^{\mathbb{Q}_p}}^{\rm ABV, pure}(H)$. When $H$ is symplectic, $\tilde{\iota}$ induces a bijection
\[
\bigsqcup_{(p, q): p + q = n,\, p \equiv \frac{n-1}{2} \, {\rm mod} \, 2 } \Pi_{\phi^{\mathbb{R}}}^{\rm ABV}(U(p,q)) \xrightarrow{\simeq} \Pi_{\phi^{\mathbb{Q}_p}}^{\rm ABV, pure}(H).
\] 
\end{enumerate}
\end{theorem}

We can also have analogs of Theorem~\ref{thm: comparison of A-packets} and Theorem~\ref{thm: comparison of endoscopic maps} in the case where $\lambda^\BR$ has bad parity. In that situation $\Pi_{{\rm pure}}(\lambda^{\mathbb{R}}, G) = \Pi(G)_{\lambda^{\mathbb{R}}}$ (resp. $\Pi_{{\rm pure}}(\lambda^{\mathbb{Q}_p}, H) = \Pi(H)_{\lambda^{\mathbb{Q}_p}}$) which is in bijection with $\Phi(G)_{\lambda^{\mathbb{R}}}$ (resp. $\Phi(H)_{\lambda^{\mathbb{Q}_p}}$). For an Arthur parameter $\psi^{\mathbb{R}}$ with infinitesimal character $\lambda^{\mathbb{R}}$, its base change has the form
\[
BC(\psi^{\mathbb{R}}) = \bigoplus_{i = 1}^{r} \, (z/\bar{z})^{\frac{k_i}{2}} \boxtimes S_{m_i}
\]
where $k_i \in \mathbb{Z}, m_i \in \mathbb{Z}_{> 0}$ and $k_i + m_i$ is odd. We may attach a $p$-adic A-parameter $\psi^{\BQ_p}$ to $\psi^\BQ$ exactly as in the good parity case. Let $\phi_{\psi^{\mathbb{R}}}$ (resp. $\phi_{\psi^{\mathbb{Q}_p}}$) be the associated Langlands parameter of $\psi^{\mathbb{R}}$ (resp. $\psi^{\mathbb{Q}_p}$) (see \eqref{eq: associated L-parameter real} \eqref{eq: associated L-parameter p-adic}). Then $\phi_{\psi^{\mathbb{Q}_p}} = \iota (\phi_{\psi^{\mathbb{R}}})$ and $\Pi_{\psi^{\mathbb{R}}}(G)$ (resp. $\Pi_{\psi^{\mathbb{Q}_p}}(H)$) consists of a single representation parametrized by $\phi_{\psi^{\mathbb{R}}}$ (resp. $\phi_{\psi^{\mathbb{Q}_p}}$) (cf. \cite{MoeglinRenard:2020} \cite{Moeglin:2009}). Hence $\tilde{\iota}$ induces a bijection
\[
\Pi_{\psi^{\mathbb{R}}}(G) \xrightarrow{\simeq} \Pi_{\psi^{\mathbb{Q}_p}}(H).
\]
At last, $A_{\psi^{\mathbb{R}}} = A_{\psi^{\mathbb{Q}_p}} = 1$ and both $\epsilon^{\mathbb{R}}$ and $\epsilon^{\mathbb{Q}_p}$ map to the trivial representation. Hence the analogue of Theorem~\ref{thm: comparison of endoscopic maps} is trivial here.


\subsection{Strategy of proof}
\label{subsec: strategy}

Let us return to Theorem~\ref{thm: comparison of A-packets} and Theorem~\ref{thm: comparison of endoscopic maps}, where $\lambda^\BR$ is of good parity. When $\lambda^{\mathbb{R}}$ is regular, $\Pi_{\psi^{\mathbb{R}}}^{\rm pure}(G)$ is a pure Adams-Johnson packet, which have been computed in terms of complete Langlands parameters (see Appendix~\ref{sec: Langlands}). The corresponding $\psi^{\mathbb{Q}_p}$ satisfies the condition of discrete diagonal restriction, in which case Moeglin has provided an explicit formula for $\Pi_{\psi^{\mathbb{Q}_p}}^{\rm pure}(H)$, again in terms of complete Langlands parameters. Since $\tilde{\iota}$ is defined explicitly, Theorem~\ref{thm: comparison of A-packets} and Theorem~\ref{thm: comparison of endoscopic maps} can be verified directly (see Subsection~\ref{subsec: regular-case}). When $\lambda^{\mathbb{R}}$ is singular, $\Pi_{\psi^{\mathbb{R}}}^{\rm pure}(G)$ can be obtained from the regular case by applying translation functors, while $\Pi_{\psi^{\mathbb{Q}_p}}^{\rm pure}(H)$ can be obtained from the case of discrete diagonal restriction by taking Jacquet restrictions (see Subsection~\ref{subsec: singular-case}). Therefore, it suffices to establish comparison results between translations and Jacquet restrictions. We achieve this by comparing of the geometry of the Langlands parameter spaces for real and $p$-adic groups (see Section~\ref{sec: translation and derivative}).  

The Langlands parameter space (i.e., ABV space) for $G$ with infinitesimal character $\lambda^{\mathbb{R}}$ is a partial flag variety $X_{\lambda^{\mathbb{R}}}$ of $\D{G}$, equipped with a left action by $\D{K} := \D{G}^{\sigma}$ for some involution $\sigma$ on $\D{G}$. There is a bijection
\[
\Phi(G)_{\lambda^{\mathbb{R}}} \xrightarrow{\simeq} \D{K} \backslash X_{\lambda^{\mathbb{R}}}.
\]
The Langlands parameter space (i.e., Vogan variety) for $H$ with infinitesimal character $\lambda^{\mathbb{Q}_p}$ is an affine space $V_{\lambda^{\mathbb{Q}_p}}$ in the Lie algebra of $\D{H}$, equipped with the adjoint action of a reductive subgroup $H_{\lambda^{\mathbb{Q}_p}}$ in $\D{H}$. There is a bijection
\[
\Phi(H)_{\lambda^{\mathbb{Q}_p}} \xrightarrow{\simeq} H_{\lambda^{\mathbb{Q}_p}} \backslash V_{\lambda^{\mathbb{Q}_p}}.
\]
Under this bijection, the image of $\iota$ corresponds to an $H_{\lambda^{\mathbb{Q}_p}}$-stable open subset $V^{reg}_{\lambda^{\mathbb{Q}_p}}$ of $V_{\lambda^{\mathbb{Q}_p}}$. We will show in Section~\ref{sec: geometry} that $\iota$ can be upgraded to a correspondence between quotient stacks as follows
\[
[\D{K} \backslash X_{\lambda^{\mathbb{R}}}] \xleftarrow{\theta} [H_{\lambda^{\mathbb{Q}_p}}\backslash V^{reg}_{\lambda^{\mathbb{Q}_p}}] \xrightarrow{\iota} [H_{\lambda^{\mathbb{Q}_p}}\backslash V_{\lambda^{\mathbb{Q}_p}}],
\]
where we denote the natural inclusion of open substack by $\iota$ again and $\theta$ is an isomorphism if $H$ is orthogonal, and is a finite 2-to-1 morphism if $H$ is symplectic. Theorem~\ref{thm: comparison of ABV-packets} above essentially follows from this observation. 

Let $K\Pi_{{\rm pure}}(\lambda^{\mathbb{R}}, G)$ (resp. $K\Pi_{{\rm pure}}(\lambda^{\mathbb{Q}_p}, H)$) be the free $\mathbb{Z}$-module generated by $\Pi_{{\rm pure}}(\lambda^{\mathbb{R}}, G)$ (resp. $\Pi_{{\rm pure}}(\lambda^{\mathbb{Q}_p}, H)$). Let $K{\rm Per}(\D{K} \backslash X_{\lambda^{\mathbb{R}}})$ (resp. $K{\rm Per}(H_{\lambda^{\mathbb{Q}_p}} \backslash V_{\lambda^{\mathbb{Q}_p}})$) be the Grothendieck group of equivariant perverse sheaves. Through the local Langlands correspondence for pure inner forms $G$ and $H$ respectively (cf. Subsection~\ref{subsec: hypothesis LLC}), one can define perfect pairings
\begin{align}
\label{eq: pairing real}
\langle \cdot, \cdot \rangle: K\Pi_{{\rm pure}}(\lambda^{\mathbb{R}}, G) \times K{\rm Per}(\D{K} \backslash X_{\lambda^{\mathbb{R}}}) \rightarrow \mathbb{Z}
\end{align}
and
\begin{align}
\label{eq: pairing p-adic}
\langle \cdot, \cdot \rangle: K\Pi_{{\rm pure}}(\lambda^{\mathbb{Q}_p}, H) \times K{\rm Per}(H_{\lambda^{\mathbb{Q}_p}} \backslash V_{\lambda^{\mathbb{Q}_p}}) \rightarrow \mathbb{Z}.
\end{align}
Consider the map
\begin{align}
\label{eq: geometric restriction}
\mathcal{R}: K{\rm Per}(H_{\lambda^{\mathbb{Q}_p}} \backslash V_{\lambda^{\mathbb{Q}_p}}) \rightarrow K{\rm Per}(\D{K} \backslash X_{\lambda^{\mathbb{R}}}),
\end{align}
defined by $\mathcal{R} := R\theta_{*} \circ \iota^*$. It has an adjoint with respect to the pairings defined above
\begin{align}
\label{eq: adjoint}
\mathcal{R}^{*}: K\Pi_{{\rm pure}}(\lambda^{\mathbb{R}}, G) \rightarrow K\Pi_{{\rm pure}}(\lambda^{\mathbb{Q}_p}, H).
\end{align}
We will show that $\mathcal{R}^{*}$ is equal to $\tilde{\iota}$ on the irreducible representations up to some signs (cf. Proposition~\ref{prop: iota vs iotageom}). At last, we will compare the translation and Jacquet restriction by comparing their transposes under the pairings. The transpose of translation is given by certain pullback-pushforward functors (cf. Theorem~\ref{thm: translation vs pushpull}). The transpose of Jacquet restriction is given by Lusztig's induction of sheaves (cf. Conjecture~\ref{conj: adjoint of derivative}); this will be proven in \cite{CDX:adjunction}.

The idea of relating real and $p$-adic groups by the geometry of their Langlands parameter spaces is not new. Earlier works in this direction include Lusztig-Zelevinsky \cite{Zelevinskii:1985}, Ciubotaru-Trapa \cite{CiubotaruTrapa:2012}, and Barchini-Trapa \cite{BarchiniTrapa}. In particular, Barchini-Trapa constructed, for each reductive group $G$ of classical type and certain exceptional types, a map from a Vogan variety attached to $G(\BQ_p)$ to an ABV space attached to $G(\BR)$. Our construction goes in the other direction, where the ABV space of a real group is mapped (up to a possibly two-to-one map) into the Vogan variety of a $p$-adic group of different type and having a larger rank. We hope to explore the relationship between our map and theirs in a future work.

\subsection{Working hypotheses concerning the Local Langlands correspondence}
\label{subsec: hypothesis LLC}

Our results depend on the existence of the local Langlands correspondence for pure inner forms of $G$ and $H$, since both the definitions of $\tilde{\iota}$ and the pairings \eqref{eq: pairing real} \eqref{eq: pairing p-adic} do. For $G$, we use the results of Langlands \cite{Langlands:1989} and Shelstad \cite{Shelstad:1979}. For $H$, we use the results of Arthur \cite{Arthur:2013} and Moeglin-Renard \cite{MoeglinRenard:2018}, which still depends on the validity of the twisted weighted fundamental lemma after the recent work of Atobe-Gan-Ichino-Kaletha-Minguez-Shin \cite{AGIKMS:2024}. In particular, we need to assume that their results can also be extended to the full even orthogonal groups (cf. Conjecture~\ref{conj: LLC full even orthogonal}). In the course of the proof, we also use the Kazhdan-Lusztig conjecture (cf. Theorem~\ref{thm: LLC/R}, Conjecture~\ref{conj: LLC/Qp}, Conjecture~\ref{conj: disconnected LLC/Qp}), which says that the standard representations (resp. irreducible representations) and standard sheaves (resp. simple perverse sheaves) are dual to each other up to some signs under the pairings \eqref{eq: pairing real}, \eqref{eq: pairing p-adic}. For $G$, this is a consequence of Vogan duality \cite{Vogan:1982} (also see \cite[Theorem 1.24]{ABV:1992}). For $H$, we first make the following hypothesis.
\begin{hypothesis}
$\Pi_{{\rm pure}}(\lambda^{\mathbb{Q}_p}, H)$ consists of unipotent representations in the sense of Lusztig. 
\end{hypothesis}
For unipotent representations of inner forms of unramified $p$-adic reductive groups, Lusztig \cite{Lusztig1:1995} \cite{Lusztig:2002} (in case of adjoint groups) and Solleveld \cite{Solleveld:2023} (in general) also construct a local Langlands correspondence. Moreover, they have proved the Kazhdan-Lusztig conjecture for the pairing defined with respect to their constructions of the local Langlands correspondence \cite{Lusztig:1995B} \cite{Lusztig1:2002} \cite{Solleveld:2025}. To compare different choices of the local Langlands correspondence, we make another hypothesis:
\begin{hypothesis}
The local Langlands correspondence for unipotent representations established by Lusztig and Solleveld coincide with that of Arthur and Moeglin-Renard for $\Pi_{{\rm pure}}(\lambda^{\mathbb{Q}_p}, H)$.
\end{hypothesis}
In the case when $H$ is special odd orthogonal and $p$ is sufficiently large, both hypotheses have been proved by Waldspurger and Moeglin \cite{MW:2003} \cite{Waldspurger1:2007} \cite{Waldspurger:2016}. For the purpose of understanding the pure Arthur packets of $G$, we can always restrict to this case by taking $\delta \equiv n + 1$ mod $2$.

\subsection{Acknowledgement} 
This work was initiated through extensive computational experiments conducted using the \textsf{Atlas} program. The results of these experiments served as a significant source of inspiration and encouragement throughout this project. The authors wish to express their gratitude to the developers of \textsf{Atlas}. In addition, Q.Z. would like to thank Peter Trapa for helpful discussions. 

B.X. is supported by Ministry of Science and Technology of China, No. 2021YFA1000700 and National Natural Science Foundation of China, No.12571027. T.D. is supported by National Natural Science Foundation of
China, No. 12401013 and Beijing Natural Science Foundation, No. 1244042.

\section{Geometry}\label{sec: geometry}

\subsection{Langlands correspondence for $U(p,q)$}
\label{subsec: LLC real}

\begin{clause}[Langlands correspondence for real groups]\label{cls: LLC/R}
	Let $G$ be a quasi-split real group and $\D G$ be its (complex) Langlands dual group with a fixed pinning. In particular, we have fixed a Borel subgroup with a Cartan subgroup $\D B \supset \D T$. Let $\hat{\mathfrak{g}}$, $\hat{\mathfrak{t}}$ be the Lie algebras of $\D{G}, \D{T}$ respectively. Let $\tau \in \operatorname{Gal}(\BC/\BR)$ be the nontrivial element and form the L-group ${}^L G := \D G \rtimes \{1,\tau\}$. Write $W_\BR = \langle \BC^\times, j\rangle$ for the real Weil group. Let $\lambda^{\mathbb{R}} \in \hat{\mathfrak{t}}$ represent an integral infinitesimal character of $G$ and write $\Phi(G)_{\lambda^{\mathbb{R}}}$ for the set of $\D{G}$-conjugacy classess of Langlands parameters $\phi^\BR: W_\BR \to {}^L G$ with infinitesimal character $\lambda^{\mathbb{R}}$. Here we say $\phi^\BR$ has infinitesimal character $\lambda^{\mathbb{R}}$ if its restriction to $\BC^\times$ is given by
	\begin{equation*}
		\phi^\BR(z) = z^{\lambda'} \bar z^{\Ad \phi^\BR(j) \cdot \lambda'}
	\end{equation*}
    for some semisimple element $\lambda' \in \D \fg$ conjugate to $\lambda^{\mathbb{R}}$ under $\D{G}$.
	Then the Langlands correspondence for pure inner forms of $G$ is a bijection
	\[
		\Pi_{{\rm pure}}(\lambda^{\mathbb{R}}, G) \bijects 
		\Xi(\lambda^{\mathbb{R}},G) := 
		\{ (\phi^\BR, \epsilon) \mid 
			\phi^\BR \in \Phi(G)_{\lambda^{\mathbb{R}}}, \,
			\epsilon \in \D A_{\phi^\BR}
		\}
	\]
	(cf. \cite[Theorem 6]{Vogan:1993}). Here $\Pi_{{\rm pure}}(\lambda^{\mathbb{R}},G)$ is the set of isomorphism classes of irreducible admissible representations of pure inner forms of $G$ with infinitesimal character $\lambda^{\mathbb{R}}$, $A_{\phi^\BR}$ is the component group of the centralizer $Z_{\D G}(\phi^\BR)$ of $\phi^\BR$, and $\D A_{\phi^\BR}$ is the set of its irreducible representations. The set $\Xi(\lambda^{\mathbb{R}}, G)$ is called the set of \textit{complete Langlands parameters} of infinitesimal character $\lambda^{\mathbb{R}}$.
	
	Adams-Barbasch-Vogan \cite{ABV:1992} reformulated the right hand side as a set of irreducible perverse sheaves. In more details, let 
	\[
		\cI(\lambda^{\mathbb{R}}) = \{ y \in \D G \rtimes \tau \subset {}^L G \mid y^2 \in {\rm Ad}(\D G) (e^{2\pi i\lambda^{\mathbb{R}}}) \}.
	\]
    Since $\lambda^{\mathbb{R}}$ is assumed to be integral, the element $e^{2\pi i\lambda^{\mathbb{R}}}$ lies in the center $Z(\D G)$ of $\D G$. Therefore, $\cI(\lambda^{\mathbb{R}}) = \{ y \in \D G \rtimes \tau \subset {}^L G \mid y^2 = e^{2\pi i\lambda^{\mathbb{R}}} \}$.
	Let $\D P(\lambda^{\mathbb{R}}) \subset \D G$ be the parabolic subgroup stabilizing the canonical flat through $\lambda^{\mathbb{R}}$, and let $X_{\lambda^{\mathbb{R}}} = \D G / \D P(\lambda^{\mathbb{R}})$ be the corresponding partial flag variety. The product $\cI(\lambda^{\mathbb{R}}) \times X_{\lambda^{\mathbb{R}}}$ equipped with the diagonal $\D G$-action
	\[
		\cX_{\lambda^{\mathbb{R}}} := [\D G \backslash (\cI(\lambda^{\mathbb{R}}) \times X_{\lambda^{\mathbb{R}}})]
	\]
	is called the \textit{ABV space}. Irreducible perverse sheaves on the ABV space have a nice parameterization, namely
	\[
		\operatorname{Irr} \operatorname{Per}( \D G \backslash ( \cI(\lambda^{\mathbb{R}}) \times X_{\lambda^{\mathbb{R}}}) ) \cong
		\left\{(C, \cV) \mid 
		\begin{array}{c}
			C \text{ a $\D G$-orbit on } \cI(\lambda^{\mathbb{R}}) \times X_{\lambda^{\mathbb{R}}},\\
			\cV \text{ a simple $\D G$-equivariant local system on } C.
		\end{array}
		\right\},
	\]
	where $(C,\cV)$ corresponds to the perverse sheaf $i_{C!*} \cV[\dim C]$. A pair $(C,\cV)$ on the right side is called \textit{complete geometric parameters} in \textit{loc. cit.} 
	
	Note that $\cI(\lambda^{\mathbb{R}})$ is a disjoint union of finitely many $\D G$-orbits. Hence, for any choice of representatives $y_i \in \cI(\lambda^{\mathbb{R}})$ of the $\D G$-orbits, we have a decomposition of quotient stacks
	\begin{equation*}
		[\D G \backslash (\cI(\lambda^{\mathbb{R}}) \times X_{\lambda^{\mathbb{R}}})] = \bigsqcup_i \, [\D K_i \backslash X_{\lambda^{\mathbb{R}}}]
	\end{equation*}
	where $\D K_i = Z_{\D G}(y_i)$. In particular, 
	\begin{equation*}
		\operatorname{Per}(\D G \backslash (\cI(\lambda^{\mathbb{R}}) \times X_{\lambda^{\mathbb{R}}})) \cong \bigoplus_i \operatorname{Per}( \D K_i \backslash X_{\lambda^{\mathbb{R}}}).
	\end{equation*}
	
	\begin{proposition}[{\cite[Chapter 5, 6]{ABV:1992}, \cite[Prop 7.4]{AdC:2009}}]\label{prop: orbit vs L-param}
		There is a bijection 
		\[
			\Phi(G)_{\lambda^{\mathbb{R}}} \bijects \D G \backslash ( \cI(\lambda^{\mathbb{R}}) \times X_{\lambda^{\mathbb{R}}} )
		\]
		defined as follows. Given $\phi^\BR \in \Phi(G)_{\lambda^{\mathbb{R}}}$, there is a unique way of writing the restriction of $\phi^\BR$ to $\BC^\times$ as
		\[
		\phi^\BR(z) = z^{\lambda'} \bar z^{\Ad \phi^\BR(j) \lambda'}
		\]
		for some semisimple element $\lambda' \in \D \fg$. Let $y = e^{\pi i \lambda'} \phi^\BR(j)$ and let $\D P(\lambda')$ be the stabilizer of the canonical flat through $\lambda'$ in $\D G$. Then 
        the above bijection sends 
		\[
			\phi^\BR \mapsto (y, \D P(\lambda')).
		\]
		There is a canonical isomorphism between $Z_{\D G}(\phi^\BR)$ and the reductive quotient of $Z_{\D G}(y, \D P(\lambda'))$, and hence the above bijection upgrades to
		\[
			\Xi(\lambda^{\mathbb{R}},G) \bijects \operatorname{Irr} \operatorname{Per}( \D G \backslash (\cI(\lambda^{\mathbb{R}}) \times X_{\lambda^{\mathbb{R}}})).
		\]
	\end{proposition}
	
	From now on, we will write $\Xi(\lambda^{\mathbb{R}}, G)$ both for the set of complete Langlands parameters $(\phi^\BR, \epsilon)$ and for the set of complete geometric parameters $(C, \cV)$.
	
	As a result, we obtain a bijection
	\begin{equation*}
		\Pi_{{\rm pure}}(\lambda^{\mathbb{R}}, G) \bijects \Xi(\lambda^{\mathbb{R}}, G) \bijects \operatorname{Irr} \operatorname{Per}( \D G \backslash (\cI(\lambda^{\mathbb{R}}) \times X_{\lambda^{\mathbb{R}}})). 
	\end{equation*}
	Adams-Barbasch-Vogan upgraded this bijection to a perfect pairing in a way so that character formulas can be computed using local multiplicities of perverse sheaves. For a parameter $\xi = (C,\cV) \in \Xi(\lambda^{\mathbb{R}},G)$, we write $\pi(\xi) \in \Pi_{{\rm pure}}(\lambda^{\mathbb{R}}, G)$ for the corresponding irreducible representation of some pure inner form $G_{\xi}$ of $G$, and write $M(\xi)$ for the standard representation having $\pi(\xi)$ as its Langlands quotient. Further, we write $\mathcal{P}(\xi) = j_{C!*} \cV[\dim C]$, $\mu(\xi) = j_{C!} \cV[\dim C]$\footnote{Our definition of $\mu(\xi)$ differs with the one in \cite{ABV:1992} by a cohomological shift by degree $\dim C$.} for the simple and standard objects in the equivariant derived category $D^b( \D G \backslash (\cI(\lambda^{\mathbb{R}}) \times X_{\lambda^{\mathbb{R}}}))$ (these objects are in fact in the perverse heart $\operatorname{Per}(\D G \backslash (\cI(\lambda^{\mathbb{R}}) \times X_{\lambda^{\mathbb{R}}}))$). Let $\Rep_{{\rm pure}}(\lambda^{\mathbb{R}}, G)$ denote the category of admissible representations of pure inner forms of $G$ with infinitesimal character $\lambda^{\mathbb{R}}$. The simple objects in $\Rep_{{\rm pure}}(\lambda^{\mathbb{R}}, G)$ are precisely $\Pi_{{\rm pure}}(\lambda^{\mathbb{R}}, G)$. Define a perfect pairing 
	\begin{equation}\label{eq: real pairing}
		\langle -,- \rangle : K \Rep_{{\rm pure}}(\lambda^{\mathbb{R}}, G) \times K D^b(\D G \backslash (\cI(\lambda^{\mathbb{R}}) \times X_{\lambda^{\mathbb{R}}})) \aro \BC
	\end{equation}
	by 
    \begin{equation*}
			\langle \pi(\xi), \mathcal{P}(\xi') \rangle = (-1)^{\dim C} e(G_{\xi}) \delta_{\xi, \xi'}, 
    \end{equation*}
	where $C$ is the orbit appearing in $\xi = (C, \cV)$, $e(G_{\xi})$ is the Kottwitz sign of $G_{\xi}$ (\cite{Kottwitz:1983}), and $\delta_{\xi, \xi'}$ is the Kronecker delta. 
	
	\begin{theorem}[{\cite[Thm 1.24]{ABV:1992}}]\label{thm: LLC/R}
		Under the above pairing, we have
		\begin{equation*}
		\langle M(\xi), \mu(\xi') \rangle = (-1)^{\dim C} e(G_{\xi}) \delta_{\xi, \xi'}.
	\end{equation*}
	\end{theorem}

For any $\mathcal{P} \in {\rm Per}(\D G \backslash (\cI(\lambda^{\mathbb{R}}) \times X_{\lambda^{\mathbb{R}}}))$, its characteristic cycle  is a formal linear combination
\[
CC(\mathcal{P}) = \sum_{C} m_{C}(\mathcal{P}) [\overline{T^{*}_{C} 
(\cI(\lambda^{\mathbb{R}}) \times X_{\lambda^{\mathbb{R}}})}], \quad m_C(\mathcal{P}) \in \mathbb{Z}_{\geq 0}
\]
where the sum is over $\D{G}$-orbits $C$ in $\cI(\lambda^{\mathbb{R}}) \times X_{\lambda^{\mathbb{R}}}$ \cite[Section 9.4]{KS:1994} or \cite[Definition 2.2.2]{HTT}. For any $\phi \in \Phi(G)_{\lambda^{\mathbb{R}}}$, let $C_{\phi}$ be the associated orbit under Proposition \ref{prop: orbit vs L-param}. Then the ABV-packet attached to $\phi$ is defined to be
\[
\Pi^{\rm ABV, pure}_{\phi}(G) := \{\pi(\xi) \, | \, \xi \in \Xi(\lambda^{\mathbb{R}}, G), m_{C_{\phi}}(\mathcal{P}(\xi)) \neq 0\}.
\]

For any A-parameter $\psi^{\mathbb{R}}$ of $G$, one can associate an L-parameter 
by
    \begin{align}
    \label{eq: associated L-parameter real}
    \phi_{\psi^{\mathbb{R}}} (w) := \psi^{\mathbb{R}}\Big( w,
		{\small \begin{pmatrix}
				|w|^{1/2}\\&|w|^{-1/2}
			\end{pmatrix}} \Big), \quad w \in W_{\mathbb{R}}
    \end{align}
Define the infinitesimal character of $\psi^{\mathbb{R}}$ to be that of $\phi_{\psi^{\mathbb{R}}}$. Denote by $\Psi(G)_{\lambda^{\mathbb{R}}}$ the set of $\D{G}$-conjugacy classes of A-parameters with infinitesimal character $\lambda^{\mathbb{R}}$. Then the rule $\psi^\BR \mapsto \phi_{\psi^\BR}$ gives an inclusion $\Psi(G)_{\lambda^{\mathbb{R}}} \hookrightarrow \Phi(G)_{\lambda^{\mathbb{R}}}$. For $\psi \in \Psi(G)_{\lambda^{\mathbb{R}}}$, we set
\[
\Pi^{\rm ABV, pure}_{\psi^{\mathbb{R}}}(G):= \Pi^{\rm ABV, pure}_{\phi_{\psi^{\mathbb{R}}}}(G).
\]

We now specialize to the unitary groups $U(p,q)$.

\end{clause}

\begin{clause}[The case of $U(p,q)$]
	Let $n$ be a positive integer, and $G$ be $U(\frac{n}{2}, \frac{n}{2})$ if $n$ is even or $U(\frac{n-1}{2}, \frac{n+1}{2})$ if $n$ is odd. Let
	\begin{equation*}
		J = {\small \begin{pmatrix}
				&& -1\\
				& \iddots\\
				(-1)^n
		\end{pmatrix}}.
	\end{equation*}
	Then $\D G = \GL_n(\BC)$ and $\tau \acts \D G$ by $\tau(g) = J {}^t g\inv J\inv$. Let $\D{B}$ (resp. $\D{T}$) be the subgroup of upper-triangular (resp. diagonal) matrices. Let $W^{\D{G}} := W(\D{G}, \D{T})$ be the Weyl group of $\D{G}$.
    Let $\lambda^{\mathbb{R}} = (\lambda_1,\ldots,\lambda_n) \in \BC^n = \hat{\mathfrak{t}}$ be a dominant integral infinitesimal character of $G$, i.e., $\lambda_i - \lambda_{i+1} \in \mathbb{Z}_{\geqslant 0}$ for all $1 \leqslant i \leqslant n -1$. Then $\D{P}(\lambda^{\mathbb{R}})$ is the standard parabolic subgroup of $\D{G}$ defined by the set of simple roots $\{\alpha_i \, | \, 1 \leqslant i \leqslant n-1, \lambda_i = \lambda_{i+1} \}$.
    
	\begin{lemma}\label{lemma: Upq param space}
		\begin{enumerate}
			\item The isomorphism classes of pure inner forms of $U(p,q)$ are represented by $U(p',q')$, where $p' + q' = n$ and $1 \le p' \le n$ (here and later, $U(p', q')$ and $U(q', p')$ denote two different pure inner forms whenever $p' \neq q'$).
			
			\item If $n$ is odd, then $\lambda^{\mathbb{R}} \in \BZ^n$, $\cI(\lambda^{\mathbb{R}}) = \Ad(\D G) \cdot J\inv \tau$, and the ABV space is isomorphic to $[O(n,\BC) \backslash \GL_n(\BC) / \D P(\lambda^{\mathbb{R}})]$. 
			
			\item If $n$ is even, then $\lambda^{\mathbb{R}} \in \BZ^n$ or $(\frac12 \BZ \backslash \BZ)^n$:
			\begin{itemize}
				\item \sloppy if $\lambda^{\mathbb{R}} \in \BZ^n$, then $\cI(\lambda^{\mathbb{R}}) = \Ad(\D G) \cdot \tau$, and the ABV space is isomorphic to $[\Sp(n,\BC) \backslash \GL_n(\BC) / \D P(\lambda^{\mathbb{R}})]$;
				\item if $\lambda^{\mathbb{R}} \in (\frac12 \BZ \backslash \BZ)^n$, then $\cI(\lambda^{\mathbb{R}}) = \Ad(\D G) \cdot J\inv \tau$, and the ABV space is isomorphic to  $[O(n,\BC) \backslash \GL_n(\BC) / \D P(\lambda^{\mathbb{R}})]$.
			\end{itemize}
		\end{enumerate}
	\end{lemma}
	
	If $\D K = O(n,\BC)$ (resp. $\D K = \Sp(n,\BC)$), we say $\lambda^{\mathbb{R}}$ has \textbf{good parity} (resp. \textbf{bad parity}). To summarize, we have the following cases:
	\begin{center}
		\begin{tabular}{cccc}
			parity of $n$ & $\lambda^{\mathbb{R}}$ & ABV space $\cX_{\lambda^{\mathbb{R}}}$
			\\ \hline
			odd & $\BZ^n$ & $[O(n,\BC) \backslash \GL_n(\BC) / \D P(\lambda^{\mathbb{R}})]$ & good parity
			\\ 
			even & $(\frac12 \BZ \backslash \BZ)^n$ & $[O(n,\BC) \backslash \GL_n(\BC) / \D P(\lambda^{\mathbb{R}})]$ & good parity
			\\
			even & $\BZ^n$ & $[\Sp(n,\BC) \backslash \GL_n(\BC) / \D P(\lambda^{\mathbb{R}})]$ & bad parity
		\end{tabular}
	\end{center}
	
	\begin{proof}
		For a real group $G$, consider the corresponding action $\operatorname{Gal}(\BC/\BR) = \{1,\tau\} \acts G_\BC$ and form the semidirect product $G_\BC \rtimes \{1,\tau\}$. A pure inner form of $G$ is the same as an order two element $g \tau$ in the non-identity component $G_\BC \rtimes \tau \subset G_\BC \rtimes \{1,\tau\}$ \cite[Def 2.6]{Vogan:1993}, and it induces an involution on $G_{\C}$ by conjugate action, whose fixed points is a real form of $G_{\C}$. 
        The isomorphism class of an pure inner form is defined to be its $G_{\C}$-conjugacy class (note that it is not determined by the associated real form). In the case $G = U(p,q)$, the action $\tau \acts G_\BC$ is given by $\tau(h) = I_{p,q} {}^t \bar h\inv I_{p,q}$, where $I_{p,q} = \diag(I_p, - I_q)$. 
        Suppose $g\tau$ has order $2$, that is to say $g\in G_\BC$ satisfies $g\tau(g) = 1 \in G_\BC$.
        Since $g\tau(g) = g I_{p, q} {}^t \bar g^{-1} I_{p, q} = (gI_{p, q}) \cdot {}^t\overline{(g I_{p, q})}^{-1}$, we see
        $g I_{p, q}$ is Hermitian.
        According to the classification of Hermitian forms, there is $h \in G_\BC$ such that $h (g I_{p, q}) {}^t\bar h = I_{p', q'}$, where $p', q'$ are the positive and the negative inertia index of $g I_{p, q}$, respectively. We may then replace $g\tau$ by its conjugate
        \[
        h g \tau h^{-1} = h g \tau(h)^{-1} \tau = h g I_{p, q} {}^t\bar h I_{p, q} \tau = I_{p', q'} I_{p, q} \tau.
        \]
        This shows that the isomorphism classes of pure inner forms of $G$ can be represented by $g\tau = I_{p', q'} I_{p, q} \tau$ for $1 \leqslant p' \leqslant n$, and the their associated real forms are exactly $U(p', q')$. 
        
		
		Next we turn to part (2) and (3). Recall that $\tau \acts \D G$ by $\tau(g) = J \, {}^t g\inv J\inv$. Let $y = g \tau \in \cI(\lambda^{\mathbb{R}})$. Since $\lambda^{\mathbb{R}}$ is integral,  
		\begin{equation} \label{eq: strong inv Upq}
			e^{2\pi i \lambda_1} I_n = e^{2\pi i \lambda^{\mathbb{R}}} = y^2 = (g \tau)^2 = g \tau(g) = g J {}^t g\inv J\inv.
		\end{equation}
	    This equation implies 
        \[
         e^{-2\pi i \lambda_1} I_n =\tau(e^{2\pi i \lambda_1} I_n)=\tau(g)g=g\tau(g)=e^{2\pi i \lambda_1} I_n
        \]  
        which forces $e^{2\pi i \lambda_1} = \pm 1$. Therefore either $\lambda_1 \in \BZ$ or $\frac12 \BZ \backslash \BZ$. Thus $\lambda^{\mathbb{R}} \in \BZ^n$ or $(\frac12 \BZ \backslash \BZ)^n$. 
		
		Suppose $\lambda^{\mathbb{R}} \in \BZ^n$. Equation (\ref{eq: strong inv Upq}) becomes $I_n = gJ {}^t g\inv J\inv$. If $n$ is odd, then $J\inv = {}^t J\inv$, and the equation reads $I_n = (gJ) \,  {}^t(gJ)\inv$, hence $gJ$ is symmetric. The $\D G$-conjugation action on $g \tau$ is given by
		\begin{align}
        \label{eq: conjugation}
			\Ad(h) \cdot g \tau = h g \tau h\inv = h g \tau(h^{-1}) \tau = (h g J \, {}^t h) J\inv \tau.
		\end{align}
		By choosing $h$ such that $hgJ\,{}^th = I_n$, we see that $\cI(\lambda^{\mathbb{R}})$ is equal to the $\D G$-orbit of $J\inv \tau$. The adjoint action $\Ad(J\inv \tau)$ is transpose inverse, hence we have $\D K = Z_{\D G}(J^{-1}\tau) = O(n,\BC)$. If $n$ is even, then $J\inv = - {}^t J\inv$, and equation (\ref{eq: strong inv Upq}) becomes $-I_n = (gJ) \, {}^t(gJ)\inv$, hence $gJ$ is skew symmetric. This time we can choose $h$ in \eqref{eq: conjugation} such that $hgJ\,{}^th = J$, and we see that $\cI(\lambda^{\mathbb{R}})$ is equal to the $\D G$-orbit of $J J\inv \tau = \tau$, and $\D K = Z_{\D G}( \tau) = \Sp(n,\BC)$.
		
		Suppose now $\lambda^{\mathbb{R}} \in (\frac12 \BZ \backslash \BZ)^n$. Then equation (\ref{eq: strong inv Upq}) becomes $I_n = -gJ {}^t g\inv J\inv$. If $n$ is odd, this forces $gJ$ to be skew symmetric. This is impossible because there is no skew symmetric invertible matrix of odd size. If $n$ is even, then $gJ$ is symmetric, and $\cI(\lambda^{\mathbb{R}})$ is equal to the $\D G$-orbit of $J\inv \tau$ with $\D K = O(n,\BC)$.
	\end{proof}
	
	In preparation for the comparison between real and $p$-adic parameter spaces, we would like to have an alternative description of the ABV spaces. Let $\operatorname{Sym}_n(\BC)_{reg}$ (resp. $\operatorname{Skew}_n(\BC)_{reg}$) denotes the set of invertible $n$ by $n$ symmetric (resp. skew symmetric) matrices equipped with the $\D P(\lambda^{\mathbb{R}})$-action by congruence.
	
	\begin{lemma}\label{lemma: sym/skew as ABV spaces}
		There is an isomorphism of stacks
		\begin{equation*}
			[O(n,\BC) \backslash \GL_n(\BC) / \D P(\lambda^{\mathbb{R}})] \bijects [\operatorname{Sym}_n(\BC)_{reg} / \D P(\lambda^{\mathbb{R}})], \quad
			g \mapsto {}^t g g.
		\end{equation*}
		If $n$ is even, there is an isomorphism
		\begin{equation*}
			[\Sp(n,\BC) \backslash \GL_n(\BC) / \D P(\lambda^{\mathbb{R}})] \bijects [\operatorname{Skew}_n(\BC)_{reg} / \D P(\lambda^{\mathbb{R}})],\quad
			g \mapsto {}^t g J\inv g.
		\end{equation*}
	\end{lemma}
	
	\begin{proof}
		Since every symmetric invertible matrix is congruent to the identity matrix, the congruence action $\GL_n(\BC) \acts \operatorname{Sym}_n(\BC)_{reg}$ is transitive, and the stabilizer of the identity matrix is the orthogonal group $O(n,\BC)$. Hence we have an isomorphism of varieties $O(n,\BC) \backslash \GL_n(\BC) \bij \operatorname{Sym}_n(\BC)_{reg}$, $g \mapsto {}^t g g$ where the right multiplication action of $\D P(\lambda^{\mathbb{R}})$ on the left hand side becomes the congruence action on the right hand side. The first claim follows. The second claim is similar.
	\end{proof}
	
	\begin{lemma}\label{lemma: Sym vs L-param}
		Let $\phi^\BR \in \Phi(G)_{\lambda^{\mathbb{R}}}$, and by $\D{G}$-conjugation we assume 
		\begin{equation*}
			\phi^\BR(z) = z^{\lambda^{\mathbb{R}}} \bar z^\mu, \quad z \in \mathbb{C}^{\times}
		\end{equation*}
        where $\mu \in \hat{\mathfrak{t}}$ and $\phi^{\R}(j)$ normalizes $\D{T}$.
		\begin{enumerate}
			\item Suppose $\lambda^{\mathbb{R}}$ has good parity. The chain of bijections
			\begin{equation*}
				\Phi(G)_{\lambda^{\mathbb{R}}}
				\bijects O(n,\BC) \backslash \GL_n(\BC) / \D P(\lambda^{\mathbb{R}})
				\bijects \operatorname{Sym}_n(\BC)_{reg} / \D P(\lambda^{\mathbb{R}})
			\end{equation*}
			sends
			\begin{equation*}
				\phi^\BR \mapsto h \mapsto \dot s
			\end{equation*}
			where $s \in W^{\D G}$ is  any involution so that $s \lambda^{\mathbb{R}} = -\mu$, $\dot s \in \operatorname{Sym}_n(\BC)_{reg}$ is the corresponding permutation matrix, and $h$ is any matrix so that ${}^t h h = \dot s$. 
			\item Suppose $\lambda^{\mathbb{R}}$ has bad parity. The chain of bijections
			\begin{equation*}
				\Phi(G)_{\lambda^{\mathbb{R}}}
				\bijects \Sp(n,\BC) \backslash \GL_n(\BC) / \D P(\lambda^{\mathbb{R}})
				\bijects \operatorname{Skew}_n(\BC)_{reg} / \D P(\lambda^{\mathbb{R}})
			\end{equation*}
			sends
			\begin{equation*}
				\phi^\BR \mapsto h \mapsto \dot s
			\end{equation*}
			where $s \in W^{\D G}$ is any  fixed point free involution so that $s\lambda^{\mathbb{R}} = -\mu$, $\dot s$ is any signed skew-symmetric permutation matrix representing $s$, and $h$ is any matrix so that ${}^t h J\inv h = \dot s$.
		\end{enumerate}		
	\end{lemma}
	
	\begin{proof}
		According to Proposition \ref{prop: orbit vs L-param}, the $\D G$-orbit of $\phi^\BR$ corresponds to the $\D G$-orbit of the pair $(y, \D P(\lambda^{\mathbb{R}})) \in \cX_{\lambda^{\mathbb{R}}}$, where $y = e^{\pi i \lambda^{\mathbb{R}}} \phi^\BR(j)$. Write $y = g J\inv \tau$ for some element $g \in \D G$. Then $g$ normalizes $\D T$ and reprsents some element $s \in W^{\D G}$. Again by Proposition \ref{prop: orbit vs L-param}, we have $\Ad(\phi^\BR(j)) \lambda^{\mathbb{R}} = \mu$. Since $e^{\pi i\lambda^{\mathbb{R}}}$ centralizes $\D \ft$, we have $\Ad(y) \lambda^{\mathbb{R}} = \mu$. Hence
		\begin{equation*}
			\mu = \Ad(g) \Ad(J\inv \tau)\lambda^{\mathbb{R}} = - \Ad(g) \lambda^{\mathbb{R}}
		\end{equation*}
		or equivalently $\Ad(g\inv) \mu = - \lambda^{\mathbb{R}}$. So $s^{-1} \mu = \Ad(g\inv) \mu = -\lambda^{\mathbb{R}}$. 

        Suppose $\lambda^{\mathbb{R}}$ has good parity. By Lemma \ref{lemma: Upq param space}, there is an element $h \in \GL_n(\BC)$ so that $\Ad(h) \cdot y = J\inv \tau$. Hence the orbit of $(y, \D P(\lambda^{\mathbb{R}}))$ is the same as the orbit of $(J\inv \tau, h \D P(\lambda^{\mathbb{R}}))$, and the image of $\phi^\BR$ in $O(n,\BC) \backslash \GL_n(\BC) / \D P(\lambda^{\mathbb{R}})$ is equal to the coset of $h$. Finally, the image of the coset of $h$ in $\operatorname{Sym}_n(\BC)_{reg} / \D P(\lambda^{\mathbb{R}})$ is ${}^thh$. Since $\Ad(h) \cdot y = (h g \,{}^th) J\inv \tau$, the condition $\Ad(h) \cdot y = J\inv \tau$ implies $h g \, {}^th = I_n$, and hence ${}^th h = g\inv$. Therefore the chain of bijections sends $\phi^\BR$ to the symmetric matrix $g\inv$. It follows that $s^2 = 1$ and $g^{-1}$, $\dot{s}$ are $\D{T}$-congruent, where $\dot{s}$ is the symmetric permutation matrix representing $s$. It remains to show that for involutions $s, s' \in W^{\D G}$ such that $s \lambda^{\R} = s' \lambda^{\R} = - \mu$, the elements $\dot{s}, \dot{s}'$ are $\D{P}(\lambda^{\R})$-congruent. By the argument in Lemma~\ref{lemma: Weyl group parametrization-surjective}, we can further assume $s, s' \in \mathfrak{I}_{\lambda^{\mathbb{R}}}$ (cf. \eqref{eqn: I-lambda-good}). Since $s^{-1}s' \in W^{M(\lambda^{\R})}$ (the Weyl group of the Levi $M(\lambda^\BR)$ of $\D P(\lambda^\BR)$), it follows from Lemma~\ref{lemma:W-orbit-by-Auv} that $s, s'$ are $W^{M(\lambda^{\R})}$-conjugate, and hence $\dot s$ and $\dot s'$ are conjugate by some premutation matrix $\dot w$ lifting an element $w \in W^{M(\lambda^\BR)}$. Note that for permutation matrices $\dot w$ we have $\dot w\inv = {}^t \dot w$. So $\dot s$ and $\dot s'$ are congruent under $\dot w$. In particular, they are $\D{P}(\lambda^{\R})$-congruent.		

        Suppose $\lambda^{\mathbb{R}}$ has bad parity. By Lemma \ref{lemma: Upq param space}, there is an element $h \in \GL_n(\BC)$ so that $\Ad(h) \cdot y = \tau$. Hence the orbit of $(y, \D P(\lambda^{\mathbb{R}}))$ is the same as the orbit of $(\tau, h \D P(\lambda^{\mathbb{R}}))$. Hence the image of $\phi^\BR$ in $\Sp(n,\BC) \backslash \GL_n(\BC) / \D P(\lambda^{\mathbb{R}})$ is equal to the coset of $h$, and the image of the coset of $h$ in $\operatorname{Skem}_n(\BC)_{reg} / \D P(\lambda^{\mathbb{R}})$ is ${}^thJ^{-1}h$. Since $\Ad(h) \cdot y = (h g \,{}^th) J\inv \tau$, then the condition $\Ad(h) \cdot y = \tau$ means $h g \, {}^th = J$, and hence ${}^th J^{-1} h = g\inv$. Therefore the chain of bijections sends $\phi^\BR$ to the skew symmetric matrix $g\inv$. It follows that $s^2 = 1$ and $g^{-1}$, $\dot{s}$ are $\D{T}$-congruent, where $\dot{s}$ is a signed skew-symmetric permutation matrix representing $s$. It remains to show that for fixed point free involutions $s, s' \in W^{\D G}$ such that $s \lambda^{\R} = s' \lambda^{\R} = - \mu$, we have that $\dot{s}, \dot{s}'$ are $\D{P}(\lambda^{\R})$-congruent. Since $s^{-1}s' \in W^{M(\lambda^{\R})}$, Lemma~\ref{lemma:W-orbit-by-Auv-skew} says $s, s'$ are $W^{M(\lambda^{\R})}$-conjugate. Hence, $\dot{s}, \dot{s}'$ are $\D{P}(\lambda^{\R})$-congruent.

	\end{proof}
\end{clause}

\subsection{Langlands correspondence for $p$-adic symplectic/orthogonal groups}
\label{subsec: LLC p-adic}

\begin{clause}[Langlands correspondence for $p$-adic groups, \cite{Vogan:1993,CFMMX:2022}]
	Let $W_{\BQ_p} = I_{\BQ_p} \rtimes \langle \Fr \rangle$ be the Weil group of $\BQ_p$, where $I_{\BQ_p}$ denotes the Inertia subgroup and $\Fr$ is a fixed lift of the arithmetic Frobenius $x \mapsto x^p$. Let $H$ be a quasisplit connected reductive algebraic group over $\BQ_p$, let $\D H$ be its complex Langlands dual group, and let ${}^L H$ be the L-group determined by $H$. Let $\lambda^{\mathbb{Q}_p}: W_F \to {}^L G$ be an infinitesimal character in the sense of \cite[\textsection 4.1]{CFMMX:2022}, and assume it is \textbf{unramified}, i.e. trivial on the inertia $I_{\BQ_p}$. Write $\Phi(H)_{\lambda^{\mathbb{Q}_p}}$ for the set of $\D{H}$-conjugacy classes of  Langlands parameters $\phi^{\mathbb{Q}_p}: W_{\BQ_p} \times \SL_2(\BC) \to {}^L H$ with infinitesimal character $\lambda^{\mathbb{Q}_p}$ (up to $\D{H}$-conjugacy). Here the infinitesimal character $\lambda^{\mathbb{Q}_p}$ of $\phi^{\mathbb{Q}_p}$ is the map
	\begin{equation*}
		\lambda^{\mathbb{Q}_p}: W_{\BQ_p} \aro {}^L H, \quad w \mapsto \phi^{\mathbb{Q}_p}\Big( w,
		{\small \begin{pmatrix}
				|w|^{1/2}\\&|w|^{-1/2}
			\end{pmatrix}} \Big).
	\end{equation*}

	The conjectural local Langlands correspondence for pure inner forms of $H$ is the existence of a natural  bijection
	\begin{equation*}
		\Pi_{{\rm pure}}(\lambda^{\mathbb{Q}_p},H) \bijects \Xi(\lambda^{\mathbb{Q}_p}, H) := \{ (\phi^{\mathbb{Q}_p}, \epsilon) \mid \phi^{\mathbb{Q}_p} \in \Phi(H)_{\lambda^{\mathbb{Q}_p}}, \epsilon \in \D A_{\phi^{\mathbb{Q}_p}} \}
	\end{equation*}
	where $\Pi_{{\rm pure}}(\lambda^{\mathbb{Q}_p},H)$ denotes the set of isomorphism classes of irreducible smooth representations of pure inner forms of $H$, $A_{\phi^{\mathbb{Q}_p}}$ denotes the component group of $Z_{\D H}(\phi^{\mathbb{Q}_p})$, and $\D A_{\phi^{\mathbb{Q}_p}}$ is the set of its irreducible representations.
	
	Similar to the real case, the set $\Xi(\lambda^{\mathbb{Q}_p},H)$ can also be parametrized by certain set of simple perverse sheaves on the \textit{Vogan variety} of $\lambda^{\mathbb{Q}_p}$, which  is defined to be 
	\begin{equation*}
		V_{\lambda^{\mathbb{Q}_p}} = \{ x \in \D \fh \mid \Ad( \lambda^{\mathbb{Q}_p}(\Fr)) x = p x \},
	\end{equation*}
    where $\D \fh$ is the Lie algebra of $\D{H}$, and it is equipped with the conjugate action of $H_{\lambda^{\mathbb{Q}_p}} = Z_{\D H}(\lambda^{\mathbb{Q}_p})$ with finitely many orbits. We may then consider the category of equivariant perverse sheaves $\operatorname{Per}(H_{\lambda^{\mathbb{Q}_p}} \backslash V_{\lambda^{\mathbb{Q}_p}})$ and the equivariant derived category $D^b(H_{\lambda^{\mathbb{Q}_p}} \backslash V_{\lambda^{\mathbb{Q}_p}})$. We have
	\begin{equation*}
		\operatorname{Irr} \operatorname{Per} (H_{\lambda^{\mathbb{Q}_p}} \backslash V_{\lambda^{\mathbb{Q}_p}}) \cong 
		\left\{(C, \cV) \mid 
		\begin{array}{c}
			C \text{ an $H_{\lambda^{\mathbb{Q}_p}}$-orbit on } V_{\lambda^{\mathbb{Q}_p}},\\
			\cV \text{ a simple $H_{\lambda^{\mathbb{Q}_p}}$-equivariant local system on } C.
		\end{array}
		\right\}.	
	\end{equation*}
	
	\begin{proposition}[{\cite[Prop 4.2.2]{CFMMX:2022}}]\label{prop: param vs Vogan}
		There is a bijection
		\begin{equation}\label{eq: param vs Vogan}
			\Phi(H)_{\lambda^{\mathbb{Q}_p}} \bijects H_{\lambda^{\mathbb{Q}_p}} \backslash V_{\lambda^{\mathbb{Q}_p}},\quad
			\phi^{\mathbb{Q}_p} \mapsto x := d \big( \phi^{\mathbb{Q}_p}|_{\SL_2(\BC)} \big) 
			{\small \begin{pmatrix}
					0&1\\
					0&0
				\end{pmatrix}}.
		\end{equation}
		Moreover, $Z_{\D H}(\phi^{\mathbb{Q}_p})$ is canonically isomorphic to the reductive quotient of $Z_{H_{\lambda^{\mathbb{Q}_p}}}(x)$. Hence the map (\ref{eq: param vs Vogan}) induces a bijection
		\begin{equation*}
			\Xi(\lambda^{\mathbb{Q}_p},H) \bijects \operatorname{Irr} \operatorname{Per} (H_{\lambda^{\mathbb{Q}_p}} \backslash V_{\lambda^{\mathbb{Q}_p}}).
		\end{equation*}
	\end{proposition}
	
	Hence we obtain a conjectural bijection
	\begin{equation*}
		\Pi_{{\rm pure}}(\lambda^{\mathbb{Q}_p}, H) \bijects \Xi(\lambda^{\mathbb{Q}_p},H) \bijects \operatorname{Irr} \operatorname{Per} (H_{\lambda^{\mathbb{Q}_p}} \backslash V_{\lambda^{\mathbb{Q}_p}}).
	\end{equation*}	
For a parameter $\xi = (C,\cV) \in \Xi(\lambda^{\mathbb{Q}_p}, H)$, we write $\pi(\xi) \in \Pi_{{\rm pure}}(\lambda^{\mathbb{Q}_p}, H)$ for the corresponding irreducible representation of some pure inner form $H_{\xi}$ of $H$, and write $M(\xi)$ for the standard representation having $\pi(\xi)$ as its Langlands quotient. Further, we write $\mathcal{P}(\xi) = j_{C!*} \cV[\dim C]$, $\mu(\xi) = (H^0 j_{C!} \cV)[\dim C]$ for the simple and standard objects in the equivariant derived category $D^b( H_{\lambda^{\mathbb{Q}_p}} \backslash V_{\lambda^{\mathbb{Q}_p}})$. Let $\Rep_{{\rm pure}}(\lambda^{\mathbb{Q}_p}, H)$ denote the category of admissible representations of pure inner forms of $H$ with infinitesimal character $\lambda^{\mathbb{Q}_p}$. The simple objects in $\Rep_{{\rm pure}}(\lambda^{\mathbb{Q}_p}, H)$ are precisely $\Pi_{{\rm pure}}(\lambda^{\mathbb{Q}_p}, H)$. Define a perfect pairing 
	\begin{equation}\label{eq: p-adic pairing}
		\langle -,- \rangle : K \Rep_{{\rm pure}}(\lambda^{\mathbb{Q}_p}, H) \times K D^b(H_{\lambda^{\mathbb{Q}_p}} \backslash V_{\lambda^{\mathbb{Q}_p}}) \aro \BC
	\end{equation}
	by 
    \begin{equation*}
			\langle \pi(\xi), \mathcal{P}(\xi') \rangle = (-1)^{\dim C} e(H_{\xi}) \delta_{\xi, \xi'},
	\end{equation*}
	where $C$ is the orbit appearing in $\xi = (C, \cV)$, $e(H_{\xi})$ is the Kottwitz sign of $H_{\xi}$ (\cite{Kottwitz:1983}), and $\delta_{\xi, \xi'}$ is the Kronecker delta.

	\begin{conjecture}[Kazhdan-Lusztig conjecture]\label{conj: LLC/Qp}	    
    Under the above pairing we have
	\begin{equation*}
		\langle M(\xi), \mu(\xi') \rangle = (-1)^{\dim C} e(H_{\xi}) \delta_{\xi, \xi'}.
	\end{equation*}
	\end{conjecture}
	
	The conjecture was first formulated in Zelevinsky \cite{Zelevinsky:1981} for $\GL_N(\BQ_p)$, the general case was also known as Lusztig's conjecture on character formula for standard modules \cite{Lusztig:1983}. Here our formulation follows \cite{Vogan:1993}. This conjecture is proven for $\GL_N(\BQ_p)$ through a combination of \cite[Theorem 8.6.23]{ChrissGinzburg:1997} and \cite{Ariki:1996}. For the unipotent representations, the conjecture was first proved by Lusztig \cite{Lusztig:1995B} \cite{Lusztig1:2002} in the case of adjoint groups and later extended to all cases by Solleveld \cite{Solleveld:2025}.  The proof goes by studying the representation theory of graded affine Hecke algebras via geometric methods.
    For more details, we refer to \cite[Theorem 5.4]{Solleveld:2025}. 

For any A-parameter $\psi^{\mathbb{Q}_p}$ of $H$, one can associate an L-parameter by 
    \begin{align}
    \label{eq: associated L-parameter p-adic}
    \phi_{\psi^{\mathbb{Q}_p}} (w, x) := \psi^{\mathbb{Q}_p}\Big( w, x, 
		{\small \begin{pmatrix}
				|w|^{1/2}\\&|w|^{-1/2}
			\end{pmatrix}} \Big), \quad w \in W_{\mathbb{Q}_p}, \, x \in SL_2(\BC)
    \end{align}
The infinitesimal character of $\psi^{\mathbb{Q}_p}$ is defined to be that of $\phi_{\psi^{\mathbb{Q}_p}}$. Denote by $\Psi(G)_{\lambda^{\mathbb{Q}_p}}$ the set of $\D{H}$-conjugacy classes of A-parameters with infinitesimal character $\lambda^{\mathbb{Q}_p}$. Then the above association gives an inclusion $\Psi(H)_{\lambda^{\mathbb{Q}_p}} \hookrightarrow \Phi(H)_{\lambda^{\mathbb{Q}_p}}$.

Lastly, for any $\phi^{\mathbb{Q}_p} \in \Phi(H)_{\lambda^{\mathbb{Q}_p}}$ (resp. $\psi^{\mathbb{Q}_p} \in \Psi(H)_{\lambda^{\mathbb{Q}_p}}$) the ABV-packet $\Pi^{\rm ABV, pure}_{\phi^{\mathbb{Q}_p}}(H)$ (resp.  $\Pi^{\rm ABV, pure}_{\psi^{\mathbb{Q}_p}}(H)$) can be defined in the same way as for the real groups.

\end{clause}

\begin{clause}[The case of $\GL_N(\BQ_p)$]\label{cls: GLNQp param space}
	For $H = \GL_N(\BQ_p)$, we use a superscript $GL$ on $V_{\lambda^{\mathbb{Q}_p}}^{GL}$ and $H_{\lambda^{\mathbb{Q}_p}}^{GL}$ to avoid confusions. 
	
	The Vogan variety $V_{\lambda^{\mathbb{Q}_p}}^{GL}$ admits an alternative description. Let $W := \BC^N$ equipped with the action of $\D \fh = \fgl_N(\BC)$. Write $W_k$ for the eigenspace of $\lambda^{\mathbb{Q}_p}(\Fr) \in \fgl_N(\BC)$ corresponding to the eigenvalue $p^k$. Then there is an isomorphism
	\begin{equation*}
		V_{\lambda^{\mathbb{Q}_p}}^{GL} \cong E_W := \bigoplus_{k \in \BC} \Hom(W_k, W_{k+1}) \subset \operatorname{End}(W) = \fgl_N(\BC).
	\end{equation*}
	Under this isomorphism, $H_{\lambda^{\mathbb{Q}_p}}^{GL}$ is identified with $G_W := \prod_{k \in \BC} \GL(W_k)$ (there are only finitely many $k$'s appearing).
	
	According to Zelevinskii \cite{Zelevinsky:1981}, the set of $H_{\lambda^{\mathbb{Q}_p}}^{GL}$-orbits in $V_{\lambda^{\mathbb{Q}_p}}^{GL}$ can be parameterized combinatorially as follows. A \textit{segment} is a subset in $\BC$ of the form
	\begin{equation*}
		[a,b] = \{a, a+1, a+2, \ldots, b-1, b\} \quad \text{ for some } a,b \in \BC \text{ with } b-a \in \BZ_{\ge 0}.
	\end{equation*}
	A \textit{multisegment} is a multiset of segments, denoted by $\bm = \{[a_j, b_j]\}_j$ or $\bm = \sum_j [a_j, b_j]$. Then there is a bijection
	\begin{align*}
		H_{\lambda^{\mathbb{Q}_p}}^{GL} \backslash V_{\lambda^{\mathbb{Q}_p}}^{GL} &\bijects
		\left\{ \text{multisegments } \bm \mid
		\begin{array}{c}
			\forany k \in \BC, \text{the number of times $k$ occurs}\\
			\text{in $\bm$ equals $\dim W_k$}
		\end{array}\right\},\\
		O_\bm &\longmapsto \bm.
	\end{align*}
	To be more precise, recall that each element $x \in V_{\lambda^{\mathbb{Q}_p}}^{GL}$ can be viewed as an endomorphism on $W = \bigoplus_k W_k$ which is nilpotent and homogeneous of degree one. Hence there exists a homogeneous basis of $W$ under which the matrix of $x$ is in Jordan normal form. This means that each basis vector $v_k \in W_k$ of degree $k$ is a part of a unique chain of the form $0 \xmapsto{x} v_a \xmapsto{x} v_{a+1} \xmapsto{x} \cdots \xmapsto{x} v_b \xmapsto{x} 0$. Such a chain is called a \textit{Jordan cell}. Then for any multisegment $\bm = \sum_j [a_j, b_j]$ of the above form, the corresponding orbit $O_\bm \subset V_{\lambda^{\mathbb{Q}_p}}^{GL}$ consists of those $x$ whose Jordan cells are in bijection with the segments $[a_j,b_j]$, and the cell corresponding to $[a_j,b_j]$ is of the form $0 \xmapsto{x} v_{a_j} \xmapsto{x} \cdots \xmapsto{x} v_{b_j} \xmapsto{x} 0$.
	
	\begin{lemma}\label{lemma: multiseg vs L-param}
		Let $\phi^{\mathbb{Q}_p}$ be a parameter of $H = \GL_N(\BQ_p)$ of the form
		\begin{equation*}
			\phi^{\mathbb{Q}_p} = \bigoplus_{i=1}^r | \cdot |_{W_{\BQ_p}}^{s_i} \boxtimes S_{t_i}
		\end{equation*}
		where $s_i \in \BC$, $t_i \in \BZ_{\ge 1}$, and $S_{t_i}$ denotes the $t_i$-dimensional irreducible representation of $\SL_2(\BC)$. Then under the bijection (\ref{eq: param vs Vogan}), the $\D H$-orbit of $\phi^{\mathbb{Q}_p}$ corresponds to the multisegment
		\begin{equation*}
			\bm = \big\{ [s_i + \tfrac12(1-t_i), s_i + \tfrac12(t_i-1)] \big\}_{i=1}^r.
		\end{equation*}
	\end{lemma}
	
	\begin{proof}
		The element $\lambda^{\mathbb{Q}_p}(\Fr)$ is a block diagonal matrix with $r$ blocks whose $i$-th block is 
		\begin{equation*}
			{\small 
			\begin{pmatrix}
				p^{s_i + \frac12 (t_i-1)}\\
				& p^{s_i + \frac12 (t_i-3)}\\
				&& \ddots\\
				&&& p^{s_i + \frac12 (1-t_i)}
			\end{pmatrix}},
		\end{equation*}
		while $x = d \big( \phi^{\mathbb{Q}_p}|_{\SL_2(\BC)} \big) 
		{\small \begin{pmatrix}
				0&1\\
				0&0
		\end{pmatrix}}$ is a block diagonal matrix with $r$ blocks whose $i$-th block is 
		\begin{equation*}
			{\small
			\begin{pmatrix}
				0 & 1\\
				& 0 & 2\\
				& & \ddots & \ddots\\
				& & & 0 & t_i-1\\
				&&&&0
			\end{pmatrix}}.
		\end{equation*}
		Hence the standard basis vectors $e_j$ corresponding to the $i$-th block form a Jordan cell of $x$ and are eigenvectors of $\lambda^{\mathbb{Q}_p}(\Fr)$ corresponding to eigenvalues $p^{s_i + \frac12 (t_i-1)}, p^{s_i + \frac12 (t_i-3)}, \ldots, p^{s_i + \frac12 (1-t_i)}$. So the $i$-th block corresponds to the segment $[s_i + \frac12 (1-t_i), s_i + \frac12 (t_i-1)]$. The lemma follows.		
	\end{proof}
\end{clause}

\begin{clause}[The case of symplectic/orthogonal groups]\label{cls: SpSO param space}
	We now let $H$ be the $p$-adic symplectic group or split special orthogonal group whose dual group $\D H$ is $\SO(N,\BC)$ or $\Sp(N,\BC)$. 
    
    When $H$ is special even orthogonal, we also denote the corresponding full even orthogonal group by $H^{+}$ and its dual group by $\widehat{H}^{+} = O(N, \mathbb{C})$. Define $\Phi(H^{+})_{\lambda^{\mathbb{Q}_p}}$ to be the quotient of $\Phi(H)_{\lambda^{\mathbb{Q}_p}}$ by $\D H^{+}$-conjugation. Let $H'$ be the non-split inner form of $H$ defined by some quadratic form of discriminant one, and $H'^{+}$ be the corresponding full even orthogonal group. Then
    \[
    \Pi_{\rm pure}(\lambda^{\mathbb{Q}_p}, H) = \Pi(H)_{\lambda^{\mathbb{Q}_p}} \bigsqcup \Pi(H')_{\lambda^{\mathbb{Q}_p}} 
    \]
    and we define 
    \[
    \Pi_{\rm pure}(\lambda^{\mathbb{Q}_p}, H^{+}) := \Pi(H^{+})_{\lambda^{\mathbb{Q}_p}} \bigsqcup \Pi(H'^{+})_{\lambda^{\mathbb{Q}_p}} \]
    the set of isomorphism classes of irreducible admissible representations of $H^{+}$ and $H'^{+}$, obtained from $\Pi_{\rm pure}(\lambda^{\mathbb{Q}_p}, H)$ by induction. The local Langlands correspondence is expected to be extended to the disconnected group $H^{+}$ as follows.
    \begin{conjecture}
    \label{conj: LLC full even orthogonal}
    \begin{equation*}
		\Pi_{{\rm pure}}(\lambda^{\mathbb{Q}_p},H^{+}) \bijects \Xi(\lambda^{\mathbb{Q}_p}, H^{+}) := \{ (\phi^{\mathbb{Q}_p}, \epsilon) \mid \phi^{\mathbb{Q}_p} \in \Phi(H^{+})_{\lambda^{\mathbb{Q}_p}}, \, \epsilon \in \D A^{+}_{\phi^{\mathbb{Q}_p}} \},
	\end{equation*}	
    where $A^{+}_{\phi^{\mathbb{Q}_p}}$ and $\D A^{+}_{\phi^{\mathbb{Q}_p}}$ are defined as in the case of connected reductive groups. 
    \end{conjecture}
    We have analogs of Proposition \ref{prop: param vs Vogan}, the pairing (\ref{eq: p-adic pairing}), and the Kazhdan-Lusztig conjecture \ref{conj: LLC/Qp} for $H^+$, namely we have bijections
    \begin{gather}\label{eq: disconnected param vs Vogan}
			\Phi(H^{+})_{\lambda^{\mathbb{Q}_p}} \bijects H^{+}_{\lambda^{\mathbb{Q}_p}} \backslash V_{\lambda^{\mathbb{Q}_p}},\quad
			\phi^{\mathbb{Q}_p} \mapsto x := d \big( \phi^{\mathbb{Q}_p}|_{\SL_2(\BC)} \big) 
			{\small \begin{pmatrix}
					0&1\\
					0&0
            \end{pmatrix}}, 
            \\
            \Pi_{{\rm pure}}(\lambda^{\mathbb{Q}_p}, H^{+}) \bijects \Xi(\lambda^{\mathbb{Q}_p},H^{+}) \bijects \operatorname{Irr} \operatorname{Per} (H^{+}_{\lambda^{\mathbb{Q}_p}} \backslash V_{\lambda^{\mathbb{Q}_p}}), \nonumber
    \end{gather}
    a perfect pairing
    \begin{equation}\label{eq: disconnected p-adic pairing}
		\langle -,- \rangle : K \Rep_{{\rm pure}}(\lambda^{\mathbb{Q}_p}, H^{+}) \times K D^b(H^{+}_{\lambda^{\mathbb{Q}_p}} \backslash V_{\lambda^{\mathbb{Q}_p}}) \aro \BC
	\end{equation}
	defined by 
    \begin{equation*}
			\langle \pi(\xi), \mathcal{P}(\xi') \rangle = (-1)^{\dim C} e(H^{+}_{\xi}) \delta_{\xi, \xi'},
	\end{equation*}
    and
    \begin{conjecture}[Twisted Kazhdan-Lusztig conjecture]\label{conj: disconnected LLC/Qp}
		Under the above pairing we have
	    \begin{equation*}
		\langle M(\xi), \mu(\xi') \rangle = (-1)^{\dim C} e(H^{+}_{\xi}) \delta_{\xi, \xi'}.
	    \end{equation*}        
	\end{conjecture}

	Returing to $H$, we now describe its Vogan variety in terms of that of $\GL_N$. Set $W$ to be the vector space $\BC^N$. Let $\langle -,- \rangle$ be a non-degenerate bilinear form on $W$ that is symmetric if $\D H = \SO(N,\BC)$, or symplectic if $\D H = \Sp(N,\BC)$. Then we may identify $\D H$ with the subgroup of $\SL_N(\BC)$ preserving the form $\langle -,- \rangle$. That is 
	\begin{equation*}
		\D H = \{g \in \GL_N(\BC) \mid g^* = g\inv, \det g = 1\},\quad 
		\D \fh = \{ x \in \fgl_N \mid x^* = -x\} \subset \fgl_N
	\end{equation*}
	where $g^*$ denotes the adjoint operator of $g: W \to W$ under the form $\langle -,- \rangle$, and similarly for $x^*$. Using this, we identify the set of infinitesimal characters (resp. Langlands parameters) of $H$ with a subset of those of $\GL_N$, and identify the Vogan variety of $H$ with 
	\begin{equation*}
		V_{\lambda^{\mathbb{Q}_p}} = V_{\lambda^{\mathbb{Q}_p}}^{GL} \cap \D \fh = \{ x \in E_W \mid x^* = - x\}.
	\end{equation*}
	and 
	\begin{equation*}
		H_{\lambda^{\mathbb{Q}_p}} = H_{\lambda^{\mathbb{Q}_p}}^{GL} \cap \D H = \{ g \in G_W \mid g^* = g\inv, \det g = 1\}.
	\end{equation*}
    We extend the definition of $H_{\lambda^{\mathbb{Q}_p}}^{+}$ to the case when $H$ is symplectic by 
    \begin{equation*}
    H_{\lambda^{\mathbb{Q}_p}}^{+} = 
    H_{\lambda^{\mathbb{Q}_p}} \times \langle -I_N \rangle.
    \end{equation*}    
    Then 
	\begin{equation*}
		H_{\lambda^{\mathbb{Q}_p}}^{+} = \{ g \in G_W \mid g^* = g\inv\}.
	\end{equation*}	
 whenever it is defined.
    
	\begin{lemma}\label{lemma: SpSO param space}
		Let $\lambda^{\mathbb{Q}_p}$ be an infinitesimal character of $H$ that is \textit{integral} in the sense that the eigenvalues of $\lambda^{\mathbb{Q}_p}(\Fr)$ on $W = \BC^N$ are of the form $p^i$ where the $i$'s differ only by integers.
		\begin{enumerate}
			\item The $i$'s are either all in $\BZ$ or all in $\frac12 \BZ \backslash \BZ$.
		\end{enumerate}
		For convenience, we say $\lambda^{\mathbb{Q}_p}(\Fr)$ is \textbf{supported in $\BZ$} (resp. \textbf{supported in $\frac12 \BZ \backslash \BZ$}) if the $i$'s are all in $\BZ$ (resp. all in $\frac12 \BZ \backslash \BZ$).
		\begin{enumerate} 
			\item[(2)] $\langle -,- \rangle$ restricts to a non-degenerate form on $W_i \times W_{-i}$ for any $i$. In particular, $\dim W_i = \dim W_{-i}$.
			
			\item[(3)] Conversely, if an unramified infinitesimal character $\lambda^{\mathbb{Q}_p}$ of $\GL_N(\BQ_p)$ is so that $\lambda^{\mathbb{Q}_p}(\Fr)$ is supported in $\BZ$ or $\frac12 \BZ \backslash \BZ$ and $\dim W_i = \dim W_{-i}$, then $\lambda^{\mathbb{Q}_p}$ comes from an infinitesimal character of a symplectic/orthogonal group.
			
			\item[(4)] We have the following explicit description of $V_{\lambda^{\mathbb{Q}_p}}$ and $H_{\lambda^{\mathbb{Q}_p}}$:
		\end{enumerate}		
		\begin{center}
			\begin{tabular}{ccccc}
				$N$ & $\langle -,- \rangle$ & $\supp\lambda^{\mathbb{Q}_p}(\Fr)$ & $V_{\lambda^{\mathbb{Q}_p}}$ & $H_{\lambda^{\mathbb{Q}_p}}$
				\\ \hline
				odd & symm &  $\BZ$ & ${\displaystyle \prod_{i \le -1} \Hom(W_i,W_{i+1})}$ & ${\displaystyle \prod_{i \le -1} \GL(W_i) \times \SO(W_0)}$ 
				\\
				even & skew & $\frac12\BZ \backslash \BZ$ & ${\displaystyle \prod_{i < -\frac12} \Hom(W_i,W_{i+1}) \times \operatorname{Sym}(W_{-\frac12})}$ & ${\displaystyle \prod_{i \le -\frac12} \GL(W_i)}$
				\\
				even & symm & $\frac12\BZ \backslash \BZ$ & ${\displaystyle \prod_{i < -\frac12} \Hom(W_i,W_{i+1}) \times \operatorname{Skew}(W_{-\frac12})}$ & ${\displaystyle \prod_{i \le -\frac12} \GL(W_i)}$
				\\
				even & skew & $\BZ$ & ${\displaystyle \prod_{i \le -1} \Hom(W_i,W_{i+1})}$ & ${\displaystyle \prod_{i \le -1} \GL(W_i) \times \Sp(W_0)}$
				\\
				even & symm & $\BZ$ & ${\displaystyle \prod_{i \le -1} \Hom(W_i,W_{i+1})}$ & ${\displaystyle \prod_{i \le -1} \GL(W_i) \times \SO(W_0)}$
			\end{tabular}
		\end{center}
    If $H_{\lambda^{\mathbb{Q}_p}} = {\displaystyle \prod_{i \le -1} \GL(W_i) \times \SO(W_0)}$, then $H^{+}_{\lambda^{\mathbb{Q}_p}} = {\displaystyle \prod_{i \le -1} \GL(W_i) \times {O}(W_0)}$. Otherwise, $H^{+}_{\lambda^{\mathbb{Q}_p}} = H_{\lambda^{\mathbb{Q}_p}}$.    
	\end{lemma}
	
	\begin{proof}
		Since $\lambda^{\mathbb{Q}_p}(\Fr)$ preserves the form $\langle -,- \rangle$, we have $W_i \perp W_j$ if $i \neq -j$. Indeed, if $v_i \in W_i$ and $v_j \in W_j$, then $\langle v_i, v_j \rangle = \langle \lambda^{\mathbb{Q}_p}(\Fr)v_i, \lambda^{\mathbb{Q}_p}(\Fr) v_j \rangle = p^{i+j} \langle v_i, v_j \rangle$ which can be nonzero only when $i=-j$. Since $\langle -,-\rangle$ is non-degenerate on $W$, its restriction to $W_i \times W_{-i}$ must be non-degenerate. Hence for any eigenvalue $p^i$ of $\lambda^{\mathbb{Q}_p}(\Fr)$, $p^{-i}$ is also an eigenvalue. Since all $i$'s differ by integers, this forces $i \in \BZ$ or $\frac12 \BZ \backslash \BZ$, and if one $i$ is in $\BZ$ (resp. $\frac12 \BZ \backslash \BZ$), so are the other $i$'s. This justifies the first two statements.
		
		The adjoint of a homogeneous degree one map $x \in E_W$ is again of degree one, and the adjoint of $g_i \in \GL(W_i)$ is an element $g_i^* \in \GL(W_{-i})$. As a result,
		\begin{multline*}
			V_{\lambda^{\mathbb{Q}_p}} = \{ x \in E_W \mid x_i^* = - x_{-i-1}\} \\
			\cong 
			\begin{cases}
				\prod_{i \le -1} \Hom(W_i,W_{i+1}) & \text{if } i \in \BZ,\\
				\prod_{i < -\frac12} \Hom(W_i, W_{i+1}) \times \{f: W_{-\frac12} \to W_{\frac12} \mid f^* = -f\} & \text{if } i \in \frac12\BZ\backslash \BZ,
			\end{cases}
		\end{multline*}
		\begin{multline*}
			H_{\lambda^{\mathbb{Q}_p}} = \{ g = (g_i) \in G_W \mid g_i^* = g_{-i}\inv, \det g_0 = 1\} \\
			\cong
			\begin{cases}
				\prod_{i \le -1} \GL(W_i) \times \{f \in \GL(W_0) \mid f^* = f\inv, \det f = 1\} & \text{if } i \in \BZ,\\
				\prod_{i \le -\frac12} \GL(W_i) & \text{if } i \in \frac12 \BZ \backslash \BZ.
			\end{cases}
		\end{multline*}
		It remains to simplify the conditions $f^* = -f$ and $f^* = f\inv$. For $f^* = f\inv$, note that if $\langle -,- \rangle$ is symmetric (resp. skew-symmetric), then $f^* = f\inv$ means $f \in O(W_0)$ (resp. $f\in \Sp(W_0)$). Consider the condition $f^* = -f$. For each such $f$ we may construct a bilinear form on $W_{-\frac12}$ by $(v,v') := \langle f(v), v' \rangle$. Using adjoints, we see that
		\begin{multline*}
			(v,v') = \langle f(v), v' \rangle = \langle v, f^*(v') \rangle = - \langle v, f(v') \rangle 
			\\
			= 
			\begin{cases}
				- \langle f(v'),v \rangle = - (v',v) & \text{if $\langle -,- \rangle$ is symmetric},\\
				\langle f(v'),v \rangle = (v',v) & \text{if $\langle -,- \rangle$ is skew-symmetric}.
			\end{cases}
		\end{multline*} 
		Hence we obtain a linear map
		\begin{equation}\label{eq: Sym vs ad}
			\{f: W_{-\frac12} \to W_{\frac12} \mid f^* = -f\} \aro 
			\begin{cases}
				\operatorname{Skew}(W_{-\frac12}) & \text{if $\langle -,- \rangle$ is symmetric},\\
				\operatorname{Sym}(W_{-\frac12}) & \text{if $\langle -,- \rangle$ is skew-symmetric},
			\end{cases}
		\end{equation}
		which is easily shown to be an isomorphism. This completes the proof.
	\end{proof}
\end{clause}

\subsection{The full rank part and the geometric comparison}
\label{subsec: comparison}

We now turn to a geometric relation of the parameter spaces $\cX_{\lambda^{\mathbb{R}}}$ and $[H_{\lambda^{\mathbb{Q}_p}} \backslash V_{\lambda^{\mathbb{Q}_p}}]$.

\begin{clause}[The case of $\GL_N$]\label{cls: geom comparison GLn}
	Let $\lambda^{\mathbb{Q}_p}$ be an integral infinitesimal character of $\GL_N(\BQ_p)$. The \textbf{full rank part} of $V_{\lambda^{\mathbb{Q}_p}}^{GL} = E_W$ is the $H_{\lambda^{\mathbb{Q}_p}}^{GL} = G_W$-stable open subspace
	\begin{equation*}
		E_W^{reg} = \prod_i \Hom(W_i, W_{i+1})_{reg},
	\end{equation*}
	where a linear map $f: W_k \to W_{k+1}$ is said to be regular if it has maximal possible rank. We would like to identify the quotient $[G_W \backslash E_W^{reg}]$ with the ABV space of a real group $\GL_n(\BC)$.
	
	Suppose $\lambda^{\mathbb{Q}_p}$ satisfies the following condition: 
	\begin{assumption}\label{assump: lambdap}
		There is a number $i_0$ so that 
		\begin{itemize}
			\item $p^{i_0}$ is an eigenvalue of $\lambda^{\mathbb{Q}_p}(\Fr)$ on $W = \BC^N$;
			\item $\forany i \le i_0$, $\dim W_i \le \dim W_{i+1}$;
			\item $\forany i \ge i_0$, $\dim W_i \ge \dim W_{i+1}$.
		\end{itemize}
	\end{assumption}	
	Consider the space
	\begin{equation*}
		E_W^{reg,\diamond} = \prod_{i \neq i_0} \Hom(W_i, W_{i+1})_{reg}
	\end{equation*}
	equipped with the natural $G_W$-action and the map
	\begin{equation*}
		\varphi: E_W^{reg} \surjects E_W^{reg,\diamond}
	\end{equation*}
	obtained by forgetting the $\Hom(W_{i_0},W_{i_0+1})_{reg}$ component. The space $E_W^{reg,\diamond}$ is equal to the product $E_W^{reg,<} \times E_W^{reg,>}$, where
	\begin{equation*}
		E_W^{reg,<} = \prod_{i < i_0} \Hom(W_i, W_{i+1})_{reg},\quad
		E_W^{reg,>} = \prod_{i > i_0} \Hom(W_i, W_{i+1})_{reg}.
	\end{equation*}
	Similarly, $G_W = G_W^< \times G_W^>$, where
	\begin{equation*}
		G_W^< = \prod_{i \le i_0} \GL(W_i),\quad
		G_W^> = \prod_{i > i_0} \GL(W_i)
	\end{equation*}
	
	Let $x^\diamond = (x^<,x^>) \in E_W^{reg,\diamond} = E_W^{reg,<} \times E_W^{reg,>}$ be any point, and write $G_W(x^\diamond) = G_W^<(x^<) \times G_W^>(x^>)$ for its stabilizer in $G_W$. Then by the above assumption on $\lambda^{\mathbb{Q}_p}$, $x^\diamond_i: W_i \to W_{i+1}$ is injective for $i < i_0$ and is surjective for $i > i_0$. Hence we obtain a filtration (i.e. a partial flag) $\operatorname{Fil}_{i_0}$ on $W_{i_0}$ obtained by taking images of the $W_i$'s, $i < i_0$ along the injective maps $(x^<)^{i_0-i}: W_i \to W_{i_0}$. Similarly, we have a filtration $\operatorname{Fil}_{i_0+1}$ on $W_{i_0+1}$ obtained by taking kernels of the surjective maps $(x^>)^{i-(i_0+1)}: W_{i_0+1} \to W_i$ for $i \ge i_0+1$. Thus we obtain parabolic subgroups $P_{i_0} \subset \GL(W_{i_0})$, $P_{i_0+1} \subset \GL(W_{i_0+1})$ as the stabilizers of the filtrations $\operatorname{Fil}_{i_0}$, $\operatorname{Fil}_{i_0+1}$, respectively. By choosing a basis of the vector spaces $W_{i_0}$ and $W_{i_0+1}$, we may identify $\GL(W_{i_0})$ and $\GL(W_{i_0+1})$ with $\GL_n(\BC)$, where $n =  \dim W_{i_0} = \dim W_{i_0+1}$, and identify $P_{i_0}$ and $P_{i_0+1}$ as parabolic subgroups of $\GL_n(\BC)$.
	
	\begin{proposition}[{\cite[Thm 3.1.5]{DHXZ:GLn}}]\label{prop: geom comparison GLn}~
		\begin{enumerate}
			\item $E_W^{reg,\diamond}$ is a single $G_W$-orbit. More precisely, $E_W^{reg,<}$ (resp. $E_W^{reg,>}$) is a single $G_W^<$- (resp. $G_W^>$-) orbit.
			
			\item The fiber $\varphi\inv(x^\diamond)$ is isomorphic to $\GL_n(\BC)$. Under this isomorphism, the action $G_W^<(x^<) \acts \varphi\inv(x^\diamond)$ is identified with $P_{i_0} \acts \GL_n(\BC)$ by right multiplication by inverse, and the action $G_W^>(x^>) \acts \varphi\inv(x^\diamond)$ is identified with $P_{i_0+1} \acts \GL_n(\BC)$ by left multiplication.
			
			\item We have an isomorphism
			\begin{equation*}
				E_W^{reg} \cong G_W \times_{G_W(x^\diamond)} \varphi\inv(x^\diamond)
			\end{equation*}
			and hence isomorphisms of stacks
			\begin{multline*}
				[G_W \backslash E_W^{reg}] 
				\cong [G_W \backslash (G_W \times_{G_W(x^\diamond)} \varphi\inv(x^\diamond))]
				\cong [G_W(x^\diamond) \backslash \varphi\inv(x^\diamond)]\\
				\cong [(P_{i_0} \times P_{i_0+1}) \backslash \GL_n(\BC)]
				\cong [\Delta \GL_n(\BC) \backslash (\GL_n(\BC)/P_{i_0} \times \GL_n(\BC)/P_{i_0+1})].
			\end{multline*}
		\end{enumerate}
	\end{proposition}

    \begin{remark}[Explicit description of $G_W^<(x^<)$]\label{rmk: stabilizer as matrices}
        We may choose a basis $\{\alpha_{i,1},\ldots, \alpha_{i,\dim W_i}\}$ for each $W_i$ so that $x^<(\alpha_{i,j}) = \alpha_{i+1,j}$ for $i < i_0$. Write $W_i^\circ = \operatorname{span}\{\alpha_{i, 1+\dim W_{i-1}}, \ldots, \alpha_{i, \dim W_i}\}$, a complement to $x^<(W_{i-1})$ inside $W_i$, so that $W_i = x^<(W_{i-1}) \oplus W_i^\circ$. Let $i_{st}$ be the smallest index so that $W_{i_{st}} \neq 0$. Then the stabilizer $G_W^<(x^<)$ can be described as
        \begin{multline*}
            G_W^<(x^<) = 
            \left\{ \left(
            g_{i_{st}},
            \begin{pmatrix}
                g_{i_{st}} & *_{i_{st}, i_{st}+1}\\
                0 & g_{i_{st}+1}
            \end{pmatrix},
            \begin{pmatrix}
                g_{i_{st}} & *_{i_{st}, i_{st}+1} & *_{i_{st}, i_{st}+2}\\
                0 & g_{i_{st}+1} & *_{i_{st}+1, i_{st}+2}\\
                0 & 0 & g_{i_{st}+2}
            \end{pmatrix}
            ,\ldots,
            \vphantom{\begin{pmatrix}
				g_{i_{st}} & *_{i_{st}, i_{st}+1} & \cdots & *_{i_{st}, i_0}\\
				0 & g_{i_{st}+1} & \ddots & \vdots\\
				\vdots & \ddots & \ddots &  *_{i_0-1,i_0}\\
				0 & \cdots & 0 & g_{i_0}
			\end{pmatrix}}\right.\right.\\
            \left.\left. \ldots,
            \begin{pmatrix}
				g_{i_{st}} & *_{i_{st}, i_{st}+1} & \cdots & *_{i_{st}, i_0}\\
				0 & g_{i_{st}+1} & \ddots & \vdots\\
				\vdots & \ddots & \ddots &  *_{i_0-1,i_0}\\
				0 & \cdots & 0 & g_{i_0}
			\end{pmatrix}
            \right) \mid 
            g_i \in \GL(W_i^\circ), *_{jk} \text{ arbitrary} \right\}.
        \end{multline*}
        Here, if $W_i^\circ = 0$ (i.e. $W_{i-1}$ and $W_i$ have the same dimension), we delete the row and the column containing $g_i$. The isomorphism $G_W^<(x^<) \bij P_{i_0}$ is simply the projection onto the last factor. 

        There is a similar description for $G_W^>(x^>)$.
    \end{remark}
    
    \begin{remark}
    The full rank part above was first introduced and studied in the first author's thesis \cite{Deng:2016} by taking unions of orbits indexed by multisegments of parabolic type. A special case named
    symmetric multisegments together with the construction of the map $\varphi$ was studied in \cite[Section 4]{Deng:2023}.
    \end{remark}
    
	Note that the stack $\cX_{\lambda^{\mathbb{R}}}^{GL} = [\Delta \GL_n(\BC) \backslash (\GL_n(\BC)/P_{i_0} \times \GL_n(\BC)/P_{i_0+1})]$ is the ABV space of $\GL_n(\BC)$ (viewed as a real group) for some integral infinitesimal character $\lambda^\BR$. Hence we obtain an open immersion
	\begin{equation*}
		\cX_{\lambda^{\mathbb{R}}}^{GL} \cong [G_W \backslash E_W^{reg}] \injects [G_W \backslash E_W].
	\end{equation*}
\end{clause}

\begin{clause}[A variant]\label{cls: geom comparison GLn variant}
	There is a variant of the above construction. Suppose $\lambda^{\mathbb{Q}_p}$ satisfies the following strengthened condition:
	\begin{assumption}\label{assump: lambdap strong}
		There is a number $i_0$ so that
		\begin{itemize}
			\item $p^{i_0}$ is an eigenvalue of $\lambda^{\mathbb{Q}_p}(\Fr)$ on $W = \BC^N$;
			\item $\forany i \le i_0+1$, $\dim W_i \le \dim W_{i+1}$;
			\item $\forany i \ge i_0$, $\dim W_i \ge \dim W_{i+1}$.
		\end{itemize}
	\end{assumption}
	Then $\dim W_{i_0} = \dim W_{i_0+1} = \dim W_{i_0+2}$. Consider the space
	\begin{equation*}
		E_W^{reg,\diamond} = \prod_{i \neq i_0, i_0+1} \Hom(W_i, W_{i+1})_{reg}
	\end{equation*}
	together with the projection
	\begin{equation*}
		\varphi: E_W^{reg} \surjects E_W^{reg,\diamond}.
	\end{equation*}
	obtained by forgetting both $\Hom(W_{i_0},W_{i_0+1})_{reg}$ and $\Hom(W_{i_0+1},W_{i_0+2})_{reg}$. For any point $x^\diamond \in E_W^{reg,\diamond}$, we again obtain a filtration $\operatorname{Fil}_{i_0}$ of $W_{i_0}$ and a filtration $\operatorname{Fil}_{i_0+2}$ of $W_{i_0+2}$, and hence parabolic subgroups $P_{i_0}$, $P_{i_0+2} \subseteq \GL_n(\BC)$. 
	
	\begin{proposition}~
		\begin{enumerate}
			\item $E_W^{reg,\diamond}$ is a single $G_W$-orbit.
			\item The fiber $\varphi\inv(x^\diamond)$ is isomorphic to $\GL_n(\BC) \times \GL_n(\BC)$. Under this isomorphism, the action $G_W(x^\diamond) \acts \varphi\inv(x^\diamond)$ is identified with
			\begin{gather*}
				(P_{i_0} \times \GL_n(\BC) \times P_{i_0+2}) \acts (\GL_n(\BC) \times \GL_n(\BC)),\\
				(p_{i_0}, g, p_{i_0+2}) \cdot (g_1,g_2) = (p_{i_0+2}g_1 g\inv, g g_2 p_{i_0}\inv).
			\end{gather*}
			\item We have an isomorphism
			\begin{multline*}
				[G_W \backslash E_W^{reg}] 
				\cong [G_W \backslash (G_W \times_{G_W(x^\diamond)} \varphi\inv(x^\diamond))]
				\cong [G_W(x^\diamond) \backslash \varphi\inv(x^\diamond)]\\
				\cong [(P_{i_0} \times \GL_n(\BC) \times P_{i_0+2}) \backslash (\GL_n(\BC) \times \GL_n(\BC))]\\
				\cong [\Delta \GL_n(\BC) \backslash (\GL_n(\BC)/P_{i_0} \times \GL_n(\BC)/P_{i_0+2})].
			\end{multline*}
		\end{enumerate}
	\end{proposition}
	The proof is almost identical to \ref{prop: geom comparison GLn}.
\end{clause}

\begin{clause}[The case of symplectic/orthogonal groups]\label{cls: geom comparison}
	Let us return to the setup of \ref{cls: SpSO param space}. We are going to describe an analogous statement to Proposition \ref{prop: geom comparison GLn}.
	
	Let $\lambda^{\mathbb{Q}_p}$ be an integral infinitesimal character of $H$. Viewing it as an infinitesimal character of $\GL_N(\BQ_p)$ via the inclusion $\D H \inj \GL_N(\BC)$, we assume that $\lambda^{\mathbb{Q}_p}$ satisfies Assumption \ref{assump: lambdap}. Note that since $\dim W_i = \dim W_{-i}$ for any $i$ by Lemma \ref{lemma: SpSO param space}, if $\lambda^{\mathbb{Q}_p}(\Fr)$ is supported in $\BZ$ (resp. $\frac12 \BZ \backslash \BZ$), the largest possible $\dim W_i$ is achieved by $i = -1,0,1$ (resp. $i = -\frac12, \frac12$), and the number $i_0$ in Assumption \ref{assump: lambdap} can be chosen to be $i_0 = -1$ (resp. $i_0 = -\frac12$). Hence if $\lambda^{\mathbb{Q}_p}(\Fr)$ is supported in $\BZ$, $\lambda^{\mathbb{Q}_p}$ in fact satisfies the stronger Assumption \ref{assump: lambdap strong}.
	
	Set
	\begin{multline*}
		V_{\lambda^{\mathbb{Q}_p}}^{reg} := E_W^{reg} \cap V_{\lambda^{\mathbb{Q}_p}}\\
		\cong
		\begin{cases}
			{\displaystyle \prod_{i \le -1} \Hom(W_i,W_{i+1})_{reg}} & \text{if }\lambda^{\mathbb{Q}_p}(\Fr) \text{ supported in } \BZ
				\\
			{\displaystyle \prod_{i< -\frac12} \Hom(W_i, W_{i+1})_{reg} \times \operatorname{Sym}(W_{-\frac12})_{reg}} & \text{if }\lambda^{\mathbb{Q}_p}(\Fr) \text{ supported in } \frac12 \BZ \backslash \BZ, \quad \langle-,- \rangle \text{ is skew}
				\\
			{\displaystyle \prod_{i< -\frac12} \Hom(W_i, W_{i+1})_{reg} \times \operatorname{Skew}(W_{-\frac12})_{reg}} & \text{if }\lambda^{\mathbb{Q}_p}(\Fr) \text{ supported in } \frac12 \BZ \backslash \BZ, \quad \langle-,- \rangle \text{ is symm}
		\end{cases}
	\end{multline*}
	where $\operatorname{Sym}(W_{-\frac12})_{reg}$ (resp. $\operatorname{Skew}(W_{-\frac12})_{reg}$) denotes the set of invertible symmetric (resp. skew-symmetric) matrices (cf. Lemma~\ref{lemma: SpSO param space}). This is an $H_{\lambda^{\mathbb{Q}_p}}$-stable open subvariety of $V_{\lambda^{\mathbb{Q}_p}}$. We then set
	\begin{equation*}
		V_{\lambda^{\mathbb{Q}_p}}^{reg,\diamond} := \varphi(V_{\lambda^{\mathbb{Q}_p}}^{reg})
		=
		\begin{cases}
			{\displaystyle \prod_{i < -1} \Hom(W_i,W_{i+1})_{reg} } & \lambda^{\mathbb{Q}_p}(\Fr) \text{ is supported in } \BZ,\\
			{\displaystyle \prod_{i < -\frac12} \Hom(W_i,W_{i+1})_{reg} } & \lambda^{\mathbb{Q}_p}(\Fr) \text{ is supported in } \frac12\BZ \backslash \BZ,
		\end{cases}
	\end{equation*}
	where $\varphi:	E_W^{reg} \surj E_W^{reg,\diamond}$ is the map defined in \ref{cls: geom comparison GLn} if $\lambda^{\mathbb{Q}_p}(\Fr)$ is supported in $\frac12\BZ \backslash \BZ$, or the one defined in \ref{cls: geom comparison GLn variant} if $\lambda^{\mathbb{Q}_p}(\Fr)$ is supported in $\BZ$. By abuse of notation, we denote the projection $V_{\lambda^{\mathbb{Q}_p}}^{reg} \surj V_{\lambda^{\mathbb{Q}_p}}^{reg,\diamond}$ again by $\varphi$.
	
	As in the $\GL_N$ case, any choice of point $x^\diamond \in V_{\lambda^{\mathbb{Q}_p}}^{reg,\diamond}$ induces a filtration $\operatorname{Fil}_{i_0}$ on $W_{i_0}$, whence a parabolic subgroup $P(\lambda^{\mathbb{Q}_p}) := P_{i_0}$ of $\GL(W_{i_0}) \cong \GL_n(\BC)$, where $n = \dim W_{i_0}$. 
	
	\begin{proposition}\label{prop: full rank part}~
		Suppose $\lambda^{\mathbb{Q}_p}$ satisfies Assumption \ref{assump: lambdap}, and let $n = \dim W_{i_0}$. The space $V_{\lambda^{\mathbb{Q}_p}}^{reg,\diamond}$ is a single $H_{\lambda^{\mathbb{Q}_p}}$-orbit, and hence for any point $x^\diamond \in V_{\lambda^{\mathbb{Q}_p}}^{reg,\diamond}$, we have an isomorphism
		\begin{equation*}
			V_{\lambda^{\mathbb{Q}_p}}^{reg} \simeq  H_{\lambda^{\mathbb{Q}_p}} \times_{H_{\lambda^{\mathbb{Q}_p}}(x^\diamond)} \varphi\inv(x^\diamond)
		\end{equation*}
		and hence
		\begin{equation*}
			[H_{\lambda^{\mathbb{Q}_p}} \backslash V_{\lambda^{\mathbb{Q}_p}}^{reg}] \cong [H_{\lambda^{\mathbb{Q}_p}}(x^\diamond) \backslash \varphi\inv(x^\diamond)].
		\end{equation*}
		The explicit description of the quotient $[H_{\lambda^{\mathbb{Q}_p}}(x^\diamond) \backslash \varphi\inv(x^\diamond)]$ and its alternative description in each case are gathered in the following table. 
		\begin{center}
			\begin{tabular}{ccccc}
				$N$ & $\langle -,- \rangle$ & $\supp\lambda^{\mathbb{Q}_p}(\Fr)$ & $[H_{\lambda^{\mathbb{Q}_p}}(x^\diamond) \backslash \varphi\inv(x^\diamond)]$ & alternative description 
				\\ \hline
				odd & symm &  $\BZ$ & $[(P(\lambda^{\mathbb{Q}_p}) \times \SO(n,\BC)) \backslash \GL_n(\BC)]$ & $[\SO(n,\BC) \backslash \GL_n(\BC) / P(\lambda^{\mathbb{Q}_p})]$ 
				\\
				even & skew & $\frac12\BZ \backslash \BZ$ & $[P(\lambda^{\mathbb{Q}_p}) \backslash \operatorname{Sym}_n(\BC)_{reg}]$ & $[{O}(n,\BC) \backslash \GL_n(\BC) / P(\lambda^{\mathbb{Q}_p})]$
				\\
				even & symm & $\frac12\BZ \backslash \BZ$ & $[P(\lambda^{\mathbb{Q}_p}) \backslash \operatorname{Skew}_n(\BC)_{reg}]$ & $[\Sp(n,\BC) \backslash \GL_n(\BC) / P(\lambda^{\mathbb{Q}_p})]$
				\\
				even & skew & $\BZ$ & $[(P(\lambda^{\mathbb{Q}_p}) \times \Sp(n,\BC)) \backslash \GL_n(\BC)]$ & $[\Sp(n,\BC) \backslash \GL_n(\BC) / P(\lambda^{\mathbb{Q}_p})]$ 
				\\
				even & symm & $\BZ$ & $[(P(\lambda^{\mathbb{Q}_p}) \times \SO(n,\BC)) \backslash \GL_n(\BC)]$ & $[\SO(n,\BC) \backslash \GL_n(\BC) / P(\lambda^{\mathbb{Q}_p})]$ 
			\end{tabular}
		\end{center}
	\end{proposition}
	
	\begin{proof}
		The proof is similar to the $\GL_N$ case. We include details here for completeness. 
		
		Suppose first that $\lambda^{\mathbb{Q}_p}(\Fr)$ is supported in $\frac12\BZ \backslash \BZ$. Recall that in this case $$V_{\lambda^{\mathbb{Q}_p}}^{reg,\diamond} = \prod_{i < -\frac12} \Hom(W_i,W_{i+1})_{reg}, \ H_{\lambda^{\mathbb{Q}_p}} = \prod_{i \le -\frac12} \GL(W_i)$$ 
        and the fiber $\varphi\inv(x^\diamond)$ is identified with either $\operatorname{Sym}(W_{-\frac12})_{reg}$ or $\operatorname{Skew}(W_{-\frac12})_{reg}$ depending on $\langle -,- \rangle$. This is almost identical to the situation $G_W^< \acts E_W^{reg,<}$ studied in the $\GL_N$ case, except the left multiplcation action of the parabolic subgroup $P_{i_0}$ on the fiber $\GL_n(\BC)$ in the $\GL_N$ case is replaced by the congruent action of $P_{i_0}$ on $\operatorname{Sym}(W_{-\frac12})_{reg}$ or $\operatorname{Skew}(W_{-\frac12})_{reg}$. Hence by Proposition \ref{prop: geom comparison GLn}, $V_{\lambda^{\mathbb{Q}_p}}^{reg,\diamond}$ is a single $H_{\lambda^{\mathbb{Q}_p}}$-orbit and the stabilizer $H_{\lambda^{\mathbb{Q}_p}}(x^\diamond)$ can be identified with the parabolic subgroup $P(\lambda^{\mathbb{Q}_p}) = P_{i_0}$ of $\GL(W_{-\frac12})$. Upon choosing a basis of $W_{-\frac12}$, we may identify $\operatorname{Sym}(W_{-\frac12})_{reg}$ (resp. $\operatorname{Skew}(W_{-\frac12})$) with the space $\operatorname{Sym}_n(\BC)_{reg}$ (resp. $\operatorname{Skew}_n(\BC)_{reg}$) which, by Lemma \ref{lemma: sym/skew as ABV spaces}, is the same as $O(n,\BC) \backslash \GL_n(\BC)$ (resp. $\Sp(n,\BC) \backslash \GL_n(\BC)$). This provides the desired description of the quotient $[H_{\lambda^{\mathbb{Q}_p}}(x^\diamond) \backslash \varphi\inv(x^\diamond)]$.
		
		Suppose now $\lambda^{\mathbb{Q}_p}(\Fr)$ is supported in $\BZ$. In this situation $V_{\lambda^{\mathbb{Q}_p}}^{reg,\diamond} = \prod_{i < -1} \Hom(W_i,W_{i+1})_{reg}$, $H_{\lambda^{\mathbb{Q}_p}} = \prod_{i \le -1} \GL(W_i) \times S\operatorname{Isom}(W_0)$ where $S\operatorname{Isom}(W_0)$ is the group of isometries preserving the form $\langle -,- \rangle$ on $W_0$ with determinant $1$,
        and the fiber $\varphi\inv(x^\diamond)$ is simply $\Hom(W_{-1},W_0)_{reg} \cong \GL_n(\BC)$. Evidently $S\operatorname{Isom}(W_0)$ acts trivially on $V_{\lambda^{\mathbb{Q}_p}}^{reg,\diamond}$, so it remains to consider the action $\prod_{i \le -1} \GL(W_i) \acts \prod_{i < -1} \Hom(W_i,W_{i+1})_{reg}$ which again is covered by the study of the $\GL_N$ case. Hence the stabilizer of a point $x^\diamond$ is $P(\lambda^{\mathbb{Q}_p}) \times S\operatorname{Isom}(W_0)$, as required.
	\end{proof}

    \begin{corollary}
     In the above setup, suppose $H$ is special even orthogonal. Then we also have isomorphisms
     \begin{equation*}
			V_{\lambda^{\mathbb{Q}_p}}^{reg} \simeq  H^{+}_{\lambda^{\mathbb{Q}_p}} \times_{H^{+}_{\lambda^{\mathbb{Q}_p}}(x^\diamond)} \varphi\inv(x^\diamond),
            \qquad
            [H^{+}_{\lambda^{\mathbb{Q}_p}} \backslash V_{\lambda^{\mathbb{Q}_p}}^{reg}] \cong [H^{+}_{\lambda^{\mathbb{Q}_p}}(x^\diamond) \backslash \varphi\inv(x^\diamond)],
		\end{equation*}
        and we have the following explicit descriptions
		\begin{center}
			\begin{tabular}{ccccc}
				$N$ & $\langle -,- \rangle$ & $\supp\lambda^{\mathbb{Q}_p}(\Fr)$ & $[H^{+}_{\lambda^{\mathbb{Q}_p}}(x^\diamond) \backslash \varphi\inv(x^\diamond)]$ & alternative description 
				\\ \hline
				even & symm & $\frac12\BZ \backslash \BZ$ & $[P(\lambda^{\mathbb{Q}_p}) \backslash \operatorname{Skew}_n(\BC)_{reg}]$ & $[\Sp(n,\BC) \backslash \GL_n(\BC) / P(\lambda^{\mathbb{Q}_p})]$
				\\
				even & symm & $\BZ$ & $[(P(\lambda^{\mathbb{Q}_p}) \times {O}(n,\BC)) \backslash \GL_n(\BC)]$ & $[{O}(n,\BC) \backslash \GL_n(\BC) / P(\lambda^{\mathbb{Q}_p})]$ 
			\end{tabular}
		\end{center}     
    \end{corollary}
	
	\begin{corollary}\label{cor: geom comparison}
		In the above setup, let $n = \dim W_{i_0}$. There exists an integral infinitesimal character $\lambda^{\mathbb{R}}$ of the quasisplit real unitary group $G = U(\frac n2, \frac n2)$ or $U(\frac{n-1}2, \frac{n+1}2)$ so that 
		\begin{itemize}
			\item $P(\lambda^{\mathbb{Q}_p}) = \D P(\lambda^{\mathbb{R}})$, and
			\item there exist natural maps
			\begin{equation}\label{eq: geometric correspondence 1}
			[H_{\lambda^{\mathbb{Q}_p}} \backslash V_{\lambda^{\mathbb{Q}_p}}] \xhookleftarrow[open]{\iota_{geom}} [H_{\lambda^{\mathbb{Q}_p}} \backslash V_{\lambda^{\mathbb{Q}_p}}^{reg}] \xrightarrow{\theta} \cX_{\lambda^{\mathbb{R}}}
			\end{equation}
			described in the table below.
		\end{itemize}
		
		\begin{center}
        {\small
			\begin{tabular}{c|ccc|cc|c}
				& $N$ & $\langle -,- \rangle$ & $\supp \lambda^{\mathbb{Q}_p}(\Fr)$ & $n$ & $\lambda^{\mathbb{R}}$ &  $\theta: [H_{\lambda^{\mathbb{Q}_p}} \backslash V_{\lambda^{\mathbb{Q}_p}}^{reg}] \to \cX_{\lambda^{\mathbb{R}}}$
				\\ \hline
				(a) & odd & symm &  $\BZ$ & odd & $\BZ^n$ &
				$[\SO(n,\BC) \backslash \GL_n(\BC) / P(\lambda^{\mathbb{Q}_p})] \to [O(n,\BC) \backslash \GL_n(\BC) / \D P(\lambda^{\mathbb{R}})]$
				\\ \hline
				(b) & \multirow{2}{*}{even} & \multirow{2}{*}{skew} & \multirow{2}{*}{$\frac12\BZ \backslash \BZ$} & odd & $\BZ^n$ & 
				\multirow{2}{*}{$[O(n,\BC) \backslash \GL_n(\BC) / P(\lambda^{\mathbb{Q}_p})] = [O(n,\BC) \backslash \GL_n(\BC) / \D P(\lambda^{\mathbb{R}})]$}
				\\
				(c) &&&& even & $(\frac12 \BZ \backslash \BZ)^n$
				\\ \hline
				(d) & even & symm & $\frac12\BZ \backslash \BZ$ & even & $\BZ^n$ &
				$[\Sp(n,\BC) \backslash \GL_n(\BC) / P(\lambda^{\mathbb{Q}_p})] = [\Sp(n,\BC) \backslash \GL_n(\BC) / \D P(\lambda^{\mathbb{R}})]$
				\\ \hline
				(e) & even & skew & $\BZ$ & even & $\BZ^n$ & 
				$[\Sp(n,\BC) \backslash \GL_n(\BC) / P(\lambda^{\mathbb{Q}_p})] = [\Sp(n,\BC) \backslash \GL_n(\BC) / \D P(\lambda^{\mathbb{R}})]$ 
				\\ \hline
				(f) & even & symm & $\BZ$ & even & $(\frac12 \BZ \backslash \BZ)^n$ &
				$[\SO(n,\BC) \backslash \GL_n(\BC) / P(\lambda^{\mathbb{Q}_p})] \to [O(n,\BC) \backslash \GL_n(\BC) / \D P(\lambda^{\mathbb{R}})]$
			\end{tabular}
            }
		\end{center}
        Here $\cX_{\lambda^{\mathbb{R}}}$ is the ABV space of $G$ with infinitesimal character $\lambda^{\mathbb{R}}$, and the map $\theta$ is either an isomorphism or a finite 2-to-1 map. 
        \begin{enumerate}
            \item If $H$ is not special even orthogonal, the maps $\iota_{geom}$ and $\theta$ induce an inclusion of conjugacy classes of L-parameters, which we again denote by $\iota_{geom}$
            \begin{equation*}
                \iota_{geom}: \Phi(G)_{\lambda^{\mathbb{R}}} \injects \Phi(H)_{\lambda^{\mathbb{Q}_p}}.
            \end{equation*}
            The composition $\theta_* \circ \iota_{geom}^*$ defines a map
            \begin{equation*}
                \cR: K \operatorname{Per}(H_{\lambda^{\mathbb{Q}_p}} \backslash V_{\lambda^{\mathbb{Q}_p}}) \aro K \operatorname{Per}(\cX_{\lambda^{\mathbb{R}}})
            \end{equation*}
            whose adjoint 
            \begin{equation*}
                \cR^*: K \Rep_{{\rm pure}}(\lambda^{\mathbb{R}}, G)  \aro  K \Rep_{{\rm pure}}(\lambda^{\mathbb{Q}_p}, H)
            \end{equation*}
            under the pairings \ref{thm: LLC/R}, \ref{conj: LLC/Qp} \ref{conj: disconnected LLC/Qp} sends standard representations to standard representations up to some signs and restricts to a map on irreducible objects (ignoring the signs)
            \begin{equation*}
                \tilde \iota_{geom}: \Pi_{{\rm pure}}(\lambda^{\mathbb{R}}, G) \aro \Pi_{{\rm pure}}(\lambda^{\mathbb{Q}_p}, H).
            \end{equation*}
            Moreover, $\tilde \iota_{geom}$ is injective if $H$ is not symplectic; if $H$ is symplectic, $\tilde \iota_{geom}$ induces an injection on the quotient that identifies the irreducible representations of $U(p,q)$ with $U(q,p)$.
        
            \item If $H$ is even orthogonal, we have maps 
            \begin{equation}\label{eq: geometric correspondence 2}
        		[H^{+}_{\lambda^{\mathbb{Q}_p}} \backslash V_{\lambda^{\mathbb{Q}_p}}] \xhookleftarrow[open]{\iota_{geom}} [H^{+}_{\lambda^{\mathbb{Q}_p}} \backslash V_{\lambda^{\mathbb{Q}_p}}^{reg}] \xrightarrow{\theta^{+}} \cX_{\lambda^{\mathbb{R}}}
            \end{equation}
            where $\theta^+$ is an isomorphism. In the same way these maps induce
            \begin{gather*}
                \iota_{geom}: \Phi(G)_{\lambda^{\mathbb{R}}} \injects \Phi(H^{+})_{\lambda^{\mathbb{Q}_p}},\\
                \cR^{+}: K \operatorname{Per}(H^{+}_{\lambda^{\mathbb{Q}_p}} \backslash V_{\lambda^{\mathbb{Q}_p}}) \aro K \operatorname{Per}(\cX_{\lambda^{\mathbb{R}}}),\\
                \cR^{+*}: K \Pi_{{\rm pure}}(\lambda^{\mathbb{R}}, G)  \aro  \Pi_{{\rm pure}}(\lambda^{\mathbb{Q}_p}, H^{+}),\\
                \tilde \iota_{geom}: \Pi_{{\rm pure}}(\lambda^{\mathbb{R}}, G) \injects \Pi_{{\rm pure}}(\lambda^{\mathbb{Q}_p}, H^{+}),
            \end{gather*}
            where $\cR^{+*}$ sends standard representations to standard representations up to some signs and $\tilde \iota_{geom}$ is injective.
        \end{enumerate}

	\end{corollary}
	
	Observe that all the identification of quotient stacks that have appeared above are induction equivalences, i.e. they are all of the form $[G_2 \backslash X] \cong [G_1 \backslash Y]$ for some algebraic groups $G_1 \supseteq G_2$ and some $G_2$-variety $X$, where $Y = G_1 \times_{G_2} X$. In this setting, it is shown in \cite[Prop 20.2]{ABV:1992} that the equivalence of categories $\operatorname{Per}(G_2 \backslash X) \cong \operatorname{Per}(G_1 \backslash Y)$ preserves microlocal multiplicities. As a result, the map $\tilde \iota_{geom}$ matches ABV packets on both sides. In more detail, we have the following theorem.
	
	\begin{theorem}
    \label{thm: ABV-packet}
		In the above setup, suppose $\iota_{geom}(\phi^\BR) = \phi^{\mathbb{Q}_p}$. If $H$ is odd orthogonal, then 
		\begin{equation*}
			\tilde \iota_{geom}: \Pi_{\phi^{\mathbb{R}}}^{\rm ABV, pure}(G) \bijects \Pi_{\phi^{\mathbb{Q}_p}}^{\rm ABV, pure}(H).
		\end{equation*}
        If $H$ is even orthogonal, then  
         \begin{equation*}
			\tilde \iota_{geom}: \Pi_{\phi^{\mathbb{R}}}^{\rm ABV, pure}(G) \bijects \Pi_{\phi^{\mathbb{Q}_p}}^{\rm ABV, pure}(H^{+}).
		\end{equation*}
        If $H$ is symplectic, then          
		\[
			\tilde \iota_{geom}:
			\bigsqcup_{\substack{p + q = n \\ p \equiv \frac{n-1}{2} \, {\rm mod} \, 2 }} \Pi_{\phi^{\mathbb{R}}}^{\rm ABV}(U(p,q)) 
			\bijects
			\Pi_{\phi^{\mathbb{Q}_p}}^{\rm ABV, pure}(H).
		\]
	\end{theorem}
\end{clause}

\subsection{Explicit correspondence}
\label{subsec: explicit correspondence}

The relation between real and $p$-adic Langlands parameter spaces described in Corollary \ref{cor: geom comparison} is quite flexible. Namely, for an integral infinitesimal character $\lambda^{\mathbb{R}}$ of $G$, there are different split $p$-adic classical groups $H$ and infinitesimal characters $\lambda^{\mathbb{Q}_p}$ for which $\tilde \iota_{geom}$ can be constructed. In the introduction, we have constructed for each integral $\lambda^{\mathbb{R}}$ a family of  $(H,\lambda^{\mathbb{Q}_p})$ and we will show that they fall into the cases of Corollary \ref{cor: geom comparison}. Let us recall our construction here. Let $\lambda^{\mathbb{R}} = (\lambda_1,\ldots, \lambda_n)$ be an integral dominant infinitesimal character of $G$. We choose $\delta \in \mathbb{Z}$ such that $\lambda_n > (1 - \delta)/2$. Let $\tilde{\lambda}_i = \lambda_i + (\delta - 1)/2$ for $1 \leqslant i \leqslant n$. Then $\tilde{\lambda}_n > 0$.

\begin{clause}[The assignment $\lambda^{\mathbb{R}} \mapsto \lambda^{\mathbb{Q}_p}$]\label{cls: lambda to lambdap}
	We describe an assignment of infinitesimal characters that fit into the cases (a), (c), (e) (resp. (b), (d), (f)) of \ref{cor: geom comparison} if $\delta$ is odd (resp. even). Let 
	\begin{equation*}
		N = \sum_{x \in [-\tilde{\lambda}_1,\tilde{\lambda}_1]} (\text{the number of $\tilde{\lambda}_i$'s with } \tilde{\lambda}_i \ge |x|) = \sum_{i = 1}^{n} (2\tilde{\lambda}_i + 1) =  \sum_{i = 1}^{n} (2\lambda_i + \delta).	\end{equation*}
	Here, as in \ref{cls: GLNQp param space}, $[-\tilde{\lambda}_1,\tilde{\lambda}_1]$ denotes the segment $\{-\tilde{\lambda}_1, -\tilde{\lambda}_1+1,\ldots, \tilde{\lambda}_1\}$. To $\lambda^{\mathbb{R}}$ we attach the unramified integral infinitesimal character $\lambda^{\mathbb{Q}_p}$ of $\GL_N(\BQ_p)$ so that, as multisets,
	\begin{equation*}
		\big\{ \text{eigenvalues of } \lambda^{\mathbb{Q}_p}(\Fr) \big\}
		= \sum_{x \in [-\tilde{\lambda}_1,\tilde{\lambda}_1]} (\text{the number of $\tilde{\lambda}_i$'s with } \tilde{\lambda}_i \ge |x|) \cdot [p^x]
	\end{equation*}
	where $c \cdot [d]$ denotes the multiset $\{\underbrace{d,\ldots,d}_{c \text{ times}}\}$. 
	
	Notice that the multiplicities of $[p^x]$ are symmetric with respect to $x=0$ and increase as $x$ gets closer to $0$. Hence $\lambda^{\mathbb{Q}_p}$ comes from an infinitesimal character of a symplectic/special orthogonal group $H$ and satisfies Assumption \ref{assump: lambdap}. The conditions in each case is collected in the following table. 
\begin{center}
			\begin{tabular}{c|c|c|c|c|c|c}
				& $n$ & $\lambda^{\mathbb{R}}$ & $\delta$ & $N$ &$H$ & parity
				\\ \hline
				(a) & odd & $\mathbb{Z}^n$ & odd & odd & $Sp(N-1)$ & good
				\\ \hline
				(b) & odd & $\mathbb{Z}^n$ & even & even & $SO(N+1)$ & good
				\\ \hline
				(c) & even & $(\frac{1}{2}\mathbb{Z} \backslash \mathbb{Z})^n$ & odd & even & $SO(N+1)$ & good
				\\ \hline
                (d) & even & $\mathbb{Z}^n$ & even & even & $O(N)$ & bad
				\\ \hline
                (e) & even & $\mathbb{Z}^n$ & odd & even & $SO(N+1)$ & bad
				\\ \hline
                (f) & even & $(\frac{1}{2}\mathbb{Z} \backslash \mathbb{Z})^n$ & even & even & $O(N)$ & good
                \end{tabular}
		\end{center}
Here in the even orthogonal case, we will take $H$ to be the split full even orthogonal group, which has been denoted by $H^{+}$ in the previous subsection.

\end{clause}

\begin{clause}[Explicit map between L-parameters]
	In the introduction, we have defined a map
	\begin{equation}\label{eq: iota}
		\iota: \Phi(G)_{\lambda^{\mathbb{R}}} \aro \Phi(H)_{\lambda^{\mathbb{Q}_p}},\quad
		\phi^\BR \mapsto \phi^{\mathbb{Q}_p}
	\end{equation}
	where if 
	\begin{equation}\label{eq: iota a}
		BC(\phi^\BR) = \bigoplus_{i = 1}^n (z/\bar z)^{t_i} (z \bar z)^{s_i}
	\end{equation}
	then
	\begin{equation}\label{eq: iota b}
		{\rm std}_{\D{H}} \circ \phi^{\mathbb{Q}_p} := \bigoplus_{i = 1}^n | \cdot |_{W_{\BQ_p}}^{s_i} \boxtimes S_{2t_i+\delta}.
	\end{equation}
	On the other hand, from Corollary \ref{cor: geom comparison} we have a map on L-parameters constructed from geometry
	\begin{equation*}
		\iota_{geom}: \Phi(G)_{\lambda^{\mathbb{R}}} \aro \Phi(H)_{\lambda^{\mathbb{Q}_p}}.
	\end{equation*}
	Both $\iota$ and $\iota_{geom}$ lift to maps on irreducible representations
	\begin{equation*}
		\tilde \iota, \; \tilde \iota_{geom}: \Pi_{{\rm pure}}(\lambda^{\mathbb{R}}, G) \aro \Pi_{{\rm pure}}(\lambda^{\mathbb{Q}_p}, H)
	\end{equation*}
    (See Subsection~\ref{subsec: main} for  the definition of $\tilde{\iota}$).
	
	\begin{proposition}\label{prop: iota vs iotageom}
		Fixing a pairing $\langle -,- \rangle$ on $W=\C^N$ which induces an embedding of $\D H$ into $\GL_N(\C)$ , a point $x^\diamond$ in $V_{\lambda^{\mathbb{Q}_p}}^{reg,\diamond}$, and a basis on $W$, we have 
		\begin{equation*}
			\iota = \iota_{geom}, \quad \tilde \iota = \tilde \iota_{geom}.
		\end{equation*}
	\end{proposition}

    In order to match the maps, it would be convenient to choose a basis of $W$ that plays well with $x^\diamond$.

    \begin{clause}[Explicit basis of $W_i$]\label{cls: basis of W}
			In the setup of \ref{cls: lambda to lambdap}, the eigenvalues of $\lambda^{\mathbb{Q}_p}(\Fr)$ are of the form $p^k$ where $k$ is in the segment
			\begin{equation*}
				I := [- \tilde{\lambda}_1, \tilde{\lambda}_1] = \{- \tilde{\lambda}_1, -\tilde{\lambda}_1 +1, \ldots, \tilde{\lambda}_1\}.
			\end{equation*}
            Fix an element $x^\diamond \in V_{\lambda^{\mathbb{Q}_p}}^{reg,\diamond}$ and view it as an element of $E_W^{reg,\diamond}$ via the inclusion $V_{\lambda^{\mathbb{Q}_p}} \inj E_W$. As explained in \ref{cls: geom comparison GLn}, we may decompose $x^\diamond = (x^<,x^>)$, and $x^<$ induces a filtration $\operatorname{Fil}_{i_0}$ on $W_{i_0}$, where $i_0 = -1$ if $\lambda^{\mathbb{Q}_p}(\Fr)$ is supported in $\BZ$, and $i_0 = -\frac12$ if $\lambda^{\mathbb{Q}_p}(\Fr)$ is supported in $\frac12\BZ \backslash \BZ$. Namely, the pieces in the filtration $\operatorname{Fil}_{i_0}$ are the images of the injective maps $(x^<)^{i_0-i}: W_i \to W_{i_0}$. 
			
			We may then choose a basis $\{\alpha_1,\ldots, \alpha_n\}$ of $W_{i_0}$ so that its subsets form bases of the subspaces in the filtration $\operatorname{Fil}_{i_0}$. More precisely, $\{\alpha_1,\ldots, \alpha_{\dim W_i}\}$ is a basis for the subspace $\Im((x^<)^{i_0-i}: W_i \to W_{i_0})$. By means of the injective map $(x^<)^{i_0-i}$, we obtain a basis for each $W_i$ ($i \le i_0$) which we denote again by the $\alpha_j$'s. Then $x^<$ sends basis vectors to basis vectors. In particular, $\{\alpha_1,\ldots, \alpha_i\}$ is a basis for $W_{-\tilde{\lambda}_i}$ if $\tilde{\lambda}_i > \tilde{\lambda}_{i+1}$.		
			
			Using the perfect pairing $\langle -,- \rangle$ on each $W_i \times W_{-i}$, we obtain bases of the $W_{-i}$'s dual to the ones on $W_i$, and since $(x^>)^* = -x^<$, $x^>$ sends a basis vector to a basis vector or to zero. Write $\{\alpha_1^*,\ldots, \alpha_n^*\}$ for the dual basis on $W_{-i_0}$ and identify $\{\alpha_1^*,\ldots, \alpha_i^*\}$ with the basis of $W_{\tilde{\lambda}_i}$ in case $\tilde{\lambda}_i > \tilde{\lambda}_{i+1}$.

            If $\lambda^{\mathbb{Q}_p}(\Fr)$ is supported in $\BZ$, we also choose a basis $\{\beta_1, \cdots, \beta_n\}$ of $W_0$ such that
            \[
            \langle\beta_i, \beta_j\rangle = \begin{cases} 
            \delta_{ij} & \text{$\lambda^{\mathbb{R}}$ good parity,}
            \\(-1)^{i+1}\delta_{i, n+1-j} & \text{$\lambda^{\mathbb{R}}$ bad parity.}            \end{cases}
            \]
            \end{clause}
    
	\begin{proof}[Proof of $\iota = \iota_{geom}$]
		We first verify $\iota = \iota_{geom}$. Note that two L-parameters of $H$ are $\D H$-conjugate if and only if they are $\GL_N(\BC)$-conjugate when viewed as maps into $\GL_N(\BC)$. Similarly, two points in the Vogan variety $V_{\lambda^{\mathbb{Q}_p}}$ are $H_{\lambda^{\mathbb{Q}_p}}$-conjugate if and only if their image under the inclusion $V_{\lambda^{\mathbb{Q}_p}} \inj V_{\lambda^{\mathbb{Q}_p}}^{GL} = E_W$ are $H_{\lambda^{\mathbb{Q}_p}}^{GL} = G_W$-conjugate. Hence it suffices to show that the following diagram commutes for any $\phi^\BR \in \Phi(G)_{\lambda^{\mathbb{R}}}$
		\begin{equation*}
			\begin{tikzcd}
				\Phi(G)_{\lambda^{\mathbb{R}}} \ar[r, "\iota", "(\ref{eq: iota})"'] \ar[d, "\text{\ref{prop: orbit vs L-param}}"']
				& \Phi(H)_{\lambda^{\mathbb{Q}_p}} \ar[r, "\text{\ref{lemma: multiseg vs L-param}}"']
				& G_W \backslash E_W
				\\
				\D K \backslash \D G / \D P(\lambda^{\mathbb{R}}) \ar[r, "\text{\ref{lemma: sym/skew as ABV spaces}}"']
				& \operatorname{S}_n(\BC)_{reg}/ \D P(\lambda^{\mathbb{R}}) \ar[r, "\text{\ref{cor: geom comparison}}"', "\iota_{geom}"]
				& H_{\lambda^{\mathbb{Q}_p}} \backslash V_{\lambda^{\mathbb{Q}_p}} \ar[u]
			\end{tikzcd}
			\quad
			\begin{tikzcd}
				\phi^\BR \ar[r, mapsto] \ar[d, mapsto]
				& \phi^{\mathbb{Q}_p} \ar[r, mapsto]
				& O_\bm
				\\
				h \ar[r, mapsto]
				& \dot s \ar[r, mapsto]
				& x \ar[u, mapsto]
			\end{tikzcd}.
		\end{equation*}
		Here $\operatorname{S}_n$ denotes $\operatorname{Sym}_n$ or $\operatorname{Skew}_n$, depending on whether $\lambda$ is of good parity or bad parity.
		
		Suppose $BC(\phi^\BR)$ is of the form (\ref{eq: iota a}). Tracing through the top path, $\phi^{\mathbb{Q}_p}$ is given by (\ref{eq: iota b}), and by Lemma \ref{lemma: multiseg vs L-param} the orbit $O_\bm \in G_W \backslash E_W$ corresponds the multisegment
		\begin{equation}\label{eqn: assign-multisegment}
			\bm = \big\{ [s_i + \tfrac12(1-(2t_i+\delta)), s_i + \tfrac12((2t_i+\delta)-1)] \big\}_{i=1}^r
			= \big\{ [s_i - t_i -(\delta - 1)/2, s_i + t_i + (\delta - 1)/2] \big\}_{i=1}^r
		\end{equation}
		On the other hand, by Lemma \ref{lemma: Sym vs L-param}, the left vertical arrow sends $\phi^\BR$ to $h$ which is then sent to the matrix $\dot s$. If $\lambda^{\mathbb{R}}$ has good parity, then $\dot s$ is a permutation matrix so that the corresponding permutation $s$ satisfies
        (cf. Lemma \ref{lemma: Sym vs L-param})
		\begin{equation}\label{eq: slambda=mu}
			s \cdot (s_1 + t_1, \ldots, s_n + t_n) = (-s_1 + t_1, \ldots, -s_n + t_n),
		\end{equation}
		and $h$ is any matrix so that ${}^t h h = \dot s$. If $\lambda^{\mathbb{R}}$ has bad parity, $\dot s$ is a signed skew symmetric permutation matrix satisfying the above equation, and $h$ is so that ${}^t h J\inv h = \dot s$. To proceed, we need to trace through the definition of $\iota_{geom}$ more carefully. 
		
		Suppose $\lambda^{\mathbb{Q}_p}(\Fr)$ is supported in $\BZ$ and $\lambda^{\mathbb{R}}$ has good parity (resp. bad parity). Then $\iota_{geom}(\dot s)$ is the $H_{\lambda^{\mathbb{Q}_p}}$-orbit of the map $x = (x^<, x_{-1}, x_0, x^>) \in V_{\lambda^{\mathbb{Q}_p}} \subset E_W$ where $x^\diamond = (x^<, x^>)$ is the element fixed in this proposition, and $x_0 \circ x_{-1}: W_{-1} \to W_1$ corresponds, under the isomorphisms $W_{-1} \cong \BC^n \cong W_1$ induced by the bases chosen in \ref{cls: basis of W}, to the map $(-{}^t h) h = -\dot s : \BC^n \to \BC^n$ (resp. $-{}^th J^{-1}h = -
        \dot{s}: \BC^n \to \BC^n$).
        By (\ref{eq: slambda=mu}), we have
		\begin{equation*}
			-\dot s \cdot ( s_1 - t_1, \ldots, s_n - t_n) = (s_1 + t_1, \ldots, s_n + t_n)
		\end{equation*}
		which implies the multisegment corresponding to $x$ is $\bm$. Hence the diagram commutes.
		
		Suppose $\lambda^{\mathbb{Q}_p}(\Fr)$ is supported in $ \frac{1}{2}\mathbb{Z}\backslash \mathbb{Z}$ and $\lambda^{\mathbb{R}}$ has good parity (resp. bad parity). Then $\iota_{geom}(\dot s)$ is the $H_{\lambda^{\mathbb{Q}_p}}$-orbit of $x = (x^<, x_{-\frac12}, x^>)$ where $x_{-\frac12}: W_{-\frac12} \to W_{\frac12}$ corresponds to $\dot s$ under the identification
		\begin{equation} 
			\{f: W_{-\frac12} \to W_{\frac12} \mid f^* = -f\} \bijects
			\operatorname{Sym}(W_{-\frac12}) \bijects
			\operatorname{Sym}_n(\BC),
		\end{equation}
        \begin{equation} 
			\Big({\rm resp.} \quad \{f: W_{-\frac12} \to W_{\frac12} \mid f^* = -f\} \bijects
			\operatorname{Skew}(W_{-\frac12}) \bijects
	     \operatorname{Skew}_n(\BC) \Big)
		\end{equation}
		where the first map is (\ref{eq: Sym vs ad}) sending $f$ to the symmetric (resp. skew-symmetric) form $(v,v') = \langle f(v), v' \rangle$ and the second map takes the matrix of the form $(-,-)$. Therefore the matrix of $x_{-\frac12}$ with respect to the bases of $W_{-\frac12}$ and $W_{\frac12}$ chosen in \ref{cls: basis of W} is equal to $\dot s$, and by the same argument as in the previous case, the multisegment corresponding to $x$ is $\bm$. Thus the diagram commutes. This shows that $\iota = \iota_{geom}$.
	\end{proof}

	\begin{proof}[Proof of $\tilde{\iota} = \tilde{\iota}_{geom}$]
        We now turn to the proof of the equality $\tilde{\iota} = \tilde{\iota}_{geom}$.
        In the bad parity case (cf. (d) (e)) the component groups $A_{\phi^{\BR}}$ and $A_{\phi^{\BQ_p}}$ are both trivial, hence there is nothing to prove. Consider the good parity cases. We need to show that the isomorphism $A_{\phi^{\mathbb{R}}} \cong A^{+}_{\phi^{\mathbb{Q}_p}}$ for $H$ symplectic (resp. $A_{\phi^{\mathbb{R}}} \cong A_{\phi^{\mathbb{Q}_p}}$ for $H$ orthogonal) defined combinatorially in \textsection \ref{subsec: main} is the same as the one given by \eqref{eq: geometric correspondence 2} (resp. \eqref{eq: geometric correspondence 1}). 
        
        Recall from Lemma \ref{lemma: Sym vs L-param} that $\phi^\BR$ corresponds to an involution $s$ in the Weyl group $W^{\D G}$ so that $BC(\phi^\BR) = z^{\lambda^{\mathbb{R}}} \bar z^{-s\lambda^{\mathbb{R}}}$. Viewing $W^{\D G}$ as the symmetric group, in fact $s$ can be chosen to satisfy $s(i) = i$ if and only if $\lambda_i = \lambda_{s(i)}$, and such $s$ is unique up to conjugation by $W^{\D M(\lambda^{\mathbb{R}})}$, see Lemma \ref{prop: Weyl group parametrization} . Taking $\dot s\in \GL_n(\BC)$ to be a permutation matrix lifting $s$ and taking $h\in \GL_n(\BC)$ so that $\dot s = {}^t h h$, it follows from the proof of Lemma \ref{lemma: Sym vs L-param} that $\phi^{\mathbb{R}}$ corresponds to the element $(y,\D P(\lambda^{\mathbb{R}}))$ in the ABV space, where $y=\Ad(h)^{-1}(J^{-1}\tau)\in \cI(\lambda^{\mathbb{R}})$. Recall also that we have identified
        \[
            A_{\phi^{\mathbb{R}}} \cong (\mathbb{Z}/2\mathbb{Z})^{I^{+}}, 
        \]
        where $I^+$ is defined in (\ref{eq: index set}). Our first task is to describe $I^+$ in terms of the element $s$.
        
        \begin{lemma}\label{lem: characterization-I-plus}
            The set $I^{+}$ is in bijection with the set $\{\lambda_j | 1 \leqslant j \leqslant n, s(j)=j\}$. More precisely, writing
            \[
            BC(\phi^{\BR})=z^{\lambda^{\mathbb{R}}} \bar z^{-s\lambda^{\mathbb{R}}}
            =\bigoplus_{j=1}^n (z/\bar{z})^{\frac{\lambda_j+\lambda_{s(j)}}{2}}(z\bar{z})^{\frac{\lambda_j-\lambda_{s(j)}}{2}},
            \]
            the irreducible components of $BC(\phi^\BR)$ indexed by $\lambda_j$ is given by
            \begin{equation*}
            	\bigoplus_{\substack{i : s(i) = i\\\lambda_i=\lambda_j}} (z/\bar{z})^{\frac{\lambda_i+\lambda_{s(i)}}{2}}(z\bar{z})^{\frac{\lambda_i-\lambda_{s(i)}}{2}}
            	= \bigoplus_{\substack{i : s(i) = i\\\lambda_i=\lambda_j}} (z/\bar{z})^{\lambda_j}.
            \end{equation*}
        \end{lemma}
        \begin{proof}[Proof of Lemma \ref{lem: characterization-I-plus}]
			 From the definition (\ref{eq: index set}), $I^{+}$ is indexed by the set
			 \[
				\Big\{ \frac{\lambda_j+\lambda_{s(j)}}{2}| 1 \leqslant j \leqslant n, \lambda_j-\lambda_{s(j)}=0 \Big\}
				=\{\lambda_j|1 \leqslant j \leqslant n, \lambda_j =\lambda_{s(j)}\}.
			 \]
			 We have chosen our $s$ so that $\lambda_j = \lambda_{s(j)}$ if and only if $s(j) = j$. So $I^+$ is indexed by the set $\{\lambda_j |1 \leqslant j \leqslant n, s(j)=j\}$.
        \end{proof}


		Let us choose any $j$ with $s(j) = j$. We first construct an explicit element in the generator of $\BZ/2\BZ \subset A_{\phi^\BR}$ indexed by $\lambda_j$. Define $\varepsilon_j \in GL_{n}(\mathbb{C})$ by $\varepsilon_j(e_i) = (-1)^{\delta_{ij}}e_i$ for the standard basis $\{e_1, \cdots, e_n\}$, i.e. $\varepsilon_j$ is the diagonal matrix with $1$'s on the diagonal except a $-1$ on the $j$-th entry. Note that 
		\[
			y=\Ad(h)^{-1}(J^{-1}\tau)={\dot s}^{-1}J^{-1}\tau
		\]
		and $s(j) = j$ implies $\dot s = (\dot s)\inv$ commutes with $\varepsilon_j = \varepsilon_j\inv = {}^t\varepsilon_j$. Hence we have 
		\[
			\Ad(\varepsilon_j) y
			= \varepsilon_j{\dot s}^{-1} {}^t\varepsilon_j J^{-1}\tau
			= y
		\]
		and $\varepsilon_j$ indeed centralizes $(y, \D P(\lambda^{\mathbb{R}}))$. By construction $\varepsilon_j$ lies in the block orthogonal subgroup of $\GL_n(\BC)$ supported on the rows and columns indexed by the $i$'s with $s(i) = i$ and $\lambda_i=\lambda_j$. In view of the description of centralizer
		\[
		S_{\phi^{\mathbb{R}}} \cong \prod_{i \in I^{+}} O(l_i, \mathbb{C}) \times \prod_{j \in J} GL(l_j)
		\] 
		given in \textsection \ref{subsec: main}, $\varepsilon_j$ lies in the factor $O(l_i, \mathbb{C})$ indexed by the element $i \in I^+$ corresponding to $\lambda_j$, and because of the $-1$ entry in $\varepsilon_j$ it descends to the nontrivial element of the component group of $O(l_i, \mathbb{C})$.
		
		We now find the corresponding element in $A^{+}_{\phi^{\mathbb{Q}_p}}$ (resp. $A_{\phi^{\mathbb{Q}_p}}$) for $H$ symplectic (resp. orthogonal) by tracing through the geometric construction. Through \eqref{eq: geometric correspondence 2} (resp. \eqref{eq: geometric correspondence 1}), $\phi^{\mathbb{R}}$ corresponds to the $H_{\lambda^{\mathbb{Q}_p}}$-orbit of certain $x \in V^{reg}_{\lambda^{\mathbb{Q}_p}} \cap \varphi^{-1}(x^\diamond)$.  In case $\lambda^{\mathbb{Q}_p}(\Fr)$ is supported in $\BZ$, as discussed above we have 
		$x= (x^<, x_{-1}, x_0, x^>)\in V_{\lambda^{\mathbb{Q}_p}}$, where $x^\diamond = (x^<, x^>)$ is the element fixed in \ref{cls: basis of W}, and $x_{-1}: W_{-1} \to W_0$ corresponds, under the isomorphisms $W_{-1} \cong \BC^n \cong W_0$ induced by the bases chosen in \ref{cls: basis of W}, to the map $h: \BC^n \to \BC^n$. Moreover, we have 
		$x_0=-{}^t h$ and 
		\[
		-\dot s =-{}^t h h=x_0x_{-1}.
		\]
		Identifying $P(\lambda^{\mathbb{Q}_p}) = \widehat{P}(\lambda^{\mathbb{R}})$ (cf. Corollary \ref{cor: geom comparison}), we have 
		\[
		H^{+}_{\lambda^{\mathbb{Q}_p}}(x^\diamond )\cong \widehat{P}(\lambda^{\mathbb{R}}) \times \operatorname{Isom}(W_0) .
		\]
		We claim that under the isomorphism
		\begin{multline}\label{eq: ABV vs slice again}
			[\D G \backslash (\D G \cdot y \times \D G/ \D P(\lambda^{\mathbb{R}}))] \cong [O(n,\BC) \backslash \GL_n(\BC)/ \D P(\lambda^{\mathbb{R}})] \\
			\cong [(\D P(\lambda^{\mathbb{R}}) \times \operatorname{Isom}(W_0)) \backslash \varphi\inv(x^\diamond)] = [H^{+}_{\lambda^{\mathbb{Q}_p}}(x^\diamond) \backslash \varphi\inv(x^\diamond)]
		\end{multline}
		the element $\varepsilon_j$ stabilizing $(y,\D P(\lambda^{\mathbb{R}}))$ corresponds to the element $\tilde{\varepsilon}_j=(x_{-1} \circ \varepsilon_j \circ x^{-1}_{-1}, \varepsilon_j) \in \operatorname{Isom}(W_0) \times \widehat{P}(\lambda^{\mathbb{R}})$ stabilizing $x$. To see this, first observe that under the first map in (\ref{eq: ABV vs slice again}),
		\begin{equation*}
			(y, \D P(\lambda^{\mathbb{R}}))
			\mapsto h  
		\end{equation*}
        where $h \in \D{G}$ satisfies that $\Ad(h) (y) = J\inv \tau$. The corresponding maps on the stabilizers are
		\begin{equation*}
			Z_{\D G}(y, \D P(\lambda^{\mathbb{R}})) \xrightarrow[\cong]{\Ad h} Z_{\D G}(J\inv \tau, h \D P(\lambda^{\mathbb{R}}))
			= Z_{O(n,\BC)}( h \D P(\lambda^{\mathbb{R}}))
			\xrightarrow[\cong]{k \mapsto (k, h\inv k h)} Z_{O(n,\BC) \times \D P(\lambda^{\mathbb{R}})}(h)
		\end{equation*}
		under which 
		\begin{equation*}
			\varepsilon_j \mapsto h \varepsilon_j h\inv \mapsto (h \varepsilon_j h\inv, \varepsilon_j) = (x_{-1} \circ \varepsilon_j \circ x_{-1}^{-1}, \varepsilon_j) = \tilde \varepsilon_j.
		\end{equation*}
		This proves the claim.		
		
		Finally, we check that $\tilde \varepsilon_j$ lies in the component $\BZ/2\BZ \subset A^{+}_{\phi^{\mathbb{Q}_p}}$ (resp. $A_{\phi^{\mathbb{Q}_p}}$) indexed by the irreducible factor $1_{W_{\mathbb{Q}_p}} \boxtimes S_{2\lambda_j + \delta}$ in ${\rm std}_{\D{H}} \circ \phi^{\BQ_p}$. To do so, we decompose $x$ as a direct sum of quiver subrepresentations corresponding to the multisegment $\bm$ in (\ref{eqn: assign-multisegment}). Then $\tilde{\varepsilon}_j$ acts nontrivially only on the irreducible quiver subrepresentation indexed by the segment $[-\tilde{\lambda}_j, \tilde{\lambda}_j]$ which corresponds to the factor $1_{W_{\mathbb{Q}_p}} \boxtimes S_{2\lambda_j + \delta}$. Therefore the isomorphism $A_{\phi^\BR} \cong A^{+}_{\phi^{\BQ_p}}$ (resp. $A_{\phi^{\mathbb{Q}_p}}$) induced by geometry agrees with the one in \textsection \ref{subsec: main}, and the proof is complete in case when $\lambda^{\mathbb{Q}_p}(\Fr)$ is supported in $\BZ$.
		
		
		We continue to consider the case when $\lambda^{\mathbb{Q}_p}(\Fr)$ is supported in $\frac{1}{2}\mathbb{Z}/\mathbb{Z}$.
		In this case
		\[
			H_{\lambda^{\mathbb{Q}_p}}(x^\diamond ) \cong \D P(\lambda^{\mathbb{R}}).
		\]
	    Here $\iota_{geom}(\dot s)$ is the $H_{\lambda^{\mathbb{Q}_p}}$-orbit of $x = (x^<, x_{-\frac12}, x^>)$ where $x_{-\frac12}: W_{-\frac12} \to W_{\frac12}$ corresponds to $\dot s$ under the identification
		\begin{equation*} 
			\{f: W_{-\frac12} \to W_{\frac12} \mid f^* = -f\} \bijects
			\operatorname{Sym}(W_{-\frac12}) \bijects
			\operatorname{Sym}_n(\BC).
		\end{equation*}
		We argued in the proof of $\iota = \iota_{geom}$ that the matrix of $x_{-\frac12}$ with respect to the bases of $W_{-\frac12}$ and $W_{\frac12}$ chosen is equal to $\dot s$.
		Now the element $\varepsilon_j\in \D P(\lambda^{\mathbb{R}})=P(\lambda^{\mathbb{Q}_p})$ gives rise to 
		an element which stabilizes $x$. As in the previous case,
		$\varepsilon_j$ induces the generator in $A^{+}_{\phi^{\mathbb{Q}_p}}$ (resp. $A_{\phi^{\mathbb{Q}_p}}$) indexed by $1_{W_{\mathbb{Q}_p}} \boxtimes S_{2\lambda_j + \delta}$. This finishes the proof.

	\end{proof}	
\end{clause}

\begin{clause}[The assignment $\psi^\BR \mapsto \psi^{\mathbb{Q}_p}$]\label{cls: psiR to psip}
	We now turn to the assignment of A-parameters. Let $\psi^\BR$ be an A-parameter for $G$ of infinitesmial character $\lambda^{\mathbb{R}}$. Its base change to $\BC$ can be written in the form
	\begin{equation*}
		BC(\psi^\BR) = \bigoplus_{i=1}^r (z/\bar z)^{\frac{k_i}2} \boxtimes S_{m_i}
	\end{equation*}
	where $k_i \in \BZ$, $m_i \in \BZ_{\ge 1}$, $S_{m_i}$ is the irreducible $m_i$-dimensional representation of $\SL_2(\BC)$. As multisets 
	\begin{equation*}
		\{\lambda_1,\ldots,\lambda_n\} = \bigsqcup_{i=1}^r \{ \frac12(k_i+m_i-1), \frac12 (k_i + m_i -3), \ldots, \frac12 (k_i + 1 - m_i) \}.
	\end{equation*}
    Since $\lambda_n > (1 -\delta)/2$ by our assumption, then
    $k_i + \delta > m_i$ for each $i$. 
	
	To $\psi^\BR$ we attach the A-parameter $\psi^{\mathbb{Q}_p}$ of $H$ such that
	\begin{equation*}
		{\rm std}_{\D{H}} \circ \psi^{\mathbb{Q}_p} = \bigoplus_{i = 1}^{r} 1_{W_{\mathbb{Q}_p}} \boxtimes S_{a_i} \boxtimes S_{b_i}
	\end{equation*}
	where $b_i = m_i$ and $a_i = k_i + \delta$. It is straightforward to check that the infinitesimal character $\lambda^{\mathbb{Q}_p}$ of $\psi^{\mathbb{Q}_p}$ is the one attached to $\lambda^{\mathbb{R}}$ in \ref{cls: lambda to lambdap}. Moreover, it is easy to see that
    \(
    \iota(\phi_{\psi^{\mathbb{R}}}) = \phi_{\psi^{\mathbb{Q}_p}}.
    \)

\end{clause}
\section{Proof of Theorem~\ref{thm: comparison of A-packets} and \ref{thm: comparison of endoscopic maps}}\label{sec: reduction}

From now on, we work under the setup of \textsection \ref{subsec: explicit correspondence} in the good parity case. In this section we prove our main theorems \ref{thm: comparison of A-packets} and \ref{thm: comparison of endoscopic maps} in the regular case, i.e., $\lambda^{\mathbb{R}}$ is regular. We then explain how we may reduce the singular case to the regular case, modulo a comparison result which will be proven in the next section.

\subsection{Regular case}\label{subsec: regular-case}

In this subsection. we verify Theorem~\ref{thm: comparison of A-packets} and Theorem~\ref{thm: comparison of endoscopic maps} in the case that $\lambda^{\mathbb{R}}$ is regular, i.e., $\lambda_1 > \lambda_2 > \cdots > \lambda_n$.

\begin{clause}[Adams-Johnson packet over $\BR$]\label{cls: A-J packet}
	On the real side, the A-packet $\Pi_{\psi^{\mathbb{R}}}(U(p,q))$ agrees with the packet constructed by Adams-Johnson (see \cite{AMR:2018}). We first recall this story, and then give an alternative parameterization of the packet in a way that matches the parameterization on the $p$-adic side. 

    Recall 
\[
    BC(\psi^{\mathbb{R}}) = \bigoplus_{i = 1}^{r} \, (z/\bar{z})^{\frac{k_i}{2}} \boxtimes S_{m_i}, \quad z \in \mathbb{C}^{\times}
\]
where $k_i \in \mathbb{Z}, m_i \in \mathbb{Z}_{> 0}$ such that $n \equiv k_i + m_i$ mod $2$. Let $A_i = \frac{m_i - 1}{2} + \frac{k_i}{2}$ and $B_i = - \frac{m_i - 1}{2} + \frac{k_i}{2}$, then the $[B_i, A_i]$'s are disjoint in this case. Let us arrange the subscripts so that $A_i > A_{i+1}$. The Arthur packet $\Pi_{\psi^{\mathbb{R}}}(U(p,q))$ is cohomological and can be constructed as follows. Let 
	\[
	G^{an} := U(n, 0) = \{g \in GL_n(\mathbb{C}) \, | \, {}^t\bar{g}g = I_n\},
	\]
	and $T$ be the subgroup of diagonal matrices and
	\[
	L := U(m_1, 0) \times \cdots \times U(m_r, 0)
	\]
	be the subgroup of block matrices. Define its inner form $G' := U(p,q)$ by twisting the rational structure of $G^{an}$ by $I_{p,q}$. Then $G'_{\mathbb{C}} = G^{an}_{\mathbb{C}}$ and $T$ is a subgroup of $G'$ under this identification. Let $K' = U(p) \times U(q)$ in $G'$. Let $Q_{\mathbb{C}}$ be the standard complex parabolic subgroup of $G'_{\mathbb{C}}$ with Levi factor $L_{\mathbb{C}}$ and denote its Lie algebra by $\mathfrak{q}$. Consider the A-parameter of $L$
	\[
	\psi^{\mathbb{R}}_{L}: W_{\mathbb{R}} \times SL_2(\mathbb{C}) \rightarrow \L{L},
	\]
	such that $SL_2(\mathbb{C})$ is sent to the principle $SL_2(\mathbb{C})$ of $\D{L}$ and 
	\[
	\psi^{\mathbb{R}}_{L}(z) = \psi^{\mathbb{R}}(z) (\frac{z}{\bar{z}})^{-\delta_{\D{G}} + \delta_{\D{L}}}, \quad z \in \mathbb{C}^{\times}, 
	\]
    where $\delta_{\D{G}}$ (resp. $\delta_{\D{L}}$) denotes the half sum of positive coroots in the $\D{G}$ (resp. $\D{L}$). The associated A-packet consists of a single character 
	\[
	\Lambda_{L, \psi^{\mathbb{R}}}: (t_1, \cdots, t_r) \mapsto {\rm det}(t_1)^{a_1} \cdots {\rm det}(t_r)^{a_r},
	\]
	where $t_i \in U(m_i, 0)$ and
	\[
	a_i = \frac{k_i + m_i - n}{2} + \sum_{j < i} m_j \in \mathbb{Z}.
	\]
	Here $a_1 \geqslant a_2 \geqslant \cdots \geqslant a_r$. For any $w \in W(G'_\mathbb{C}, T_\mathbb{C})$, let $L_w$ be the subgroup of $G'$ with $L_{w}(\mathbb{C}) = w(L(\mathbb{C}))$. We translate the character $\Lambda_{L, \psi^{\mathbb{R}}}$ to $L_w$, denoted by $w \Lambda_{L, \psi^{\mathbb{R}}}$. Define
	\[
	\pi(\psi^{\mathbb{R}}, w) := \mathcal{R}_{w(\mathfrak{q})}^{\frac{1}{2} {\rm dim}(K' \cap L_w \backslash K')} (w \Lambda_{L, \q^{\mathbb{R}}}).
	\]
	This is a cohomological induction in the good range, also denoted by $A_{w(\mathfrak{q})}(w \Lambda_{L, \q^{\mathbb{R}}})$ in the literature. Hence, $\pi(\psi^{\mathbb{R}}, w)$ is an irreducible representation of $G'$.
	
	\begin{theorem}[Adams-Johnson \cite{AJ:1987}, Arancibia-M{\oe}glin-Renard \cite{AMR:2018}]
		There is a bijection
		\begin{align*}
			W(G', T) \backslash W(G'_\mathbb{C}, T_\mathbb{C}) / W(L_\mathbb{C}, T_\mathbb{C}) \xrightarrow{\simeq} \Pi_{\psi^{\mathbb{R}}}(G') \\
			w \mapsto \pi(\psi^{\mathbb{R}}, w).
		\end{align*}
	\end{theorem}
	
	We now give a more combinatorial description of the double quotient in this theorem. Let us identify $W(G'_\mathbb{C}, T_\mathbb{C})$ with the symmetric group $S_{n}$. For any $w \in W(G'_\mathbb{C}, T_\mathbb{C})$, let 
	\[
	p_i := \sharp \Big\{\sum_{ j < i} m_j < a \leqslant \sum_{j \leqslant i} m_j  \, | \, w(a) \leqslant p \Big\}, \quad q_i := m_i - p_i
	\]
	for $1 \leqslant i \leqslant r$. Let $\underline{p} = (p_1, \cdots, p_r)$ and $\underline{q} = (q_1, \cdots, q_r)$. This induces a bijection 
	\[
	W(G', T) \backslash W(G'_\mathbb{C}, T_\mathbb{C}) / W(L_\mathbb{C}, T_\mathbb{C}) 
    \xrightarrow{\simeq} 
    \Big\{(\underline{p}, \underline{q}) \in \mathbb{Z}_{\geqslant 0 }^{r} \times \mathbb{Z}_{\geqslant 0 }^{r} \, | \, p_i + q_i = m_i, \sum_i p_i = p, \sum_i q_i = q \Big\}.
	\]
	Note if $w$ is sent to $(\underline{p}, \underline{q})$, 
	\[
	L_w \cong U(p_1, q_1) \times \cdots \times U(p_r, q_r).
	\]
	By taking disjoint union over $(p,q)$ such that $p + q = n$, we get a bijection
	\[
	\Big\{(\underline{p}, \underline{q}) \in \mathbb{Z}_{\geqslant 0 }^{r} \times \mathbb{Z}_{\geqslant 0 }^{r} \, | \, p_i + q_i = m_i \Big\} \xrightarrow{\simeq} \Pi_{\psi^{\mathbb{R}}}^{\rm pure}(G).
	\]
	
	In order to compare with the $p$-adic side, we shall give a different parametrization of $\Pi_{\psi^{\mathbb{R}}}^{\rm pure}(G)$. For each pair $(\underline{p}, \underline{q})$, let $l_i = {\rm min} \{p_i, q_i\}$ and $\eta_i = (-1)^{1+ \sum_{j> i} m_j} \operatorname{sgn}(p_i - q_i)$ if $p_i \neq q_i$ and is arbitrary in $\{\pm 1\}$ otherwise. Let $\ul = (l_1,\ldots, l_r)$ and $\ueta = (\eta_1,\ldots,\eta_r)$. Then the assignment $(\underline{p},\underline{q}) \mapsto (\ul, \ueta)$ defines a bijection
	\[
	\Big\{(\underline{p}, \underline{q}) \, | \, p_i + q_i = m_i \Big\} \xrightarrow{\simeq} \Big\{ (\ul, \ueta) \in \mathbb{Z}_{\geqslant 0}^{r} \times \{\pm 1\}^r \, | \, l_i \in [0, \lfloor m_i/2 \rfloor ] \Big\} / \sim,
	\]
	where the equivalence relation is given by $(\ul, \ueta) \sim (\ul', \ueta')$ if $l_i = l'_i$ for all $i$ and $\eta_i = \eta'_i$ unless $l_i = m_i/2$. We denote by $\pi(\q^{\mathbb{R}}; \underline{p}, \underline{q})$ (resp. $\pi(\psi^{\mathbb{R}}; \ul, \ueta)$) the representation in $\Pi_{\psi^{\mathbb{R}}}^{\rm pure}(G)$ corresponding to $(\underline{p}, \underline{q})$ (resp. $(\ul, \ueta)$). 
	
	For the $\psi^\BR$'s under consideration, the endoscopy map \eqref{eq: endoscopy real} lands in the subset of irreducible representations
	\[
	\Pi_{\psi^{\mathbb{R}}}^{\rm pure}(G) \rightarrow \D{A}_{\psi^{\mathbb{R}}}, \quad \pi \mapsto \epsilon_{\pi}.
	\]
	We can identify both $A_{\psi^{\mathbb{R}}}$ and $\D{A}_{\psi^{\mathbb{R}}}$ with $(\mathbb{Z}/2\mathbb{Z})^{Jord(\psi^{\mathbb{R}})}$, where $Jord(\psi^{\mathbb{R}})$ is the set of irreducible components of $BC(\psi^{\mathbb{R}})$. Suppose $\pi = \pi(\q^{\mathbb{R}}; \underline{p}, \underline{q}) = \pi(\psi^{\mathbb{R}}; \ul, \ueta)$. Then $\epsilon_{\pi} = (\epsilon_i)_{i=1}^{r}$, where
	\begin{align}
		\label{eq: character}
		\epsilon_{i} = (-1)^{m_i(1 + \sum_{j > i} m_j) + m_i(m_i - 1)/2 + q_i} = \eta^{m_i}_i (-1)^{(m_i - 2l_i)(m_i - 1 - 2l_i)/2}.
	\end{align}
	The restriction of $\epsilon_{\pi}$ to $Z(\D{G})^{\operatorname{Gal}(\BC/\BR)} = \{\pm I_n\}$ is determined by 
	\[
	\epsilon_{\pi}(-I_n) = \prod_{i = 1}^{r} \epsilon_i = \begin{cases} (-1)^{\frac{p-q}{2}} \quad & n \text{ even, } \\
		(-1)^{\frac{p - q + 1}{2}} \quad & n \text{ odd. }
	\end{cases}
	\]
	For $(\ul, \ueta)$, we define $\epsilon_{\ul, \ueta} := (\epsilon_i)_{i=1}^{r}$ by \eqref{eq: character}.

	Finally, one may calculate the complete Langlands parameter of $\pi(\psi^{\mathbb{R}}; \ul, \ueta)$. This is known to experts, and we present a detailed calculation in Appendix \ref{sec: Langlands}.
	
	\begin{theorem}[Vogan-Zuckerman \cite{VZ:1984}, Knapp-Vogan \cite{KV:1995}]
		\label{Langlands-sub}
        Let $(\phi^{\mathbb{R}}, \epsilon)$ be the complete Langlands parameter of $\pi(\psi^{\mathbb{R}}; \ul, \ueta)$. Then
        \[
        BC(\phi^{\mathbb{R}}) = \bigoplus_{i = 1}^{r} \bigoplus_{j = 0 }^{l_i-1} \Big((\frac{z}{\bar{z}})^{\frac{k_i}{2}} (z\bar{z})^{\frac{m_i - 1}{2} - j} \oplus(\frac{z}{\bar{z}})^{\frac{k_i}{2}} (z\bar{z})^{-\frac{m_i - 1}{2} + j}\Big) \bigoplus BC(\phi^{\mathbb{R}}_{-})         
        \]
        where 
		\[
		BC(\phi^{\mathbb{R}}_{-}) = \bigoplus_{i = 1}^{r} \bigoplus_{j = 0}^{A_i - B_i - 2l_i} (\frac{z}{\bar{z}})^{B_i + l_i + j},
		\]
        and $\epsilon$ corresponds to $\epsilon_{-}$ under the bijection
        \(
        \D A_{\phi^{\mathbb{R}}} \rightarrow \D A_{\phi^{\mathbb{R}}_{-}},
        \)
		where $\epsilon_{-}$ alternates over $(\frac{z}{\bar{z}})^{B_i + l_i + j}$ for $j \in [0, A_i - B_i - 2l_i]$ and takes value $\eta_i$ at $j = 0$.
	\end{theorem}
\end{clause}

We now turn to the $p$-adic side. Note in the even orthogonal case,  $H$ denotes the full even orthogonal group. Recall
\[
{\rm std}_{\D{H}} \circ \psi^{\mathbb{Q}_p} = \bigoplus_{i = 1}^{r} 1_{W_{\mathbb{Q}_p}} \boxtimes S_{a_i} \boxtimes S_{b_i}
\]
where $a_i = k_i + \delta$ and $b_i = m_i$. Our assumption on $\delta$ implies that $a_i > b_i$. Let $\widetilde{A}_i := (a_i + b_i)/2 - 1 = A_i + (\delta - 1)/2$ and $\widetilde{B}_i := (a_i - b_i)/2 = B_i + (\delta - 1)/2$. Since the $[B_i, A_i]$'s are disjoint and so are $[\widetilde{B}_i, \widetilde{A}_i]$, then $\psi^{\mathbb{Q}_p}$ has \textit{discrete diagonal restriction} in the sense of Moeglin \cite{Moeglin:2009}.

\begin{theorem}[Moeglin \cite{Moeglin:2009}]\,
\begin{enumerate}
\item
If $H$ is orthogonal, there is a bijection
\begin{align}
\label{eq: DDR even}
\Pi_{\psi^{\mathbb{Q}_p}}^{\rm pure}(H) \rightarrow \Big\{ (\ul, \ueta) \in \mathbb{Z}_{\geqslant 0}^{r} \times \{\pm 1\}^r \, | \, l_i \in [0, \lfloor m_i/2\rfloor ] \Big\} / \sim.
\end{align}
If $H$ is symplectic, there is a bijection
\begin{align}
\label{eq: DDR odd}
\Pi_{\psi^{\mathbb{Q}_p}}^{\rm pure}(H) \rightarrow \Big\{ (\ul, \ueta) \in \mathbb{Z}_{\geqslant 0}^{r} \times \{\pm 1\}^r \, | \, l_i \in [0, \lfloor m_i/2 \rfloor], \prod_{i} \epsilon_i = 1 \Big\} / \sim.
\end{align}
\item 
Let $\pi(\psi^{\mathbb{Q}_p}; \ul, \ueta)$ be the representation corresponding to $(\ul, \ueta)$, and $(\phi^{\mathbb{Q}_p}, \epsilon)$ be its complete Langlands parameter. Then
\[
{\rm std}_{\D{H}} \circ \phi^{\mathbb{Q}_p} = \bigoplus_{i = 1}^{r} \bigoplus_{j = 0}^{l_i - 1} \Big(|\cdot|_{W_{\mathbb{Q}_p}}^{\frac{b_i - 1}{2} - j} \boxtimes S_{a_i} \oplus |\cdot|_{W_{\mathbb{Q}_p}}^{-\frac{b_i - 1}{2} + j} \boxtimes S_{a_i}\Big) \bigoplus ({\rm std}_{\D{H}_{-}} \circ \phi^{\mathbb{Q}_p}_{-})
\]
where $H_{-}$ is of the same type as $H$ and
\[
{\rm std}_{\D{H}_{-}} \circ \phi^{\mathbb{Q}_p}_{-} = \bigoplus_{i = 1}^{r} \bigoplus_{j = 0}^{A_i - B_i - 2l_i} 1_{W_{\mathbb{Q}_p}} \boxtimes S_{2(B_i + l_i + j) + \delta} 
\]
and $\epsilon$ corresponds to $\epsilon_{-}$ under the bijection
        \(
        \D A_{\phi^{\mathbb{Q}_p}} \rightarrow \D A_{\phi^{\mathbb{Q}_p}_{-}},
        \)
where $\epsilon_{-}$ alternates over $1_{W_{\mathbb{Q}_p}} \boxtimes S_{2(B_i + l_i + j) + \delta}$ for $j \in [0, A_i - B_i - 2l_i]$ and takes value $\eta_i$ at $j = 0$.
\item
The endoscopic map has image in the subset of irreducible representations
\[
\Pi_{\psi^{\mathbb{Q}_p}}^{\rm pure}(H) \rightarrow \D{A}_{\psi^{\mathbb{Q}_p}}, \quad \pi \mapsto \epsilon_{\pi}
\]
and $\epsilon_{\pi} = \epsilon_{\ul, \ueta}$.

\end{enumerate}
\end{theorem}

Comparing with the parametrization of $\Pi_{\psi^{\mathbb{R}}}^{\rm pure}(G)$ in \ref{cls: A-J packet}, we immediately obtain:

\begin{corollary}
If $H$ is orthogonal, $\tilde{\iota}(\pi(\psi^{\mathbb{R}}; \ul, \ueta)) = \pi(\psi^{\mathbb{Q}_p}; \ul, \ueta)$. If $H$ is symplectic, then $\tilde{\iota}(\pi(\psi^{\mathbb{R}}; \ul, \ueta)) = \pi(\psi^{\mathbb{Q}_p}; \ul, \ueta')$, where $\ueta' = \pm \ueta$ such that $(\ul, \ueta')$ appear in \eqref{eq: DDR odd}.
\end{corollary}

This proves Theorem~\ref{thm: comparison of A-packets} and Theorem~\ref{thm: comparison of endoscopic maps} in the regular case.

\subsection{Singular case}
\label{subsec: singular-case}

In this subsection, we detail the strategy for proving Theorem \ref{thm: comparison of A-packets} and Theorem~\ref{thm: comparison of endoscopic maps} in the case of singular infinitesimal characters.

Let us index the irreducible components of $BC(\psi^{\mathbb{R}})$ such that if $A_i > A_j$ and $B_i > B_j$ then $i < j$. Then we can define $\pi(\psi^{\mathbb{R}}; \ul, \ueta)$ for any $(\ul, \ueta)$ by cohomological induction as in the regular case, and Trapa \cite{Trapa:2001} proved that $\pi(\psi^{\mathbb{R}}; \ul, \ueta)$ is irreducible or zero. In order to compare with the $p$-adic side, we shall describe these representations as images of translation functors applied to representations in the Adams-Johnson packets, following the same idea as in \cite{Trapa:2001} and \cite{MR:2019}. 

Let us choose $t_i \in \mathbb{Z}_{\geqslant 0}$ such that the shifted intervals $[B_i + t_i, A_i + t_i]$ are disjoint and $A_{i} + t_{i} > A_{i+1} + t_{i+1}$. Let $\psi^{\mathbb{R}}_{\gg} \in \Psi(G)$ be defined by
\[
BC(\psi^{\mathbb{R}}_{\gg}) = \bigoplus_{i = 1}^{r} \, (z/\bar{z})^{\frac{k_i}{2} + t_i} \boxtimes S_{m_i}.
\]
Then the infinitesimal character $\lambda^{\mathbb{R}}_{\gg}$ of $\psi_\gg^{\mathbb{R}}$ is regular. Let $\psi^{\mathbb{Q}_p}_{\gg} \in \Psi(H_{\gg})$ be the Arthur parameter associated with $\psi^{\mathbb{R}}_{\gg}$ as in \ref{cls: psiR to psip}. It follows from the previous subsection that Theorem \ref{thm: comparison of A-packets} and Theorem~\ref{thm: comparison of endoscopic maps} hold for the A-parameters $\psi^{\mathbb{R}}_{\gg}$ and $\psi_{\gg}^{\mathbb{Q}_p}$. To relate $\Pi_{\psi_{\gg}^{\mathbb{R}}}^{\rm pure}(G)$ with $\Pi_{\psi^{\mathbb{R}}}^{\rm pure}(G)$, we shall apply the translation functors. For any $\lambda \in \mathbb{C}^{n}$, we define the translation functor
\[
\mathcal{T}_{i} := \psi^{\lambda - e_i}_{\lambda}
\]
where $e_1, \cdots, e_n$ is the standard basis of $\mathbb{C}^{n}$. We have suppressed $\lambda$ in this notation, which should be determined by the context.

\begin{proposition}
\label{prop: singular by translation}
We have
\begin{multline*}
    \pi(\psi^{\mathbb{R}}; \ul, \ueta) 
    = \circ_{i = 1}^{r} \, (\mathcal{T}_{1 + \sum_{j < i} m_j} \circ \cdots \circ \mathcal{T}_{\sum_{j \leqslant i} m_j})^{t_i} \, \pi(\psi^{\mathbb{R}}_{\gg}; \ul, \ueta)\\
    = (\mathcal{T}_1 \circ \mathcal{T}_2 \circ \cdots \circ  \mathcal{T}_{m_1})^{t_1} \circ (\mathcal{T}_{m_1+1} \circ \mathcal{T}_{m_1+2} \circ \cdots \circ \mathcal{T}_{m_1+m_2})^{t_2} \circ \cdots \\
    \cdots \circ (\mathcal{T}_{1+\sum_{j<r} m_j} \circ \mathcal{T}_{2+\sum_{j<r} m_j} \circ \cdots \circ \mathcal{T}_{\sum_{j \le r} m_j})^{t_r} \, \pi(\psi^{\mathbb{R}}_{\gg}; \ul, \ueta).
\end{multline*}
\end{proposition}

\begin{proof}
This follows from Prop~\ref{prop: Trapa translation} and Lemma~\ref{lemma: comparison of translation}.
\end{proof}

There is a canonical surjective homomorphism 
\[
A_{\psi^{\mathbb{R}}_{\gg}} \rightarrow A_{\psi^{\mathbb{R}}}
\]
which sends the generator of $A_{\psi^{\mathbb{R}}_{\gg}}$ indexed by $i$ to the generator of $A_{\psi^{\mathbb{R}}}$ indexed by the irreducible component $(z/\bar{z})^{\frac{k_i}{2}} \boxtimes S_{m_i}$ in $BC(\psi^{\mathbb{R}})$ (cf. \eqref{eq: real centralizer}). This induces an inclusion
\[
\D{A}_{\psi^{\mathbb{R}}} \hookrightarrow \D{A}_{\psi_{\gg}^{\mathbb{R}}}.
\]
For any $(\ul, \ueta)$, let $\epsilon_{\ul, \ueta}$ define an element in $\D{A}_{\psi_{\gg}^{\mathbb{R}}}$.

\begin{theorem}[Moeglin-Renard \cite{MR:2019}]
\label{thm: singular real A-packet}\,
\begin{enumerate}

\item $\pi(\psi^{\mathbb{R}}; \ul, \ueta) = 0$ if $\epsilon_{\ul, \ueta} \notin \D{A}_{\psi^{\mathbb{R}}}$.

\item 
\(
\Pi_{\psi^{\mathbb{R}}}^{\rm pure}(G)
\)
consists of irreducible constituents of $\{ \pi(\psi^{\mathbb{R}}; \ul, \ueta) \}_{(\ul, \ueta)}$. 

\item
For any $\pi \in \Pi_{\psi^{\mathbb{R}}}^{\rm pure}(G)$, 
\[
\epsilon_{\pi} = \bigoplus_{(\ul, \ueta)}  {\rm mult}(\pi, \pi(\psi^{\mathbb{R}}; \ul, \ueta)) \, \epsilon_{\ul, \ueta}.
\]
Moreover, $\epsilon_{\pi}$ is irreducible.
\end{enumerate}
\end{theorem}

There is a similar construction in the $p$-adic setting. Note the rank of $H_{\gg}$ is bigger than that of $H$. To relate $\Pi_{\psi_{\gg}^{\mathbb{Q}_p}}^{\rm pure}$ with $\Pi_{\psi^{\mathbb{Q}_p}}^{\rm pure}(H)$, we shall apply the Jacquet functor. 

\begin{definition}
	Let $H'$ be a pure inner form of $H$ and $P'$ be a standard maximal parabolic subgroup of $H'$ with Levi factor $GL_1(\BQ_p) \times H'_{-}$, where $H'_{-}$ is of the same type as $H'$. For any finite length smooth representation $\pi$ of $H'$ and any $x \in \mathbb{R}$, we define 
	\(
		\mathcal{D}_{x} \pi
    \)
    to be the projection of the Jacquet module ${\rm Jac}_{P'} \pi$ to the $|\cdot |^{x}$-primary component for the action of $GL_1(\mathbb{Q}_p)$.	
\end{definition}

Define
\begin{multline*} 
\pi(\psi^{\mathbb{Q}_p}; \ul, \ueta) 
:= \circ_{i = 1}^{r} \Big((\mathcal{D}_{\widetilde{A}_i + 1} \circ \cdots \circ \mathcal{D}_{\widetilde{B}_i + 1}) \circ \cdots \circ (\mathcal{D}_{\widetilde{A}_i + t_i} \circ \cdots \circ \mathcal{D}_{\widetilde{B}_i + t_i})\Big) \pi(\psi^{\mathbb{Q}_p}_{\gg}; \ul, \ueta)\\
= \Big((\mathcal{D}_{\widetilde{A}_1 + 1} \circ \cdots \circ \mathcal{D}_{\widetilde{B}_1 + 1}) \circ \cdots \circ (\mathcal{D}_{\widetilde{A}_1 + t_1} \circ \cdots \circ \mathcal{D}_{\widetilde{B}_1 + t_1})\Big) \circ \cdots\\
\cdots \circ \Big((\mathcal{D}_{\widetilde{A}_r + 1} \circ \cdots \circ \mathcal{D}_{\widetilde{B}_r + 1}) \circ \cdots \circ (\mathcal{D}_{\widetilde{A}_r + t_r} \circ \cdots \circ \mathcal{D}_{\widetilde{B}_r + t_r})\Big) \pi(\psi^{\mathbb{Q}_p}_{\gg}; \ul, \ueta).
\end{multline*}
Similar to $\psi^\BR$, there is a canonical surjective homomorphism
\[
A_{\psi^{\mathbb{Q}_p}_{\gg}} \rightarrow A_{\psi^{\mathbb{Q}_p}}
\]
by matching the irreducible components as in the real case, which induces an inclusion
\[
\D{A}_{\psi^{\mathbb{Q}_p}} \hookrightarrow \D{A}_{\psi_{\gg}^{\mathbb{Q}_p}}.
\]

\begin{theorem}[Moeglin \cite{Moeglin:2009}] \label{thm: singular p-adic A-packet}\,
\begin{enumerate}

\item $\pi(\psi^{\mathbb{Q}_p}; \ul, \ueta) = 0$ if $\epsilon_{\ul, \ueta} \notin \D{A}_{\psi^{\mathbb{Q}_p}}$.

\item 
\(
\Pi_{\psi^{\mathbb{Q}_p}}^{\rm pure}(H)
\)
consists of irreducible constituents of $\{ \pi(\psi^{\mathbb{Q}_p}; \ul, \ueta) \}_{(\ul, \ueta)}$.

\item For any $\pi \in \Pi_{\psi^{\mathbb{Q}_p}}^{\rm pure}(H)$, 
\[
\epsilon_{\pi} = \bigoplus_{(\ul, \ueta)}  {\rm mult}(\pi, \pi(\psi^{\mathbb{R}}; \ul, \ueta)) \, \epsilon_{\ul, \ueta}.
\]
Moreover, $\epsilon_{\pi}$ is irreducible.

\end{enumerate}
\end{theorem}

Our goal is to show that

\begin{theorem}
\label{thm: comparison singular}
If $H$ is orthogonal, $\tilde{\iota}(\pi(\psi^{\mathbb{R}}; \ul, \ueta)) = \pi(\psi^{\mathbb{Q}_p}; \ul, \ueta)$. If $H$ is symplectic, then $\tilde{\iota}(\pi(\psi^{\mathbb{R}}; \ul, \ueta)) = \pi(\psi^{\mathbb{Q}_p}; \ul, \ueta')$, where $\ueta' = \pm \ueta$ such that $\epsilon_{\ul, \ueta'} \in \D{A}_{\psi^{\mathbb{Q}_p}}$. 
\end{theorem}

Theorem~\ref{thm: comparison singular} implies Theorem~\ref{thm: comparison of A-packets} and Theorem~\ref{thm: comparison of endoscopic maps} in the singular case. Moreover, it shows that the nonvanishing conditions of $\pi(\psi^{\mathbb{R}}; \ul, \ueta)$ and $\pi(\psi^{\mathbb{Q}_p}; \ul, \ueta)$ are equivalent. The former was given by P.Trapa \cite{Trapa:2001} and the latter was given by the third named author \cite{Xu:Comb} (also see \cite{Atobe:2022}). The combinatorics involved in these conditions are very complicated and also very different in their appearances. Before our work, the second named author  \cite{Hang:2025} has already proved that they are equivalent by a direct comparison. This was a hard task, and it provided us with strong confidence in our theorem at the beginning.

By Theorem \ref{thm: singular real A-packet} and \ref{thm: singular p-adic A-packet}, $\Pi_{\psi^{\mathbb{R}}}^{\rm pure}(G)$ (resp. $\Pi_{\psi^{\mathbb{Q}_p}}^{\rm pure}(H)$) is obtained from the regular case by applying translation functors (resp. Jacquet functors). Hence Theorem \ref{thm: comparison singular} follows if the map $\tilde \iota$ intertwines translation functors and Jacquet functors. By Proposition \ref{prop: iota vs iotageom} $\tilde \iota$ agrees with the map $\tilde \iota_{geom}$ defined geometrically, and the latter agrees with the map $\cR^*$ on the set of simple objects up to signs, see Corollary \ref{cor: geom comparison}. Therefore we are reduced to proving the following theorem, which also includes the bad parity case. It will be proved in the next section.

\begin{theorem}
	\label{thm: translation and derivative}
	Let $\lambda^{\mathbb{R}} = (\lambda_1, \cdots, \lambda_n)$ be a dominant integral infinitesimal character of $G$. Suppose there exists $1 \leqslant j \leqslant n$ such that $\lambda_{j} > \lambda_{j + 1}$ (set $\lambda_{n+1} = 1 + (1 - \delta)/2$). Then the following diagram commutes
	\begin{align*}
		\xymatrix{K\Pi_{{\rm pure}}(\lambda^{\mathbb{R}}, G) \ar[r]^{\mathcal{R}^{*}} \ar[d]_{\psi^{\lambda^{\mathbb{R}}_{-}}_{\lambda^{\mathbb{R}}}} & K\Pi_{{\rm pure}}(\lambda^{\BQ_p}, H) \ar[d]^{\mathcal{D}_{\tilde{\lambda}_j}} \\
			K\Pi_{{\rm pure}}(\lambda^{\mathbb{R}}_{-}, G) \ar[r]^{\mathcal{R}^{*}} & K\Pi_{{\rm pure}}(\lambda_{-}^{\BQ_p}, H_{-}),
		}
	\end{align*}
	where $\lambda^{\mathbb{R}}_{-} = \lambda^{\mathbb{R}} - e_j$, and $\lambda_-^{\BQ_p}$ is the image of $\lambda^{\mathbb{R}}_-$ under the assignment \ref{cls: lambda to lambdap}.
\end{theorem}
\section{Comparison of translation and derivative}
\label{sec: translation and derivative}

The purpose of this section is to prove Theorem~\ref{thm: translation and derivative}, and thus complete the proof of Theorem~\ref{thm: comparison of A-packets} and \ref{thm: comparison of endoscopic maps}. 

Theorem~\ref{thm: translation and derivative} is equivalent to the commutativity of the diagram
\begin{align*}
\xymatrix{K{\rm Per}(H_{\lambda_{-}^{\mathbb{Q}_p}} \backslash V_{\lambda_{-}^{\mathbb{Q}_p}}) \ar[r]^{\mathcal{R}} \ar[d]_{\mathcal{D}^{*}_{\tilde{\lambda}_j}} & K{\rm Per}(\D{K} \backslash X_{\lambda^{\mathbb{R}}_{-}}) \ar[d]^{(\psi^{\lambda^{\mathbb{R}}_{-}}_{\lambda^{\mathbb{R}}})^*} \\
K{\rm Per}(H_{\lambda^{\mathbb{Q}_p}} \backslash V_{\lambda^{\mathbb{Q}_p}}) \ar[r]^{\mathcal{R}} & K{\rm Per}(\D{K} \backslash X_{\lambda^{\mathbb{R}}})
}
\end{align*}
where $\cD_{\tilde{\lambda}_j}^*$ (resp. $(\psi_{\lambda^{\mathbb{R}}}^{\lambda^{\mathbb{R}}_-})^*$) is the adjoint operator to $\cD_{\tilde{\lambda}_j}$ (resp. $\psi_{\lambda^{\mathbb{R}}}^{\lambda^{\mathbb{R}}_-}$) under the pairing \ref{eq: p-adic pairing} (resp. \ref{eq: real pairing}). 

First, we would like to give a geometric construction of $(\psi_{\lambda^{\mathbb{R}}}^{\lambda^{\mathbb{R}}_-})^*$. Let $\D{P}(\lambda^{\mathbb{R}}, \lambda^{\mathbb{R}}_{-}) := \D{P}(\lambda^{\mathbb{R}}_{-}) \cap \D{P}(\lambda^{\mathbb{R}})$. There are natural projections
\[
X_{\lambda^{\mathbb{R}}_{-}} = \D{G}/\D{P}(\lambda^{\mathbb{R}}_{-}) \xleftarrow{p_1} \D{G}/ \D{P}(\lambda^{\mathbb{R}}, \lambda^{\mathbb{R}}_{-}) \xrightarrow{p_2} \D{G}/\D{P}(\lambda^{\mathbb{R}}) = X_{\lambda^{\mathbb{R}}}.
\]

\begin{theorem}\label{thm: translation vs pushpull}
We have
\[
(\psi_{\lambda^{\mathbb{R}}}^{\lambda^{\mathbb{R}}_-})^* = Rp_{2 *} \circ p^{*}_1
\]
as maps from $K{\rm Per}(\D{K} \backslash X_{\lambda^{\mathbb{R}}_{-}})$ to $K{\rm Per}(\D{K} \backslash X_{\lambda^{\mathbb{R}}})$.
\end{theorem}

\begin{proof}
	By Lemma \ref{lemma: second decomposition}, we have a decomposition $\psi_{\lambda^{\mathbb{R}}}^{\lambda^{\mathbb{R}}_-} = \psi_{\lambda'}^{\lambda^{\mathbb{R}}_-} \circ \psi_{\lambda^{\mathbb{R}}}^{\lambda'}$ as maps on the Grothendieck groups, where 
	\begin{equation*}
		\lambda' = \lambda^{\mathbb{R}} + (e_1 + \cdots + e_{j-1}).
	\end{equation*}
	It is straightforward to check that $\D P(\lambda') = \D P(\lambda^{\mathbb{R}},\lambda^{\mathbb{R}}_-)$. Hence it suffices to show that $(\psi_{\lambda^{\mathbb{R}}}^{\lambda'})^* = R p_{2*}$ and $(\psi_{\lambda'}^{\lambda^{\mathbb{R}}_{-}})^* = p_1^*$. This follows from \cite[Proposition 2.2.1]{DHXZ:GLn}. 
\end{proof}






Next we propose a geometric construction of the adjoint operator $\cD_{\tilde{\lambda}_j}^*$ of $\cD_{\tilde{\lambda}_j}$ using an induction procedure introduced by Lusztig \cite{Lusztig:1995}. 

Fix an element $x^\diamond \in V_{\lambda^{\BQ_p}}^{reg,\diamond}$. Recall the graded vector space $W = \BC^N$ equipped with a pairing $\langle -,- \rangle$ and the graded basis adapted to $x^\diamond$ introduced in \ref{cls: basis of W}. We decompose $W = W' \oplus W^- \oplus W''$ as $I=[-\tilde{\lambda}_1,\tilde{\lambda}_1]$-graded vector spaces, where $W'$ (resp. $W''$) is $1$-dimensional, concentrated at degree $-\tilde{\lambda}_j$ (resp. $\tilde{\lambda}_j$) and spanned by the vector $\alpha_j$ (resp. $\alpha_j^*$), and $W^-$ is spanned by the rest of the basis vectors. In other words,
\begin{gather*}
	W_k' = 
	\begin{cases}
		\BC \alpha_j \subset W_{-\tilde{\lambda}_j} & \text{if } k = -\tilde{\lambda}_j\\
		0 & \text{otherwise}
	\end{cases},\quad
	W_k'' = 
	\begin{cases}
		\BC \alpha_j^* \subset W_{\tilde{\lambda}_j} & \text{if } k = \tilde{\lambda}_j\\
		0 & \text{otherwise}
	\end{cases},
	\\
	W_k^- = 
	\begin{cases}
		\operatorname{span}\{\alpha_1,\ldots, \alpha_{j-1}\} \subset W_{-\tilde{\lambda}_j} & \text{if } k = -\tilde{\lambda}_j\\
		\operatorname{span}\{\alpha_1^*,\ldots, \alpha_{j-1}^*\} \subset W_{\tilde{\lambda}_j} & \text{if } k = \tilde{\lambda}_j\\
		W_k &\text{otherwise}
	\end{cases}
\end{gather*}
Then $\D H_-$ may be identified with the subgroup of $\GL(W^-)$ preserving the form $\langle -,- \rangle$. We set $\D P \subset \D H$ be the parabolic subgroup stabilizing the flag $0 \subset W'' \subset W'' \oplus W^- \subset W$, and let $\D M \subset \D P$ be the Levi that stabilizes the subspace $W^-$. Write $\D N$ for the radical of $\D P$, so that $\D P = \D M \D N$, and write $P = MN$ for the corresponding subgroup of $H$. Then we have natural isomorphisms
\begin{equation*}
	M \cong \GL_1 \times H_-,\quad
	\D M \cong \GL_1(\BC) \times \D H_-
\end{equation*}
through which we obtain the infinitesimal character 
\begin{equation*}
	\lambda_M^{\BQ_p} = (|\cdot|^{\tilde{\lambda}_j}, \lambda_-^{\BQ_p})
\end{equation*}
of $M$.

We may take the Vogan varieties for the groups $P$, $M$, and $H_-$, respectively, equipped with corresponding action of centralizers of infinitesimal characters. They can (and will) be identified as follows: $\D \fp$, $\D \fm$, $\D \fh_{-}$ are Lie algebras of $\D P$, $\D M$, $\D H_{-}$ respectively,
\begin{itemize}
	\item $V_{\lambda_M^{\BQ_p}}^{\D P} = V_{\lambda^{\BQ_p}} \cap \D \fp$
	
	\item $H_{\lambda_M^{\BQ_p}}^{\D P} = H_{\lambda^{\BQ_p}} \cap \D P$
	
	\item $V_{\lambda_M^{\BQ_p}} =$ image of $V_{\lambda_M^{\BQ_p}}^{\D P} \subset \D \fp$ under the quotient $\D \fp \surj \D \fm$
	
	\item $H_{\lambda_M^{\BQ_p}} =$ image of $H_{\lambda_M^{\BQ_p}}^{\D P} \subset \D P$ under the quotient $\D P \surj \D M$
	
	\item $V_{\lambda_-^{\BQ_p}} =$ image of $V_{\lambda_M^{\BQ_p}} \subset \D \fm$ under the quotient $\D \fm \cong \fgl_1 \times \D \fh_- \surj \D \fh_-$
	
	\item $H_{\lambda_-^{\BQ_p}} =$ image of $H_{\lambda_M^{\BQ_p}} \subset \D M$ under the quotient $\D M \cong \GL_1(\BC) \times \D H_- \surj \D H_-$.
\end{itemize}
Note that $H_{\lambda_M^{\BQ_p}}$ is the Levi factor of $H_{\lambda_M^{\BQ_p}}^{\D P}$. Hence we have an action $H_{\lambda_M^{\BQ_p}}^{\D P} \surj H_{\lambda_M^{\BQ_p}} \acts V_{\lambda_M^{\BQ_p}}$. Moreover, $V_{\lambda_M^{\BQ_p}}$ is isomorphic to $V_{\lambda_-^{\BQ_p}}$, since the forgotten $\fgl_1$-factor contains no nonzero nilpotent elements. These objects fit into the following commutative diagram
\begin{equation}\label{diag:L-ind}
	\begin{tikzcd}[row sep=small]
		\D H
		& \D P \ar[l, hook] \ar[r, two heads]
		& \D M \ar[r, two heads]
		& \D H_-
		\\
		H_{\lambda^{\BQ_p}} \ar[d, phantom, "\racts" {rot90, description}] \ar[u, hook]
		& H_{\lambda_M^{\BQ_p}}^{\D P} \ar[d, phantom, "\racts" {rot90, description}] \ar[l, hook]   \ar[u, hook] \ar[r, two heads] 
		& H_{\lambda_M^{\BQ_p}} \ar[d, phantom, "\racts" {rot90, description}] \ar[r, two heads] \ar[u, hook]
		& H_{\lambda_-^{\BQ_p}} \ar[d, phantom, "\racts" {rot90, description}] \ar[u, hook]
		\\
		V_{\lambda^{\BQ_p}} \ar[d, hook]
		& V_{\lambda_M^{\BQ_p}}^{\D P} \ar[l, hook, "q_2"'] \ar[r, two heads, "q_1"] \ar[d, hook] 
		& V_{\lambda_M^{\BQ_p}} \ar[d, hook] \ar[r, "u", "\cong"'] \ar[d, hook]
		& V_{\lambda_-^{\BQ_p}} \ar[d, hook]
		\\
		\D \fh
		& \D \fp \ar[l, hook] \ar[r, two heads]
		& \D \fm \ar[r, two heads]
		& \D \fh_-
	\end{tikzcd}
\end{equation}
and hence we obtain morphisms of quotient stacks
\begin{equation*}
	[H_{\lambda^{\BQ_p}} \backslash V_{\lambda^{\BQ_p}}]
	\xleftarrow{q_2} [H_{\lambda_M^{\BQ_p}}^{\D P} \backslash V_{\lambda_M^{\BQ_p}}^{\D P}]
	\xrightarrow{q_1} [H_{\lambda_M^{\BQ_p}} \backslash V_{\lambda_M^{\BQ_p}}]
	\xrightarrow{u} [H_{\lambda_-^{\BQ_p}} \backslash V_{\lambda_-^{\BQ_p}}].
\end{equation*}
The \textbf{Lusztig induction} is the functor
\begin{equation*}
	\Ind_{\D P}^{\D H} := Rq_{2*} \circ q_1^* \circ u^* : D^b(H_{\lambda_-^{\BQ_p}} \backslash V_{\lambda_-^{\BQ_p}}) \aro D^b(H_{\lambda^{\BQ_p}} \backslash V_{\lambda^{\BQ_p}})
\end{equation*}
Note that each step of the composition preverses semisimple complexes. Indeed, the pullback $u^*$ is an inflation functor along a connected normal subgroup; the map $q_1$ is smooth; the stack pushforward $Rq_{2*}$ is defined as the induction equivalence $[H_{\lambda_M^{\BQ_p}}^{\D P} \backslash V_{\lambda_M^{\BQ_p}}^{\D P}] \cong [H_{\lambda^{\BQ_p}} \backslash (H_{\lambda^{\BQ_p}} \times_{H_{\lambda_M^{\BQ_p}}^{\D P}} V_{\lambda_M^{\BQ_p}}^{\D P})]$ followed by (derived) sheaf theoretic pushforward along the proper map $H_{\lambda^{\BQ_p}} \times_{H_{\lambda_M^{\BQ_p}}^{\D P}} V_{\lambda_M^{\BQ_p}}^{\D P} \to V_{\lambda^{\BQ_p}}$, where the decomposition theorem applies. As a result, we may restrict $\Ind_{\D P}^{\D H}$ to semisimple complexes:
\begin{equation*}
	\Ind_{\D P}^{\D H} :
	\operatorname{SPer}(H_{\lambda_-^{\BQ_p}} \backslash V_{\lambda_-^{\BQ_p}}) \aro \operatorname{SPer}(H_{\lambda^{\BQ_p}} \backslash V_{\lambda^{\BQ_p}}).
\end{equation*}
The proof of the following conjecture is work in progress \cite{CDX:adjunction}.

\begin{conjecture}
\label{conj: adjoint of derivative}
	On the level of Grothendieck groups, we have $\cD_{\tilde{\lambda}_j}^* = \Ind_{\D P}^{\D H}$.
\end{conjecture}

Under this conjecture, Theorem \ref{thm: translation and derivative} is equivalent to the following statement.

\begin{theorem}
	Let $\lambda^{\mathbb{R}} = (\lambda_1, \cdots, \lambda_n)$ be a dominant integral infinitesimal character of $G$. Suppose there exits $1 \leqslant j \leqslant n$ such that $\lambda_{j} > \lambda_{j + 1}$ (set $\lambda_{n+1} = 1 + (1 - \delta)/2$). Then the following diagram commutes
	\begin{equation*}
		\begin{tikzcd}
			K\operatorname{Per}(H_{\lambda_-^{\BQ_p}} \backslash V_{\lambda_-^{\BQ_p}}) \ar[r, "\cR"] \ar[d, "\Ind_{\D P}^{\D H}"']
			& K\operatorname{Per}(\D K \backslash X_{\lambda^{\mathbb{R}}_-}) \ar[d, "Rp_{2*} \circ p_1^*"]
			\\
			K\operatorname{Per}(H_{\lambda^{\BQ_p}} \backslash V_{\lambda^{\BQ_p}}) \ar[r, "\cR"]
			& K\operatorname{Per}(\D K \backslash X_\lambda^{\mathbb{R}}),
		\end{tikzcd}
	\end{equation*}
    where $\lambda^{\mathbb{R}}_{-} = \lambda^{\mathbb{R}} - e_j$, and $\lambda_-^{\BQ_p}$ is the image of $\lambda^{\mathbb{R}}_-$ under the assignment \ref{cls: lambda to lambdap}.
    \end{theorem}

\begin{proof}
	The theorem is proven by unraveling the relations between the full rank parts of the Vogan varieties in (\ref{diag:L-ind}) and the slices within. 
	
	Write $\varphi: V_{\lambda^{\BQ_p}}^{reg} \to V_{\lambda^{\BQ_p}}^{reg,\diamond}$ for the projection defined in \ref{cls: geom comparison}. Define
	\begin{itemize}
		\item $V_{\lambda_M^{\BQ_p}}^{\D P,reg} := V_{\lambda^{\BQ_p}}^{reg} \cap \D \fp$,
		
		\item $V_{\lambda_M^{\BQ_p}}^{\D P,reg,\diamond} := \varphi(V_{\lambda_M^{\BQ_p}}^{\D P,reg})$.
	\end{itemize}
	We write $\varphi_P$ for the projection $V_{\lambda_M^{\BQ_p}}^{\D P,reg} \surj V_{\lambda_M^{\BQ_p}}^{\D P,reg,\diamond}$, which is the restriction of $\varphi$. Let $V_{\lambda_-^{\BQ_p}}^{reg}$ be the full rank part of $V_{\lambda_-^{\BQ_p}}$. Define $V_{\lambda_-^{\BQ_p}}^{reg,\diamond}$ and the projection map $\varphi_-$ by forgetting the same degrees as $\varphi$, that is, if $\varphi$ forgets the component $\Hom(W_{-\frac12}, W_{\frac12})$ (resp. $\Hom(W_{-1}, W_{0}) \times \Hom(W_{0}, W_{1})$) in $V_{\lambda^{\BQ_p}}$, then $\varphi_-$ forgets the component $\Hom(W_{-\frac12}^-, W_{\frac12}^-)$ (resp. $\Hom(W_{-1}^-, W_{0}^-) \times \Hom(W_{0}^-, W_{1}^-)$) in $V_{\lambda_-^{\BQ_p}}$. Define $\varphi_M: V_{\lambda_M^{\BQ_p}}^{reg} \to  V_{\lambda_M^{\BQ_p}}^{reg,\diamond}$ similarly. Then the isomorphism $u: V_{\lambda_M^{\BQ_p}} \bij V_{\lambda_-^{\BQ_p}}$ identifies $\varphi_M$ with $\varphi_-$. These objects fit into the commutative diagram
	\begin{equation*}
		\begin{tikzcd}
			V_{\lambda^{\BQ_p}} 
			& V_{\lambda_M^{\BQ_p}}^{\D P} \ar[l, hook, "q_2"'] \ar[r, two heads, "q_1"]
			& V_{\lambda_M^{\BQ_p}} \ar[r, "u", "\cong"']
			& V_{\lambda_-^{\BQ_p}}
			\\
			V_{\lambda^{\BQ_p}}^{reg} \ar[u, hook] \ar[d, two heads, "\varphi"']
			& V_{\lambda_M^{\BQ_p}}^{\D P,reg} \ar[l, hook] \ar[r, two heads] \ar[u, hook] \ar[d, two heads, "\varphi_P"']
			& V_{\lambda_M^{\BQ_p}}^{reg} \ar[r, "\cong"] \ar[u, hook] \ar[d, two heads, "\varphi_M"']
			& V_{\lambda_-^{\BQ_p}}^{reg} \ar[u, hook] \ar[d, two heads, "\varphi_-"']
			\\
			V_{\lambda^{\BQ_p}}^{reg,\diamond} 
			& V_{\lambda_M^{\BQ_p}}^{\D P,reg,\diamond} \ar[l, hook] \ar[r, "q_1^\diamond"]
			& V_{\lambda_M^{\BQ_p}}^{reg,\diamond} \ar[r, "\cong"]
			& V_{\lambda_-^{\BQ_p}}^{reg,\diamond}
		\end{tikzcd}.
	\end{equation*}
	Here $q_1$ is just the restriction of the quotient map $\D \fp \surj \D \fm$, and $q_1^\diamond$ is the same as $q_1$ degree-wise. In this diagram, only the middle top square is not Cartesian. 
    
    We now look at the slices, i.e. the fibers of the $\varphi$'s. By definition of $\D P$, the element $x^\diamond$ we fixed lies in $V_{\lambda_M^{\BQ_p}}^{\D P,reg,\diamond}$. Write $\bar x^\diamond = q_1^\diamond(x^\diamond) \in V_{\lambda_M^{\BQ_p}}^{reg,\diamond}$, and use the same notation for its image in $V_{\lambda_-^{\BQ_p}}^{reg,\diamond}$. Then the fibers $\varphi\inv(x^\diamond)$, $\varphi_P\inv(x^\diamond)$, $\varphi_M\inv(\bar x^\diamond)$ and $\varphi_-\inv(\bar x^\diamond)$ are all identified by the maps $q_2$, $q_1$, and $u$ with either
    \begin{equation*}
		\{(f_{-1},f_0) \in \Hom(W_{-1}, W_{0})_{reg} \times \Hom(W_{0}, W_{1})_{reg} \mid f_0^* = -f_{-1}\}
	\end{equation*}    	
	if $\lambda^{\BQ_p}(\Fr)$ is supported on $\BZ$, or 
    \begin{equation*}
		\{f_{-\frac12} \in \Hom(W_{-\frac12},W_{\frac12})_{reg} \mid f_{-\frac12}^*=-f_{-\frac12}\}
	\end{equation*}   
	if $\lambda^{\BQ_p}(\Fr)$ is supported on $\frac12\BZ\backslash \BZ$. By Proposition \ref{prop: full rank part} these fibers can be further identified with $\operatorname{Sym}_n(\BC)_{reg}$, $\operatorname{Skew}_n(\BC)_{reg}$, or $\GL_n(\BC)$, which we denote by $F$ for convenience.
	
	Now we would like to write the full rank parts as induction spaces. For $V_{\lambda^{\BQ_p}}^{reg}$, $V_{\lambda_M^{\BQ_p}}^{reg}$ and $V_{\lambda_-^{\BQ_p}}^{reg}$, Proposition \ref{prop: full rank part} applies (with a slight modification for $V_{\lambda_M^{\BQ_p}}^{reg}$ regarding the $\GL_1$-factor). For $V_{\lambda_M^{\BQ_p}}^{\D P,reg}$, we have a similar result:
	
	\begin{lemma}\label{lemma: full rank part P}~
		\begin{enumerate}
			\item $V_{\lambda_M^{\BQ_p}}^{\D P,reg,\diamond}$ forms a single $H_{\lambda_M^{\BQ_p}}^{\D P}$-orbit. As a result,
			\begin{equation*}
				[H_{\lambda_M^{\BQ_p}}^{\D P} \backslash V_{\lambda_M^{\BQ_p}}^{\D P,reg}] 
				\cong 
				[H_{\lambda_M^{\BQ_p}}^{\D P}(x^\diamond) \backslash \varphi_P\inv(x^\diamond)] .
			\end{equation*}
			
			\item Consider the diagram
			\begin{equation*}
				\begin{tikzcd}[row sep=small]
					H_{\lambda^{\BQ_p}}(x^\diamond) \ar[d, phantom, "\racts" {rot90, description}]
					& H_{\lambda_M^{\BQ_p}}^{\D P}(x^\diamond) \ar[d, phantom, "\racts" {rot90, description}] \ar[l, hook] \ar[r]
					& H_{\lambda_-^{\BQ_p}}(\bar x^\diamond) \ar[d, phantom, "\racts" {rot90, description}]
					\\
					\varphi\inv(x^\diamond)
					& \varphi_P\inv(x^\diamond) \ar[l, "q_2"', "\cong"] \ar[r, "u \, \circ \, q_1", "\cong"']
					& \varphi_-\inv(\bar x^\diamond)
				\end{tikzcd}
			\end{equation*}
            \begin{itemize}
            \item If $\lambda^{\BQ_p}(\Fr)$ is supported on $\frac12\BZ\backslash\BZ$, the diagram is isomorphic to 			\begin{equation*}
				\begin{tikzcd}[row sep=small]
					\D P(\lambda^{\mathbb{R}}) \ar[d, phantom, "\racts" {rot90, description}]
					& \D P(\lambda^{\mathbb{R}},\lambda^{\mathbb{R}}_-) \ar[d, phantom, "\racts" {rot90, description}] \ar[l, hook] \ar[r, hook]
					& \D P(\lambda^{\mathbb{R}}_-) \ar[d, phantom, "\racts" {rot90, description}]
					\\
					F
					& F \ar[l, equal] \ar[r, equal]
					& F
				\end{tikzcd}
			\end{equation*}
			\item If $\lambda^{\BQ_p}(\Fr)$ is supported on $\BZ$ and $H$ is not even orthogonal, the diagram is isomorphic to			
            \begin{equation*}
				\begin{tikzcd}[row sep=small]
					\D P(\lambda^{\mathbb{R}}) \times S\operatorname{Isom}(W_0) \ar[d, phantom, "\racts" {rot90, description}]
					& \D P(\lambda^{\mathbb{R}},\lambda^{\mathbb{R}}_-) \times S\operatorname{Isom}(W_0) \ar[d, phantom, "\racts" {rot90, description}] \ar[l, hook] \ar[r, hook]
					& \D P(\lambda^{\mathbb{R}}_-) \times S\operatorname{Isom}(W_0) \ar[d, phantom, "\racts" {rot90, description}]
					\\
					F
					& F \ar[l, equal] \ar[r, equal]
					& F
				\end{tikzcd}
			\end{equation*}
			 where $S\operatorname{Isom}(W_0)$ stands for the isometry group of $(W_0, \langle-,- \rangle)$ of determinant $1$. For $H$ even orthogonal, one just replaces $S\operatorname{Isom}$ by $\operatorname{Isom}$. 
            \end{itemize}
		\end{enumerate}
	\end{lemma}
	
	Assume this for the moment. As a consequence, if $\lambda^{\BQ_p}(\Fr)$ is supported on $\frac12\BZ\backslash\BZ$, we have a commutative diagram 
	\begin{equation*}
		{\small
			\begin{tikzcd}[row sep=small, column sep=small]
				{[\D P(\lambda^{\mathbb{R}}) \backslash F]} \ar[d, dash, "\cong"]
				& {[\D P(\lambda^{\mathbb{R}},\lambda^{\mathbb{R}}_-) \backslash F]} \ar[l, "p_2"'] \ar[rr, "p_1"] \ar[d, dash, "\cong"]
				&& {[\D P(\lambda^{\mathbb{R}}_-) \backslash F]} \ar[d, dash, "\cong"]
				\\
				{[H_{\lambda^{\BQ_p}}(x^\diamond) \backslash \varphi\inv(x^\diamond)]} \ar[d, dash, "\cong"]
				& {[H_{\lambda_M^{\BQ_p}}^{\D P}(x^\diamond) \backslash \varphi_P\inv(x^\diamond)]} \ar[l] \ar[r] \ar[d, dash, "\cong"]
				& {[H_{\lambda_M^{\BQ_p}}(\bar x^\diamond) \backslash \varphi_M\inv(\bar x^\diamond)]} \ar[r] \ar[d, dash, "\cong"]
				& {[H_{\lambda_-^{\BQ_p}}(\bar x^\diamond) \backslash \varphi_-\inv(\bar x^\diamond)]} \ar[d, dash, "\cong"]
				\\
				{[H_{\lambda^{\BQ_p}} \backslash V_{\lambda^{\BQ_p}}^{reg}]} \ar[d, hook, "\iota_{geom}"']
				& {[H_{\lambda_M^{\BQ_p}}^{\D P} \backslash V_{\lambda_M^{\BQ_p}}^{\D P,reg}]} \ar[l] \ar[r] \ar[d, hook]
				& {[H_{\lambda_M^{\BQ_p}} \backslash V_{\lambda_M^{\BQ_p}}^{reg}]} \ar[r] \ar[d, hook]
				& {[H_{\lambda_-^{\BQ_p}} \backslash V_{\lambda_-^{\BQ_p}}^{reg}]} \ar[d, hook, "\iota_{geom,-}"]
				\\
				{[H_{\lambda^{\BQ_p}} \backslash V_{\lambda^{\BQ_p}}]} \ar[ur, phantom, "\text{\small (Cartesian)}" {description}]
				& {[H_{\lambda_M^{\BQ_p}}^{\D P} \backslash V_{\lambda_M^{\BQ_p}}^{\D P}]} \ar[l, "q_2"'] \ar[r, "q_1"]
				& {[H_{\lambda_M^{\BQ_p}} \backslash V_{\lambda_M^{\BQ_p}}]} \ar[r, "u"]
				& {[H_{\lambda_-^{\BQ_p}} \backslash V_{\lambda_-^{\BQ_p}}]}
			\end{tikzcd}
		},
	\end{equation*}
	If $\lambda^{\BQ_p}(\Fr)$ is supported on $\BZ$, then the groups in the top row will be replaced by their products with $S\operatorname{Isom}(W_0)$ or $\operatorname{Isom}$. As a result, by base change, 
	\begin{equation*}
		\iota_{geom}^* \circ \Ind_{\D P}^{\D H} 
		= \iota_{geom}^* \circ Rq_{2*} \circ q_1^* \circ u^*
		\cong Rp_{2*} \circ p_1^* \circ \iota_{geom,-}^*.
	\end{equation*}
	Combined with the map $\theta$ in Corollary \ref{cor: geom comparison}, this completes the proof of the theorem.
	
	It remains to prove Lemma \ref{lemma: full rank part P}.

	For part (1), consider the commutative diagram
	\begin{equation*}
		\begin{tikzcd}[row sep =small]
			H_{\lambda_M^{\BQ_p}}^{\D P} \ar[d, phantom, "\racts" {rot90, description}] \ar[r, two heads]
			& H_{\lambda_-^{\BQ_p}} \ar[d, phantom, "\racts" {rot90, description}]
			\\
			V_{\lambda_M^{\BQ_p}}^{\D P, reg,\diamond} \ar[r, two heads, "u \, \circ \, q_1^\diamond"]
			& V_{\lambda_-^{\BQ_p}}^{reg,\diamond}
		\end{tikzcd}.
	\end{equation*}
	Recall that the action on the right is transitive by Proposition \ref{prop: full rank part}. Hence part (1) of Lemma \ref{lemma: full rank part P} follows if $H_{\lambda_M^{\BQ_p}}^{\D P}(\bar x^\diamond)$ acts transitively on the fiber $(u \circ q_1^{\diamond})^{-1}(\bar x^\diamond)$. Viewing $H_{\lambda_M^{\BQ_p}}(\bar x^\diamond)$ as a subgroup of $H_{\lambda_M^{\BQ_p}}^{\D P}(\bar x^\diamond)$, it suffices to have $H_{\lambda_M^{\BQ_p}}(\bar x^\diamond)$ acting transitively.
	
	Different elements $y^\diamond$ in the fiber $(u \circ q_1^{\diamond})^{-1}(\bar x^\diamond)$ only differ at the degree $-\tilde{\lambda}_j, \tilde{\lambda}_j - 1$ components. It suffices to consider the negative degrees by Lemma~\ref{lemma: SpSO param space}. At degree $-\tilde{\lambda}_j$, $y^\diamond$ is the sum of two parts: its restriction to $W_{-\tilde{\lambda}_j}^-$ which agrees with $x^\diamond|_{W_{-\tilde{\lambda}_j}^-} = \bar x^\diamond|_{W_{-\tilde{\lambda}_j}^-}$, and its restriction to the line $W_{-\tilde{\lambda}_j}' = \BC \alpha_j$. Hence $y^\diamond$ is completely determined by $y^\diamond(\alpha_j)$, and the condition that $y^\diamond_j$ has full rank implies $y^\diamond(\alpha_j) \not\in x^\diamond(W_{-\tilde{\lambda}_j}^-)$. To express this differently, let us write $\alpha_{i,j} \in W_i$ instead of $\alpha_j$ for the basis vectors chosen in \ref{cls: basis of W}, to avoid confusions. We may decompose
	\begin{equation*}
		W_{-\tilde{\lambda}_j+1} = W_{-\tilde{\lambda}_j+1}^\flat \oplus W_{-\tilde{\lambda}_j+1}^\sharp
	\end{equation*}
	where $W_{-\tilde{\lambda}_j+1}^\flat = x^\diamond(W_{-\tilde{\lambda}_j}^-)$ is the span of the first $j-1$ basis vectors $\alpha_{-\tilde{\lambda}_j+1,1},\ldots,\alpha_{-\tilde{\lambda}_j+1,j-1}$, and $W_{-\tilde{\lambda}_j+1}^\sharp$ is spanned by the remaining basis vectors $\alpha_{-\tilde{\lambda}_j+1,j},\ldots$. Decompose $y^\diamond(\alpha_{-\tilde{\lambda}_j,j}) = v^\flat + v^\sharp$ accordingly. Then $y^\diamond$ having full rank implies $v^\sharp \neq 0$. The above discussion can be summarized in the following picture
	\begin{equation*}
		\begin{tikzcd}[row sep=small]
			W_{-\tilde{\lambda}_j} \ar[d, equal]
			& W_{-\tilde{\lambda}_j+1} \ar[d, equal]
			&&&[-2ex]
			\\
			W_{-\tilde{\lambda}_j}^- \ar[r, "y^\diamond =x^\diamond", "\sim"'] \ar[d, phantom, "\oplus" description] 
			& W_{-\tilde{\lambda}_j+1}^\flat \ar[d, phantom, "\oplus" description] 
			&
			& v^\flat \ar[d, phantom, "+" description]
			\\
			W_{-\tilde{\lambda}_j}' \ar[ur, "y^\diamond"'] \ar[r]
			& W_{-\tilde{\lambda}_j+1}^\sharp
			& \alpha_{-\tilde{\lambda}_j,j} \ar[ur, mapsto] \ar[r, mapsto]
			& v^\sharp \ar[r, phantom, "\neq0" description]
			& {}
		\end{tikzcd}.
	\end{equation*}
	
	We now claim that $y^\diamond$ and $x^\diamond$ is in the same $H_{\lambda_M^{\BQ_p}}(\bar x^\diamond)$-orbit. To see this, recall from Remark \ref{rmk: stabilizer as matrices} that the restriction of $H_{\lambda_M^{\BQ_p}}(\bar x^\diamond)$ to $W_{-\tilde{\lambda}_j}^- \oplus W_{-\tilde{\lambda}_j+1}$ is
	\begin{multline*}
		\left\{ 
		\left((\bar{x}^{\diamond}_{-\tilde{\lambda}_j})^{-1} \circ  A \circ  \bar{x}^\diamond_{-\tilde{\lambda}_j}, 
		\begin{pmatrix}
			A  & B\\
			& C
		\end{pmatrix} \right) \mid
		A \in \GL(W_{-\tilde{\lambda}_j}^\flat), B \in \Hom(W_{-\tilde{\lambda}_j}^\sharp, W_{-\tilde{\lambda}_j}^\flat), C \in \GL(W_{-\tilde{\lambda}_j+1}^\sharp) \right\} 
	\end{multline*}
    where $(\bar{x}^{\diamond}_{-\tilde{\lambda}_j})^{-1}$ is taken on $W_{-\tilde{\lambda}_j}^\flat$. Since $C$ can be any arbitrary element in $\GL(W_{-\tilde{\lambda}_j+1}^\sharp)$, there exists a choice of $C$ so that $C v^\sharp = x^\diamond(\alpha_{-\tilde{\lambda}_j,j}) = \alpha_{-\tilde{\lambda}_j+1,j}$, the $j$-th basis vector of $W_{-\tilde{\lambda}_j+1}$. Since $B$ can be arbitrary, there exists one so that $B v^\sharp = -v^\flat$. Then, taking $A$ to be the identity matrix, we have
	\begin{equation*}
		\begin{pmatrix}
			\operatorname{id} & B\\
			& C
		\end{pmatrix} \cdot y^\diamond(\alpha_{-\tilde{\lambda}_j,j}) = v^\flat + B v^\sharp + C v^\sharp = x^\diamond(\alpha_{-\tilde{\lambda}_j,j})
	\end{equation*}
	and consequently $g \cdot y^\diamond = x^\diamond$ for any element $g \in H_{\lambda_M^{\BQ_p}}(\bar x^\diamond)$ whose restriction to $W_{-\tilde{\lambda}_j} \oplus W_{-\tilde{\lambda}_j+1}$ equals $\left( \operatorname{id}, 
	\begin{pmatrix}
		\operatorname{id} & B\\
		& C
	\end{pmatrix} \right)$.
	Therefore any $y^\diamond \in (q_1^{\diamond})^{-1}(\bar x^\diamond)$ is in the same orbit as $x^\diamond$, and part (1) follows.
	
	Now we look at part (2). For simplicity, let us work in the case where $\lambda^{\BQ_p}(\Fr)$ is supported on $\frac12\BZ \backslash \BZ$; the other case is similar. The most straightforward proof is by looking directly at the matrix descriptions of the stabilizers as in Remark \ref{rmk: stabilizer as matrices} using the basis \ref{cls: basis of W}. Recall from Remark \ref{rmk: stabilizer as matrices} that for $i < 0$ the subspace 
    \begin{equation*}
        W_i^\circ := \operatorname{span}\{\alpha_{i,1+\dim W_{i-1}}, \ldots, \alpha_{i,\dim W_i}\} \subset W_i
    \end{equation*}
    a complement to $x^\diamond(W_{i-1}) \subset W_i$. Define $W_i^{-,\circ}$ similarly obtaining a complement to $\bar x^\diamond(W_{i-1}^-)$ in $W_i^-$. 
    If we just look at degrees $-\tilde{\lambda}_j$ and $-\tilde{\lambda}_j+1$, these spaces can be nicely described by their basis vectors: 
    \begin{align*}
    	W_{-\tilde \lambda_j-1} = W_{-\tilde \lambda_j-1}^-:&\quad 
    	\alpha_1, \ldots, \alpha_d \quad (d := \dim W_{-\tilde \lambda_j-1})
    	\\
    	W_{-\tilde \lambda_j} = W_{-\tilde \lambda_j}^- \oplus \BC \alpha_j :&\quad
    	\overbrace{\underbrace{ \alpha_1, \ldots, \alpha_d}_{\bar x^\diamond(W_{-\tilde \lambda_j-1}^-)}}^{x^\diamond(W_{-\tilde \lambda_j-1})},
    	\overbrace{\underbrace{ \alpha_{d+1}, \ldots, \alpha_{j-1} }_{W_{-\tilde \lambda_j}^{-,\circ}}, \alpha_j }^{W_{-\tilde \lambda_j}^\circ}
    	\\
    	W_{-\tilde \lambda_j+1} = W_{-\tilde \lambda_j+1}^-:&\quad
    	\lefteqn{\overbrace{\phantom{\alpha_1, \ldots, \alpha_d, \alpha_{d+1}, \ldots \alpha_{j-1}, \alpha_jj}}^{x^\diamond(W_{-\tilde \lambda_j})}}
    	\underbrace{ \alpha_1, \ldots, \alpha_d, \alpha_{d+1}, \ldots, \alpha_{j-1} }_{\bar x^\diamond(W_{-\tilde \lambda_j}^-)} , 
    	\underbrace{ \alpha_j,
    	\overbrace{ \alpha_{j+1}, \ldots }^{W_{-\tilde \lambda_j+1}^\circ} }_{W_{-\tilde \lambda_j+1}^{-,\circ}}.
    \end{align*}
    Then, at degrees $-\tilde{\lambda}_j-1$, $-\tilde{\lambda}_j$ and $-\tilde{\lambda}_j+1$, an element of $H_{\lambda^{\BQ_p}}(x^\diamond)$ has the following form:
	\begin{equation*}
		\left(\ldots,
		\begin{pmatrix}
			\ddots & \vdots\\
			& g_{-\tilde{\lambda}_j-1}
		\end{pmatrix},
		\begin{pmatrix}
			\ddots & \vdots & \vdots & \vdots\\
			& g_{-\tilde{\lambda}_j-1} & * & *\\ 
			&& g_{-\tilde{\lambda}_j}^{11} & g_{-\tilde{\lambda}_j}^{12}\\
			&& g_{-\tilde{\lambda}_j}^{21} & g_{-\tilde{\lambda}_j}^{22}
		\end{pmatrix},
		\begin{pmatrix}
			\ddots & \vdots & \vdots & \vdots & \vdots\\
			& g_{-\tilde{\lambda}_j-1} & * & * & *\\
			&& g_{-\tilde{\lambda}_j}^{11} & g_{-\tilde{\lambda}_j}^{12} & *\\
			&& g_{-\tilde{\lambda}_j}^{21} & g_{-\tilde{\lambda}_j}^{22} & *\\
			&&&& g_{-\tilde{\lambda}_j+1}
		\end{pmatrix}, 
		\ldots \right)
	\end{equation*}
	where $g_i \in \GL(W_i^\circ)$, 
	\begin{equation*}
		g_{-\tilde{\lambda}_j} = 
		\begin{pmatrix}
			g_{-\tilde{\lambda}_j}^{11} & g_{-\tilde{\lambda}_j}^{12}\\
			g_{-\tilde{\lambda}_j}^{21} & g_{-\tilde{\lambda}_j}^{22}
		\end{pmatrix}
        \in \GL(W_{-\tilde{\lambda}_j}^\circ)
	\end{equation*}
	is a block matrix with $g_{-\tilde{\lambda}_j}^{11} \in \GL(W_{-\tilde{\lambda}_j}^{-,\circ})$ and $g_{-\tilde{\lambda}_j}^{22} \in \GL(\BC\alpha_{-\tilde{\lambda}_j,j})$. Each $*$ in each matrix is arbitrary but has to equal to the corresponding entry in the next matrix. Similarly, an element of $H_{\lambda_M^{\BQ_p}}^{\D P}(x^\diamond)$ is of the form
	\begin{equation*}
		\left(\ldots,
		\begin{pmatrix}
			\ddots & \vdots\\
			& g_{-\tilde{\lambda}_j-1}
		\end{pmatrix},
		\begin{pmatrix}
			\ddots & \vdots & \vdots & \vdots\\
			& g_{-\tilde{\lambda}_j-1} & * & *\\ 
			&& g_{-\tilde{\lambda}_j}^{11} & g_{-\tilde{\lambda}_j}^{12}\\
			&&& g_{-\tilde{\lambda}_j}^{22}
		\end{pmatrix},
		\begin{pmatrix}
			\ddots & \vdots & \vdots & \vdots & \vdots\\
			& g_{-\tilde{\lambda}_j-1} & * & * & *\\
			&& g_{-\tilde{\lambda}_j}^{11} & g_{-\tilde{\lambda}_j}^{12} & *\\
			&& & g_{-\tilde{\lambda}_j}^{22} & *\\
			&&&& g_{-\tilde{\lambda}_j+1}
		\end{pmatrix}, 
		\ldots \right).
	\end{equation*}
	The missing $g_{-\tilde{\lambda}_j}^{21}$ is a result of the condition $g_{-\tilde{\lambda}_j}(W_{-\tilde{\lambda}_j}^-) \subset W_{-\tilde{\lambda}_j}^-$ coming from $\D P$. Finally, an element of $H_{\lambda_-^{\BQ_p}}(\bar x^\diamond)$ has the form
	\begin{equation*}
		\left(\ldots,
		\begin{pmatrix}
			\ddots & \vdots\\
			& g_{-\tilde{\lambda}_j-1}
		\end{pmatrix},
		\begin{pmatrix}
			\ddots & \vdots & \vdots \\
			& g_{-\tilde{\lambda}_j-1} & * \\ 
			&& g_{-\tilde{\lambda}_j}'
		\end{pmatrix},
		\begin{pmatrix}
			\ddots & \vdots & \vdots & \vdots & \vdots\\
			& g_{-\tilde{\lambda}_j-1} & * & * & *\\
			&& g_{-\tilde{\lambda}_j}' & * & *\\
			&& & g_{-\tilde{\lambda}_j+1}^{11} & g_{-\tilde{\lambda}_j+1}^{12}\\
			&&& g_{-\tilde{\lambda}_j+1}^{21} & g_{-\tilde{\lambda}_j+1}^{22}
		\end{pmatrix}, 
		\ldots \right)
	\end{equation*}
	where $g_{-\tilde{\lambda}_j}' \in \GL(W_{-\tilde{\lambda}_j}^{-,\circ})$, and
	\begin{equation*}
		g_{-\tilde{\lambda}_j+1} = 
		\begin{pmatrix}
			g_{-\tilde{\lambda}_j+1}^{11} & g_{-\tilde{\lambda}_j+1}^{12}\\
			g_{-\tilde{\lambda}_j+1}^{21} & g_{-\tilde{\lambda}_j+1}^{22}
		\end{pmatrix} \in \GL(W_{-\tilde{\lambda}_j+1}^{-,\circ})
	\end{equation*}
	is a block matrix with $g_{-\tilde{\lambda}_j+1}^{22} \in \GL(W_{-\tilde{\lambda}_j+1}^\circ)$ and $g_{-\tilde{\lambda}_j+1}^{11} \in \GL(\BC \alpha_{-\lambda_j+1,j})$. From these explicit descriptions, it is clear that the projection $\pi_{-\frac12}: \prod_{i \le -\frac12} \GL(W_i) \surj \GL(W_{-\frac12})$ (resp. $\pi_{-\frac12}^-: \prod_{i \le -\frac12} \GL(W_i^-) \surj \GL(W_{-\frac12}^-) = \GL(W_{-\frac12})$) restricts to injections on these stablizer groups, and
	\begin{equation*}
		\pi_{-\frac12}\big( H_{\lambda^{\BQ_p}}(x^\diamond) \big) \cap \pi_{-\frac12}^-\big( H_{\lambda_-^{\BQ_p}}(\bar x^\diamond) \big) = \pi_{-\frac12}\big( H_{\lambda_M^{\BQ_p}}^{\D P}(x^\diamond) \big).
	\end{equation*}
	By Proposition \ref{prop: full rank part}, we have isomorphisms
	\begin{equation*}
		\begin{tikzcd}[row sep=small]
			H_{\lambda^{\BQ_p}}(x^\diamond) \ar[d, phantom, "\racts" {rot90, description}] \ar[r, "\pi_{-\frac12}", "\cong"']
			& \D P(\lambda^{\mathbb{R}}) \ar[d, phantom, "\racts" {rot90, description}] 
			\\
			\varphi\inv(x^\diamond) \ar[r, "\cong"]
			& F
		\end{tikzcd} 
		\quad
		\begin{tikzcd}[row sep=small]
			H_{\lambda_-^{\BQ_p}}(\bar x^\diamond) \ar[d, phantom, "\racts" {rot90, description}] \ar[r, "\pi_{-\frac12}^-", "\cong"']
			& \D P(\lambda^{\mathbb{R}}_-) \ar[d, phantom, "\racts" {rot90, description}] 
			\\
			\varphi_-\inv(\bar x^\diamond) \ar[r, "\cong"]
			& F
		\end{tikzcd}.
	\end{equation*}
	Thus $\pi_{-\frac12}$ sends $H_{\lambda_M^{\BQ_p}}^{\D P}(x^\diamond)$ isomorphically onto $\D P(\lambda^{\mathbb{R}}, \lambda^{\mathbb{R}}_-)$ in a way compatible with the action on $F$. This completes the proof of (2).
\end{proof}

\appendix
\section{Factorization of translation functors}\label{sec: factoring translation}

This appendix is inspired by \cite{Trapa:2001}. Let 
\[
G^{an} := U(n,0) = \{g \in GL_n(\mathbb{C}) \, | \, {}^t\bar{g}g = I_n\},
\]
and $T$ be the subgroup of diagonal matrices. Define its inner form $G' := U(p,q)$ by twisting the rational structure of $G^{an}$ by $I_{p,q}$. Then $G'_{\mathbb{C}} = G^{an}_{\mathbb{C}}$ and $T$ is a subgroup of $G'$ under this identification. Denote the complexified Lie algebra of $G^{an}$ and $T$ by $\mathfrak{g}$ and $\mathfrak{t}$ respectively. Identify $\mathfrak{g}$ with $\mathfrak{gl}_n$ and $\mathfrak{t}$ with the Lie subalgebra of diagonal matrices. Choose the Borel subalgebra $\mathfrak{b}$ to be the Lie subalgebra of uppertriangular matrices. Denote by $\Phi^{+}(\mathfrak{g}, \mathfrak{t})$ the set of positive roots associated with $\mathfrak{b}$ and $\Delta(\mathfrak{g}, \mathfrak{t})$ the set of simple roots and $W = W(\mathfrak{g}, \mathfrak{t})$ the Weyl group. Through differentiation we get an inclusion $X^{*}(T) \hookrightarrow \mathfrak{t}^{*}$. Denote its image by $\Lambda$. Let $e_i \in \mathfrak{t}^{*}$ be the image of the character
\[
\chi_i: T \rightarrow \mathbb{S}^1, \quad (t_1, \cdots, t_n) \mapsto t_i.
\]
They give a basis of $\mathfrak{t}^{*}$. Denote the $i$-th coefficient of $\lambda \in \mathfrak{t}^*$ by $\lambda_i$. For $\lambda \in \mathfrak{t}^{*}$, we say $\lambda$ is integral if $\lambda_i - \lambda_j \in \mathbb{Z}$ for all $1 \leqslant i, j \leqslant n$, and $\lambda$ is dominant if $\lambda_{i} - \lambda_{i+1} \geqslant 0$ for all $1\leqslant i < n$. 

Since $W$ acts on $\mathfrak{t}^{*}$ faithfully by permuting the $n$ coordinates, we can identify $W$ with $S_n$ by letting $e_{w(i)} := w(e_i)$ for $w \in W$. Suppose $\lambda \in \mathfrak{t}^*$. Let $W_{\lambda}$ be the stabilizer of $\lambda$ in $W$. We denote the translation functor $\psi^{\lambda - e_i}_{\lambda}$ by $\mathcal{T}_{i}$. We have suppressed $\lambda$ in $\mathcal{T}_{i}$, which should be determined in the context. We will
frequently use the following property of translation functors:
\begin{align}
\label{eq: Weyl conjugate}
\psi^{\lambda}_{\mu} = \psi^{w(\lambda)}_{w(\mu)}
\end{align}
for any $\lambda, \mu \in \mathfrak{t}^*$ and $w \in W$. 

\begin{lemma}
\label{lemma: translation by coherent family}
Suppose $\lambda \in \mathfrak{t}^{*}$ is integral and
\[
\mu = \lambda - (e_{j_1} + \cdots + e_{j_r}), \quad j_1 < \cdots < j_r.
\]
For any coherent family $\Theta$ of virtual representations of $G'$ over $\lambda + \Lambda$,
\[
\psi^{\mu}_{\lambda} (\Theta(\lambda)) = \sum_{s \in W_{\lambda}/W_{\lambda} \cap W_{\mu}} \Theta(s(\mu)).
\]
\end{lemma}

\begin{proof}
Let $L := \bigwedge^{r} ({\rm Std})$, the irreducible representation of $G'$ of extremal weight $\nu_0 := e_{j_1} + \cdots + e_{j_r}$ with respect to $T$. 
Note the set of weights $wt(L) = W \cdot \nu_0$. Hence all weights of $L$ have multiplicity one. By the definition of coherent family, we have
\[
L^{\vee} \otimes \Theta(\lambda) = \sum_{\nu' \in wt(L^{\vee})}  \Theta(\lambda+ \nu') = \sum_{\nu \in wt(L)}  \Theta(\lambda - \nu).
\]
Then 
\[
\psi^{\mu}_{\lambda} (\Theta(\lambda)) = \sum_{\substack{\nu \in wt(L): \\ \lambda - \nu \in W \cdot \mu }}  \Theta(\lambda - \nu).
\]
It is not hard to see that for any $\nu \in wt(L)$, $\lambda - \nu \in W \cdot \mu$ if and only if $\nu \in W_{\lambda} \cdot \nu_{0}$. Moreover, if $s_1(\nu_0) = s_2(\nu_0)$ for $s_1, s_2 \in W_{\lambda}$, then $s_2^{-1}s_1 \in W_{\mu} \cap W_{\lambda}$. Therefore,
\[
\psi^{\mu}_{\lambda} (\Theta(\lambda)) = \sum_{s \in W_{\lambda}/W_{\lambda} \cap W_{\mu}} \Theta(\lambda - s(\nu_0)) = \sum_{s \in W_{\lambda}/W_{\lambda} \cap W_{\mu}} \Theta(s(\lambda - \nu_0)) = \sum_{s \in W_{\lambda}/W_{\lambda} \cap W_{\mu}} \Theta(s(\mu)).
\]

\end{proof}



\begin{proposition}
\label{prop: factorization of translation}
Suppose $\lambda \in \mathfrak{t}^{*}$ is integral and 
\[
\mu = \lambda - (e_{j_1} + \cdots + e_{j_r}), \quad j_1 < \cdots < j_r.
\]
If $\lambda_{j_1} \geqslant \cdots \geqslant \lambda_{j_r}$, then
\[
m \cdot \psi^{\mu}_{\lambda} = \mathcal{T}_{j_1} \circ \cdots \circ \mathcal{T}_{j_r}
\]
in the Grothendieck group,
where $m = |{\rm Stab}_{S_{r}}(\lambda_{j_1}, \cdots, \lambda_{j_r})|$ and $S_r$ is the symmetric group acting by permuting the $\lambda_{j_i}$'s.
\end{proposition}

\begin{proof}
We can choose $w \in W$ such that $w(\lambda), w(\mu)$ are both dominant and $w(j_1) < \cdots < w(j_r)$. Then it is equivalent to show that
\[
m' \cdot \psi^{w(\mu)}_{w(\lambda)} = \mathcal{T}_{w(j_1)} \circ \cdots \circ \mathcal{T}_{w(j_r)},
\]
where $m' = |{\rm Stab}_{S_{r}}(w(\lambda)_{w(j_1)}, \cdots, w(\lambda)_{w(j_r)})| = m$. Therefore, we will assume $\lambda, \mu$ are dominant and prove the formula by induction on $r$. If $r = 1$, the statement is trivial. Let 
\[
\lambda' = \lambda - (e_{j_2} + \cdots + e_{j_r}).
\]
By induction we can assume 
\[
m' \cdot \psi^{\lambda'}_{\lambda} = \mathcal{T}_{j_2} \circ \cdots \circ \mathcal{T}_{j_r},
\]
where $m' = |{\rm Stab}_{S_{r-1}}(\lambda_{j_2}, \cdots, \lambda_{j_r})|$. Then
\[
\mathcal{T}_{j_1} \circ \cdots \circ \mathcal{T}_{j_r} = m' \cdot (\mathcal{T}_{j_1} \circ \psi^{\lambda'}_{\lambda}).
\]
For any irreducible representation $\pi$ of $G'$ of infinitesimal character $\chi_\lambda$, we take a coherent family over $\lambda + \Lambda$ such that $\Theta(\lambda) = \pi$. By Lemma~\ref{lemma: translation by coherent family}, we have
\[
\psi^{\lambda'}_{\lambda}(\pi) = \psi^{\lambda'}_{\lambda} (\Theta(\lambda)) = \sum_{s \in W_{\lambda}/W_{\lambda} \cap W_{\lambda'}} \Theta(s(\lambda')),
\]
and
\begin{align*}
\mathcal{T}_{j_1} (\Theta(s(\lambda'))) 
&= \psi_{\lambda'}^{\mu} (\Theta(s(\lambda'))) = \psi_{s(\lambda')}^{s(\mu)} (\Theta(s(\lambda'))) = \sum_{t \in W_{s(\lambda')}/W_{s(\lambda')} \cap W_{s'(\mu)}} \Theta(ts(\mu)) \\
& = \sum_{t \in W_{s(\lambda')}/W_{s(\lambda')} \cap W_{s'(\mu)}} \Theta(s(s^{-1}ts)(\mu)) = \sum_{t \in W_{\lambda'}/W_{\lambda'} \cap W_{\mu}} \Theta(st(\mu)).
\end{align*}
Hence
\[
\mathcal{T}_{j_1} \circ \psi^{\lambda'}_{\lambda}(\pi) = \sum_{s \in W_{\lambda}/W_{\lambda} \cap W_{\lambda'}} \mathcal{T}_{j_1} (\Theta(s(\lambda'))) = \sum_{s \in W_{\lambda}/W_{\lambda} \cap W_{\lambda'}} \, \sum_{t \in W_{\lambda'}/W_{\lambda'} \cap W_{\mu}} \Theta(st(\mu)).
\]
Note
\(
W_{\lambda'}/W_{\lambda'} \cap W_{\mu} 
\)
can be represented by $\{(j, j_{1}) \in S_n \, | \, j \leqslant j_{1}, \lambda_{j} = \lambda_{j_1} \} \subseteq W_{\lambda}$.
Then the natural inclusion  
\[ 
W_{\lambda'} \cap W_{\lambda} / W_{\lambda'} \cap W_{\lambda} \cap W_{\mu} \rightarrow W_{\lambda'} / W_{\lambda'} \cap W_{\mu}
\]
is a bijection. Hence
\begin{multline*}
\mathcal{T}_{j_1} \circ \psi^{\lambda'}_{\lambda}(\pi) = \sum_{s \in W_{\lambda}/W_{\lambda'} \cap W_{\lambda} \cap W_{\mu}} \Theta(s(\mu)) \\
=  |\frac{W_{\lambda} \cap W_{\mu}}{W_{\lambda'} \cap W_{\lambda} \cap W_{\mu}}| \sum_{s \in W_{\lambda}/W_{\lambda} \cap W_{\mu}} \Theta(s(\mu)) = |\frac{W_{\lambda} \cap W_{\mu}}{W_{\lambda'} \cap W_{\lambda} \cap W_{\mu}}| \psi_{\lambda}^{\mu} (\pi).
\end{multline*}
Since
\[
|\frac{W_{\lambda} \cap W_{\mu}}{W_{\lambda'} \cap W_{\lambda} \cap W_{\mu}}| = |\{1 \leqslant i \leqslant r \, | \, \lambda_{j_i} = \lambda_{j_1}\}|,
\]
then
\[
m = m'  \cdot |\frac{W_{\lambda} \cap W_{\mu}}{W_{\lambda'} \cap W_{\lambda} \cap W_{\mu}}|. 
\]
This finishes the proof.
\end{proof}

It can be slightly generalized as follows.
\begin{corollary}
\label{cor: factorization of translation}
Suppose $\lambda \in \mathfrak{t}^{*}$ is integral and
\[
\mu = \lambda - (e_{j_1} + \cdots + e_{j_r}), \quad j_1 < \cdots < j_r.
\]
Then for any $\sigma \in S_r$ such that $\lambda_{j_{\sigma(1)}} \geqslant \cdots \geqslant \lambda_{j_{\sigma(r)}}$, 
\[
m \cdot \psi^{\mu}_{\lambda} = \mathcal{T}_{j_{\sigma(1)}} \circ \cdots \circ \mathcal{T}_{j_{\sigma(r)}}
\]
in the Grothendieck group, where $m = |{\rm Stab}_{S_{r}}(\lambda_{j_1}, \cdots, \lambda_{j_r})|$.
\end{corollary}

\begin{proof}
Include $S_r$ in $W$ by permuting $\{e_{j_i}\}_{i=1}^r$. Since $\psi^{\mu}_{\lambda} = \psi^{\sigma^{-1}(\mu)}_{\sigma^{-1}(\lambda)}$, then it is equivalent to show that
\[
m \cdot \psi^{\sigma^{-1}(\mu)}_{\sigma^{-1}(\lambda)} = \mathcal{T}_{j_1} \circ \cdots \circ \mathcal{T}_{j_r}.
\]
Let $\lambda' = \sigma^{-1}(\lambda)$ and $\mu'= \sigma^{-1}(\mu)$. Then 
\[
\mu' = \lambda' - (e_{j_1} + \cdots + e_{j_r}), \quad j_1 < \cdots < j_r.
\]
and 
\(
\lambda'_{j_1} \geqslant \cdots \geqslant \lambda'_{j_r}.
\)
Note $m = |{\rm Stab}_{S_{r}}(\lambda_{j_1}, \cdots, \lambda_{j_r})| = |{\rm Stab}_{S_{r}}(\lambda'_{j_1}, \cdots, \lambda'_{j_r})|$. Then it follows from Proposition~\ref{prop: factorization of translation}.
\end{proof}

\begin{lemma}
\label{lemma: switch order}
Suppose $\lambda \in \mathfrak{t}^{*}$ is integral and
\[
\mu = \lambda - (e_{j_1} + e_{j_2}), \quad j_1 < j_2.
\]
If $|\lambda_{j_1} - \lambda_{j_2}| \neq1$, then
\[
\mathcal{T}_{j_1} \circ \mathcal{T}_{j_2} = \mathcal{T}_{j_2} \circ \mathcal{T}_{j_1}
\]
in the Grothendieck group.
\end{lemma}

\begin{proof}
The case $\lambda_{j_1} = \lambda_{j_2}$ follows from Corollary~\ref{cor: factorization of translation}. Now let us assume $|\lambda_{j_1} - \lambda_{j_2}| > 1$. We can choose $w \in W$ such that $w(\mu), w(\lambda)$ are dominant. Then
\[
w(\mu) = w(\lambda) - (e_{w(j_1)} + e_{w(j_2)}),
\]
and it is equivalent to show that
\[
\psi^{w(\mu)}_{w(\lambda)} = \mathcal{T}_{w(j_1)} \circ \mathcal{T}_{w(j_2)} = \mathcal{T}_{w(j_2)} \circ \mathcal{T}_{w(j_1)}.
\]
Therefore, we will further assume that $\lambda, \mu$ are dominant and $\lambda_{j_1} - \lambda_{j_2} > 1$. By Proposition~\ref{prop: factorization of translation}, we already have
\[
\psi^{\mu}_{\lambda} = \mathcal{T}_{j_1} \circ \mathcal{T}_{j_2}.
\]
It remains to show that
\[
\psi^{\mu}_{\lambda} = \mathcal{T}_{j_2} \circ \mathcal{T}_{j_1}.
\]

Let $\lambda' = \lambda - e_{j_1}$. For any irreducible representation $\pi$ of $G'$ of infinitesimal character $\chi_\lambda$, we take a coherent family over $\lambda + \Lambda$ such that $\Theta(\lambda) = \pi$. Following the proof of Proposition~\ref{prop: factorization of translation}, we get
\[
\mathcal{T}_{j_2} \circ \mathcal{T}_{j_1}(\pi) = \sum_{s \in W_{\lambda}/W_{\lambda} \cap W_{\lambda'}} \, \sum_{t \in W_{\lambda'}/W_{\lambda'} \cap W_{\mu}} \Theta(st(\mu)).
\]
Since $\lambda_{j_1} - \lambda_{j_2} > 1$, 
\(
W_{\lambda'}/W_{\lambda'} \cap W_{\mu} 
\)
can be represented by $\{(j, j_{2}) \in S_n \, | \, j_1 < j \leqslant j_{2}, \lambda_{j} = \lambda_{j_2} \} \subseteq W_{\lambda}$.
Then the natural inclusion  
\[ 
W_{\lambda'} \cap W_{\lambda} / W_{\lambda'} \cap W_{\lambda} \cap W_{\mu} \rightarrow W_{\lambda'} / W_{\lambda'} \cap W_{\mu}
\]
is a bijection. Hence
\begin{multline*}
\mathcal{T}_{j_2} \circ \mathcal{T}_{j_1}(\pi)  = \sum_{s \in W_{\lambda}/W_{\lambda'} \cap W_{\lambda} \cap W_{\mu}} \Theta(s(\mu)) \\
=  |\frac{W_{\lambda} \cap W_{\mu}}{W_{\lambda'} \cap W_{\lambda} \cap W_{\mu}}| \sum_{s \in W_{\lambda}/W_{\lambda} \cap W_{\mu}} \Theta(s(\mu)) = |\frac{W_{\lambda} \cap W_{\mu}}{W_{\lambda'} \cap W_{\lambda} \cap W_{\mu}}| \psi_{\lambda}^{\mu} (\pi).
\end{multline*}
At last, $W_{\lambda} \cap W_{\mu} = W_{\lambda'} \cap W_{\lambda} \cap W_{\mu}$. This finishes the proof.

\end{proof}

We also consider the following generalization of $\mathcal{T}_{i}$. For any segment $[a, b]$ in $\{1, \cdots, n\}$ and $\lambda \in \mathfrak{t}^{*}$, we define 
\[
\mathcal{T}_{[a, b]} := \psi_{\lambda}^{\lambda - \sum_{i = a}^{b} e_i}.
\] 
By Corollary \ref{cor: factorization of translation}, we have 
\begin{align}
\label{eq: factorization of translation}
\mathcal{T}_{[a, b]} = \mathcal{T}_{a} \circ \cdots \circ \mathcal{T}_{b}
\end{align}
when $\lambda_a > \cdots > \lambda_b$. 

\begin{proposition}
\label{prop: Trapa translation}
In the setup of \textsection \ref{subsec: singular-case}, let 
\[
\mathcal{T} := \mathcal{T}^{t_1 - t_2}_{[1, \cdots, m_1]} \circ \mathcal{T}^{t_2 - t_3}_{[1, m_1 + m_2]} \circ \cdots \mathcal{T}^{t_r}_{[1, n]}.
\]
Then
\[
\pi(\psi^{\mathbb{R}}; \ul, \ueta) = \mathcal{T}(\pi(\psi^{\mathbb{R}}_{\gg}; \ul, \ueta)).
\]
\end{proposition}

\begin{proof}
The is a reinterpretation of \cite[Lemma 3.13]{Trapa:2001}).
\end{proof}

\begin{lemma}
\label{lemma: comparison of translation}
In the setup of \textsection \ref{subsec: singular-case}, we have
\[
\circ_{i = 1}^{r} \, (\mathcal{T}_{1 + \sum_{j < i} m_j} \circ \cdots \circ \mathcal{T}_{\sum_{j \leqslant i} m_j})^{t_i} \, \pi(\psi^{\mathbb{R}}_{\gg}; \ul, \ueta) =  \mathcal{T}(\pi(\psi^{\mathbb{R}}_{\gg}; \ul, \ueta)).
\]
\end{lemma}

\begin{proof}
It follows from \eqref{eq: factorization of translation} and Lemma~\ref{lemma: switch order}.
\end{proof}

Suppose $\lambda \in \mathfrak{t}^*$ is integral and $1 \leqslant j \leqslant n$. In the application of the translation functor $\mathcal{T}_j$, we can always assume that $\lambda$ is dominant and  either $\lambda_j > \lambda_{j+1}$ or $j = n$ by taking $W$-conjugation (cf. \eqref{eq: Weyl conjugate}). This is also the setup of Theorem~\ref{thm: translation and derivative}. Let $\lambda_{-} = \lambda - e_{j}$. In order to find $(\psi^{\lambda_{-}}_{\lambda})^*$ (cf. Theorem~\ref{thm: translation vs pushpull}), we need to decompose $\psi^{\lambda_{-}}_{\lambda}$ as follows. 
\begin{lemma}
\label{lemma: second decomposition}
In the above setup, let $\lambda' = \lambda + (e_1 + \cdots + e_{j-1})$. Then
\[
\psi^{\lambda_{-}}_{\lambda} = \psi^{\lambda_{-}}_{\lambda'} \circ \psi^{\lambda'}_{\lambda}
\]
in the Grothendieck group.
\end{lemma}

\begin{proof}
For any irreducible representation $\pi$ of $G'$ of infinitesimal character $\lambda$, we take a coherent family $\Theta$ of virtual representations of $G'$ over $\lambda + \Lambda$ such that $\Theta(\lambda) = \pi$. By Lemma~\ref{lemma: translation by coherent family},
\[
\psi^{\lambda_{-}}_{\lambda} (\Theta(\lambda)) = \sum_{s \in W_{\lambda}/W_{\lambda} \cap W_{\lambda_{-}}} \Theta(s(\lambda_{-})).
\]
This expression is equal to $(\psi_{\lambda'}^{\lambda_-} \circ \psi_{\lambda}^{\lambda'})(\Theta(\lambda))$ by the following general fact.
\end{proof}

\begin{proposition} 
Let $F$ be a facet of $\mathfrak{t}^{*}$ and $F', F''$ be two facets in the closure of $F$. Suppose $\lambda \in F, \lambda' \in F', \lambda'' \in F''$ are all integral. For any coherent family $\Theta$ of virtual representations of $G'$ over $\lambda + \Lambda$, we have
\begin{equation}\label{eqn:trans-srs}
\psi^{\lambda''}_{\lambda} \circ \psi^{\lambda}_{\lambda'}(\Theta(\lambda'))= |W_{\lambda'} \cap W_{\lambda''}/W_{\lambda}| \sum_{s \in W_{\lambda'}/W_{\lambda'} \cap W_{\lambda''}} \Theta(s(\lambda''))
\end{equation}
\end{proposition}

\begin{proof}
    By a slight modification of \cite[Proposition 7.2.22(b)]{Vogan:1981}, 
	\begin{equation*}
		\psi^{\lambda}_{\lambda'}(\Theta(\lambda'))
		= \sum_{s \in W_{\lambda'}/W_{\lambda}} \Theta(s \lambda).
	\end{equation*}
	Moreover, by Lemma 7.3.1 and Proposition 7.2.22(a) of \textit{loc. cit.},
	\begin{equation*}
		\psi_{\lambda}^{\lambda''}(\Theta(s \lambda))
		= \psi_{s\lambda}^{s\lambda''}(\Theta( s \lambda)) 
		= \Theta(s \lambda'').
	\end{equation*}
	Therefore
	\begin{equation*}
		\psi^{\lambda''}_{\lambda} \circ \psi^{\lambda}_{\lambda'}(\Theta(\lambda'))
		= \sum_{s \in W_{\lambda'}/W_{\lambda}} \Theta(s \lambda'').
	\end{equation*}
	Since $\Theta(s \lambda'') = \Theta(st\lambda'')$ for any $t \in W_{\lambda'} \cap W_{\lambda''}$, the right hand side is equal to the right side of (\ref{eqn:trans-srs}). 
\end{proof}

To complete the proof of Lemma~\ref{lemma: second decomposition}, it suffices to notice that $|W_{\lambda} \cap W_{\lambda_{-}}/W_{\lambda'}| = 1$.
\section{Langlands parameter for Adams-Johnson packet}\label{sec: Langlands}
	In this section we will prove Theorem \ref{Langlands-sub}.
	The construction of $\pi(\psi^\R; \underline p, \underline q)$ uses cohomological induction $A_\LAq(\Lambda)$. We will begin by reviewing the general result in
	\cite[Corollary 11.219]{KV:1995} for obtaining the Langlands quotient parameter of $A_\LAq(\Lambda)$ in the good range. Then we will implement this result for $\pi(\psi^\R; \underline p, \underline q)$ in the rest of this appendix. In particular, we will follow \cite{Adams:2011} for obtaining the complete Langlands parameters of discrete series representations of $U(p,q)$.

\subsection{Langlands quotient parameter for $A_\LAq(\Lambda)$}\label{subsec: KV}
	In this subsection we review the calculation of the Langlands quotient parameter for $A_\LAq(\Lambda)$ in the good range following \cite{KV:1995}.
	
	Let $G$ be a connected real reductive group and $\mathfrak{g}$ be its complexified Lie algebra.
    Let $\LAq$ be a $\theta$-stable parabolic subalgebra of $\g$, $\l$ and $\u$ be its Levi subalgebra and nilpotent radical respectively. Let $\Lambda$ be a one dimensional character of $L$, which will be identified with its differential in $\mathfrak{l}^*$. Take a Cartan subalgebra $\h$ of $\l$, and a positive root system $\Delta^+(\l, \h) \subseteq \Delta(\l, \h)$.
	Together with $\Delta(\u, \h)$ in the decomposition $\LAq = \l \oplus \u$, we get a positive root system $\Delta^+(\g, \h):= \Delta^+(\l, \h) \sqcup \Delta(\u, \h)$ for $(\g, \h)$.
	Denote the half sum of $\Delta^+(\g, \h)$ by $\delta_G$.
	We say $(\LAq, \Lambda)$ is in the \textit{good range}, or $\Lambda$ lies in the \textit{good range} for $\LAq$, if $\forall \alpha \in \Delta(\u, \h)$,
	$\Re\la \Lambda|_{\mathfrak{h}} + \delta_G, \alpha \ra >0$. This definition does not depend on the choice of $\h$ or $\Delta^+(\l, \h)$, since $\Lambda$ is invariant under the adjoint action of $L_{\mathbb{C}}$.
	
\begin{theorem}[{\cite[Corollary 11.219]{KV:1995}}]\label{thm: Langlands-quotient}
	Suppose $\Lambda$ lies in the good range for $\LAq$, then $A_\LAq(\Lambda)$ is irreducible.
	The Langlands quotient parameter $(MAN, \xi, \nu)$ for $A_\LAq(\Lambda)$ is described as follows:
	let $\h_0 = \t_0 \oplus \a_0$ be a maximally split $\theta$-stable Cartan subalgebra of the real Lie algebra $\l_0$ of $L$,
	$\Delta^+(\l, \h)$ be a positive root system of $\l$ taking $\a_0$ before $\sqrt{-1} \t_0$, 
    and
	$\Delta^+(\g, \h) = \Delta^+(\l, \h) \sqcup \Delta(\u, \h)$, then
	\begin{itemize}
	\item $A = \exp \a_0$, $\nu =$ the half sum of positive restricted roots of $\l$;
	\item $MA = Z_G(\a_0)$;
	\item $\xi$ is the disceret series representation of $M$ associated with Harish-Chandra parameter $(\Lambda|_{\mathfrak{h}} + \delta_G)|_{\t}$; 
	\item $N$ is determined by its complex Lie algebra $\mathfrak{n}$ such that $\Delta(\n, \h)$ contains $\Delta^+_{\mathrm{real}}(\l, \h)$, and $\nu$ is dominant with restrict to $\Delta(\mathfrak{n}, \mathfrak{a})$.
	\end{itemize}
\end{theorem}
	Here, we say $\Delta^+(\l, \h)$ takes $\a_0$ before $\sqrt{-1} \t_0$, if there are spanning sets $\{E_1, \cdots, E_r\} \subset \a_0$, $\{E_{r+1}, \cdots, E_s\} \subset \sqrt{-1} \t_0$, such that $\Delta^+(\l, \h)$ is defined by the lexicographic ordering with respect to $\{E_i\}$:
	$\alpha \in \Delta(\l, \h)$ is positive iff there is an index $i$ such that $\la \alpha, E_j \ra = 0$ for $1 \leqslant j \leqslant i-1$, and $\la \alpha, E_i \ra >0$.
	
	In the case of $\pi(\psi^\R; \underline p, \underline q)$, $L \cong U(p_1, q_1) \times \cdots \times U(p_r, q_r)$. Then $A$ is a product of maximal split torus of $U(p_d, q_d)$ for $1 \leqslant \bk \leqslant r$. In particular, $A \cong \R^l_{>0}$ and $M \cong U(p-l, q-l) \times (S^1)^l$, where $l= l_1+ \cdots +l_r$ and $l_\bk = \min\{p_\bk, q_\bk\}$. The discrete series representation $\xi$ of $M$ can be written as a tensor product $\sigma \boxtimes \chi$, where $\sigma$ is a discrete series representation of $U(p-l, q-l)$ and $\chi$ is a character of $(S^1)^l$.
    Next, we would like to specify a maximal split torus in $U(p_d, q_d)$ and compute $\nu$ (see \ref{subsec: structure-of-U}); describe the complete Langlands parameter of $\sigma$ (see \ref{subsec: discrete-series}). All these results will be put together in \ref{subsec: calculation} and \ref{subsec: complete-para} to give the complete Langlands parameter of $\pi(\psi^\R; \underline p, \underline q)$.

\subsection{Maximal split torus in $U(p,q)$}\label{subsec: structure-of-U}
	Let $n = p+q$ and $l = \min\{p, q\}$. Any maximal split torus $A$ of $G = U(p, q)$ is isomorphic to $\R_{>0}^{l}$.
	We specify one by describing its Lie algebra $\a_0$.
	
	First, let us introduce some terminology.
	For $I \subseteq \{1, \cdots, n\}$ and $\nu=(\nu_1, \cdots, \nu_n) \in \C^n$, by an $I$ block of $\nu$ we mean a tuple $\nu_I : = (\nu_i)_{i\in I}$.
	For $I, J \subseteq \{1, \cdots, n\}$ and $X = (x_i^j)\in \gl_{n}$,
	by an $I\times J$ block of $X$ we mean its submatrix
	$X_I^J: = ( x_i^j )_{i\in I}^{j \in J}$.
	
	Let $\LI= \{1, \cdots, l\}$, $\MI = \{l+1, \cdots, n-l\}$, and $\UI = \{ n-l+1, \cdots, n\}$ in this subsection. 
	Then we take $A$ such that its Lie algebra  $\a_0$ consists of $X \in \gl_n$ with  
	\[X_\LI^\UI =\begin{pmatrix}
		&	&a_1\\
		&\iddots	&	\\
	a_l 	&	&	
	\end{pmatrix},\quad
	X_\UI^\LI=\begin{pmatrix}
		&	&a_l\\
		&\iddots	&	\\
	a_1 	&	&	
	\end{pmatrix},
	\textrm{ with }a_i \in \R,\]
	and zero entries outside its $\LI\times \UI$ and $\UI \times \LI$ blocks.
	For such $A$, its centralizer $Z_G(A)$ is equal to $MA$ with $M \cong U(p-l, q-l) \times (S^1)^l$.
	More explicitly, each $X \in M$ has
	\[X_\LI^\LI = \begin{pmatrix}
	e^{c_1\sqrt{-1} }	&	&	\\
		&\ddots 	&	\\
		&	&e^{c_{l}\sqrt{-1}}
	\end{pmatrix},\quad
	X_\UI^\UI =\begin{pmatrix}
	e^{c_{l}\sqrt{-1} }	&	&	\\
		&\ddots 	&	\\
		&	&e^{c_1\sqrt{-1}}
	\end{pmatrix},
	\textrm{ with } c_i\in \R,\]
	$X_\MI^\MI \in U(p-l, q-l),$
	and zero entries outside its $\LI \times \LI$, $\MI \times \MI$, and $\UI \times \UI$ blocks.

	We now calculate the half sum of positive restricted roots for $\a$ in $\gl_n$.
	Let $\t \subseteq \gl_n$ be the standard Cartan subalgebra consisting of diagonal matrices, and 
	$\b \subseteq \gl_n$ be the standard Borel subalgebra consisting of upper triangular matrices.
	They determine a positive root system $\{e_i- e_j \mid 1 \leqslant i< j \leqslant n\} \subseteq \Delta(\g, \t)$,
	whose half sum $\delta_G \in \t^* \cong \C^n$ takes the form $( \frac{n+1}2 -1, \cdots, \frac{n+1}2-n)$.
	Note that $\t$ is compact with respect to the real form $U(p,q)$.
	Consider the Cayley transforms with respect to the non-compact roots $e_i - e_{n+1- i}$, $i \in \LI$.
	They are strongly orthogonal to each other, so we can compose these Cayley transforms together, denote it by $\mathbf c$, and obtain a new Cartan $\h : = \mathbf c(\t)$.
	According to \cite[(6.66)]{Knapp96}, $\h = \t' \oplus \a$ can be described as follows.
	Let $E_i^j$ be the elementary matrix with $(i, j)$-entry $1$ and other entries $0$.
	The compact Cartan $\t$ decompose into $\t^\Delta \oplus \t_\MI \oplus \t^\nabla$ with
	\begin{itemize}
	\item $\t^\Delta$ spanned by $E_i^i + E_{n+1 -i}^{n+1 -i}$, $i \in \LI$,
	\item $\t_\MI$ spanned by $E_i^i$, $i \in \MI$, 
	\item $\t^\nabla$ spanned by $E_i^i - E_{n+1 -i}^{n+1 -i}$, $i \in \LI$.
	\end{itemize}
	Then $\mathbf c$ is identity on $\t' := \t^\Delta \oplus \t_\MI$, and sends $\t^\nabla$ to $\a$.
	More explicitly, $\mathbf c(E_i^i - E_{n+1 -i}^{n+1 -i}) = E_i^{n+1 -i} + E_{n+1 -i}^i$.
	Via the Cayley transform $\mathbf c$, the positive root system $\Delta^+(\g, \t) \subseteq \Delta(\g, \t)$ is sent to a positive root system $\Delta^+(\g, \h) \subseteq \Delta(\g, \h)$.
	One can also verify that $\Delta^+(\g, \h)$ takes $\a$ before $\t'$.
	Hence, the half sum of $\Delta^+(\g, \a)$, denoted by $\nu$, can be calculated by restricting the half sum of $\Delta^+(\g, \h) = \Delta^+(\g, \t) \circ \mathbf c^{-1}$ to $\a$.
	Under the maps
	\[\begin{tikzcd}
	\C^n \arrow[r]                          &
	\t \arrow[r, "\mathbf c"] 	& 
	\h=\t' \oplus \a       &
	\C^l \arrow[l, hook']                   \\
	{(t_1, \cdots, t_n)} \arrow[r, maps to] & 
	\begin{pmatrix} 
		t_1&	& \\ 
		& \ddots & \\ 
		& & t_n 
	\end{pmatrix}     &
	\begin{pmatrix} 
			&	&	&	& a_1	\\
			&	&	&\iddots	&	\\
			& 	&0	&	&	\\
			&\iddots	&	&	&	\\
		a_1	&	&	&	&
	\end{pmatrix}	& 
	{(a_1, \cdots, a_l)} \arrow[l, maps to]
	\end{tikzcd},\]
	$\delta_G \in \t^* \cong \C^n$ is restricted to $( n-1, \cdots, n+1-2l) \in \C^l \cong \a^*$, which represents the desired $\nu$.

\subsection{Discrete series of $U(p, q)$}\label{subsec: discrete-series}
\begin{subequations}
	In this subsection we review the complete Langlands parameter attached to a discrete series representation of $G = U(p,q)$ where $n = p + q$. We mainly follow \cite[section XI.8]{KV:1995} for the Harish-Chandra parametrization and cohomological induction construction for the discrete series of a general connected real reductive group $G$.
	If $G$ has a discrete series, then it must have a compact Cartan subgroup $T$.
	We fix a maximal compact subgroup $K$ of $G$, and assume $T\subseteq K$.
	For convenience, choose a Borel subalgebra $\b \subseteq \g$ containing $\t$.
	It determines a positive root system $\Delta^+(\g, \t) \subseteq \Delta(\g, \t)$, and we denote the half sum of positive roots by $\delta_G$.
	A regular dominant weight $\mu \in \t^*$ represents an infinitesimal character of discrete series of $G$ iff $\mu- \delta_G \in \t^*$ lifts to a character of $T$.
	In this case, for each $w\in W(\g, \t)$ there is a discrete series $\pi_{w\mu}$ with infinitesimal character $\mu$ realized by a cohomologically induced module.
	We call $w \mu$ the Harish-Chandra parameter of $\pi_{w\mu}$.
	Two discrete series $\pi_{w\mu}$ and $\pi_{w'\mu}$ are isomorphic iff $w' w^{-1}$ lies in $W(\k, \t)$,
	where $\k$ is the complex Lie algebra of $K$.
	These results are taken from \cite[Theorem 11.178]{KV:1995}.
	Consequently, the set of discrete series with infinitesimal character $\mu$ is in bijection with $W(\k, \t)\backslash W(\g, \t)$.
	
	The set $\{\pi_{w\mu} \mid w \in W(\k, \t) \backslash W(\g, \t)\}$ forms an L-packet $\Pi_{\phi^{\R}}(G)$ attached to some discrete L-parameter $\phi^{\R}: W_\R \to \D  G \rtimes {\rm Gal}(\C/\R)$.
	According to \cite[Section 5]{Adams:2011}, $\phi$ is determined by
	\begin{equation}\label{eq: Langlands-parameter}
	{BC} (\phi^{\R}): \C^\times = W_\C \to \D G,\quad 
	z \mapsto (\frac z {\bar z})^\mu
	\end{equation}
	where $\mu \in \delta_G + X^*(T) \subset \frac 1 2 X^*(T)$ is regarded as an element in $\frac 1 2 X_*(\D  T)$.
	Moreover, the adjoint action of $\phi(j) \in \D G \rtimes W_\R$ on $\D T \subseteq \D G$ is $t \mapsto t^{-1}$.
	Since $\mu$ is regular, the centralizer of $\Im \phi^{\R} \subset \D  G\rtimes W_\R$ in $\D G$ is the finite group $\{ t \in \D  T \mid t^2 =1\}$.
	Consequently, $A_{\phi^{\R}} = S_{\phi^{\R}} \cong \{\pm 1\}^n$, where $n$ is the rank of $G$.
	
	There is a parametrization $\Pi_{\phi^{\R}}(G) \to \D  S_{\phi^{\R}}$ specified by certain properties;
	the most important one is that  the generic representation in L-packet is sent to the trivial representation of $S_{\phi^{\R}}$. Such a parametrization is given in \cite[Section 7]{Adams:2011} (for real forms $G$ that are inner to a compact form). To be concise, we avoid the general setting, and quote the result for $G= U(p,q)$ in loc.cit. Section 9.
	One has to simultaneously consider all the pure inner forms of $G$.
	More explicitly, take 
	\[\Pi^{\rm pure}_{\phi^{\R}}(G) := \bigsqcup_{p+q =n} \Pi_{\phi^{\R}}(U(p,q)),\]
	where $U(p,q)$ and $U(q, p)$ represent different pure inner forms when $p \neq q$.
	Following \cite[Section 2]{Adams:2011}, we let $z = ((-1)^{n-1}, \cdots, (-1)^{n-1}) \in (S^1)^n \cong T$, $\mathcal X_1(z) := \{ t \in T \mid t^2 = z \}$, and 
	\[x_\bp = \left(\sqrt{-1} \right)^{n-1} \begin{pmatrix}
	(-1)^n	&	&	\\
		&\ddots&	\\
		&	&-1
	\end{pmatrix},\quad
	x_{p, q} = (\sqrt{-1})^{n-1}
	\begin{pmatrix}
	I_p 	&	 \\
		& -I_q
	\end{pmatrix}.\]
	Then we have
	\begin{itemize}
	\item a bijection $\Pi^{\rm pure}_{\phi^{\R}}(G) \to \mathcal X_1(z)$ sending $\pi_{w\mu} \in \Pi_{\phi^{\R}}(U(p,q))$ to $(-1)^n w^{-1}x_{p, q}$,  
	
	\item a bijection $\D  S_{\phi^{\R}} \cong \{\pm 1\}^n \to \mathcal X_1(z)$ sending
	$\epsilon= (\epsilon_1, \cdots, \epsilon_n )$ to
	$\epsilon x_\bp$.
	\end{itemize}
	The first bijection follows from \cite[(5.6)]{Adams:2011};
	here the action of $w \in S_n$ on $\mu=(\mu_1, \cdots, \mu_n) \in \C^n$ is $w\mu:= ( \mu_{w^{-1}(1)}, \cdots, \mu_{w^{-1}(n)})$.
	The second bijection follows from \cite[(6.10) and (9.3)]{Adams:2011}.
	Consequently, in the complete Langlands parameter $(\phi, \epsilon)$ of $\pi_{w\mu}$, the element $\epsilon \in \D A_{\phi^{\R}} \cong \{\pm 1\}^n$ is
	\begin{equation}\label{component-character}
	\epsilon_i = \left\{\begin{aligned}
	&(-1)^{n-i},	&w(i) \leqslant p,\\
	&(-1)^{n+1-i},	&w(i)>p.
	\end{aligned}\right.
	\end{equation}
\end{subequations}
\begin{remark}
	Our choice of base point $x_\bp$ differs from Adam's $x_b$ by a sign when $n$ is odd, which amounts to a different choice of the generic representation: 
	the bijection $\D  S_{\phi^{\R}} \to \Pi_{\phi^{\R}}^{\rm pure}(G)$ in \cite[Corollary 9.6]{Adams:2011} sends trivial character $1_{S_{\phi^{\R}}}$ to the generic representation of $U(\frac{n+1} 2, \frac{n-1} 2)$,
	while ours sends $1_{S_{\phi^{\R}}}$ to the generic representation of $U(\frac{n-1}2, \frac{n+1}2)$.
	These two generic representations can be identified via the isomorphism between $U(\frac{n+1}2,\frac{n-1}2)$ and $U(\frac{n-1}2, \frac{n+1}2)$, but are regarded as different elements in $\Pi_{\phi^{\R}}^{\rm pure}(G)$.
\end{remark}

\subsection{Explicit Langlands quotient parameter for $\pi(\psi^\R; \underline p, \underline q)$}\label{subsec: calculation}
	Let  $\psi^\R$ be an A-parameter of $U(p,q)$, which has integral regular infinitesimal character with good parity.
	For a representation in $\Pi_{\psi^\R}(U(p,q))$, we will calculate its standard module, from which one can read off its complete Langlands parameter easily.	
	 As in Section \ref{subsec: main}, decompose the base change of ${\psi^\R}$ as
	\[{BC}(\psi^\R) = \bigoplus_{i=1}^r (\frac z {\bar z})^{k_i/2} \boxtimes \FDR_{m_i}, \quad k_i, m_i \in \mathbb{Z}
    \]
	and assume
	\begin{itemize}
	\item (parity) $\forall i=1, \cdots, r$, $k_i + m_i$ has the same parity with $n =p+q$, and
	
	\item (regularity) $\forall i =1, \cdots, r-1$, $k_i- m_i \geqslant k_{i+1}+ m_{i+1}$.
	\end{itemize}
	 As mentioned in \ref{cls: psiR to psip}, the infinitesimal character of $\psi^\R$ is represented by
	\[\lambda := \left(\frac{k_1} 2 + \frac{m_1 -1 } 2, \cdots, \frac{k_1} 2 - \frac{m_1 -1} 2, 
		\cdots,
		\frac{k_r}2 + \frac{m_r -1} 2, \cdots, \frac{k_r}2 - \frac{m_r -1} 2 \right)
	\in \C^n.\]
	The parity and regularity condition implies that $\lambda = (\lambda_1, \cdots, \lambda_n)$ is a decreasing sequence with entries in $\frac{n-1}2 + \BZ$,
	so we are in the setup of Section \ref{sec: reduction}. 
	    
	The pure Arthur packet $\Pi_{\psi^\R}^\AV(G)$ is parametrized  by
	\[\MRV:=\{(\underline p, \underline q) \in \N^r \times \N^r \mid p_i + q_i= m_i \} \cong
	 \frac{\{( \underline l, \underline \eta ) \in \N^r \times \{ \pm 1\}^r  \mid l_i \leqslant \frac {m_i} 2\}}{(\frac {m_i} 2, +) \sim ( \frac {m_i} 2, -)} =: \ES.\]
	Recall the bijection 
	\[W(U(p,q), T) \backslash W(G_{\C}, T_\C) / W(L_\C, T_\C) \cong \MRV \cong \ES\]
	in Section \ref{sec: reduction}.
	In this subsection we fix a parameter $(\underline p, \underline q)$ (whence a pair $(\underline l, \underline \eta)$), and
	choose a representative $w\in S_n = W(G_{\C}, T_\C)$ for its corresponding class in the double coset.
	Let's specialize Theorem \ref{thm: Langlands-quotient} to $\pi({\psi^\R}; \underline l, \underline \eta) = \pi({\psi^\R}; \underline p, \underline q) = A_{w(\LAq)}(w\Lambda_{L, \psi^\R})$.    
    
	Recall the definition of $w(\LAq)$ and $w \Lambda_{L, \psi^\R}$.
	We have defined a subgroup $L \subseteq U(n,0)$ with complex Lie algebra $\l = \gl_{m_1} \times \cdots \times \gl_{m_r}$,
	which is the Levi factor of a standard parabolic subalgebra $\LAq \subset \gl_{n}$.
	Then $L_w \subseteq U(p,q)$ is taken to satisfy $L_w(\C) = (w L)(\C)$.
	It follows that $L_w \cong U(p_1, q_1) \times \cdots \times U(p_r, q_r)$, and
	the character $w\Lambda_{L, \psi^\R}$ takes the form $\det^{a_1} \boxtimes \cdots \boxtimes \det^{a_r}$ under this isomorphism, where
	\[a_\bk = \frac{k_\bk +m_\bk - n}2 + \sum_{\bh< \bk} m_\bh.\]
	To ease the notations, we will denote
	\begin{itemize}
	\item $w(\LAq)$ by $\LAq'$;
	\item $L_w$ by $L'$;
	\item $\l_w$ (complex Lie algebra of $L_w$) by $\l'$;
	\item the differential of $w \Lambda_{L, \psi^\R} |_{T}$ by $\Lambda' \in \t^*$.
	\end{itemize}
	Following Subsection \ref{subsec: KV}, denote the nilpotent radical of $\LAq'$ by $\u'$.
	We will use these notations throughout this subsection.

\subsubsection{Partition of indices}\label{subsubsec: partition}
\begin{subequations}
   Take $\PI_\bk , \QI_\bk \subseteq \{1, \cdots, n\}$ such that $|\PI_\bk|= p_\bk$, $|\QI_\bk|= q_\bk$, and $\PI_1< \cdots <\PI_\bk < \QI_\bk < \cdots <\QI_1$.
	Here by $I < J$ we mean all elements in $I$ are smaller than any element in $J$.
	Then our requirements force
	\[\PI_\bk = \{p_1 + \cdots +p_{\bk-1} < i \leqslant p_1 + \cdots + p_\bk\},\quad
	\QI_\bk = \{n-(q_1+\cdots + q_\bk) < i \leqslant n-(q_1 + \cdots + q_{\bk-1}) \}.\]
	We can choose $w \in S_n$ in the double coset so that it sends $\{m_1+\cdots+m_{d-1}, \cdots, m_1+\cdots +m_d\}$ to $\PI_d \sqcup \QI_d$ in an order-preserving way.
	The isomorphism $L \xrightarrow{\simeq} U(p_1, q_1) \times \cdots U(p_r, q_r)$ can be realized by
	\[X \mapsto \left( X_{\PI_1 \sqcup \QI_1}^{\PI_1 \sqcup \QI_1}, \cdots, X_{\PI_r \sqcup \QI_r}^{\PI_r \sqcup \QI_r}\right ).\]
	There is another partition $\PI_\bk \sqcup \QI_\bk = \LI_\bk \sqcup \MI_\bk \sqcup \UI_\bk$ such that $|\LI_\bk| = l_\bk = |\UI_\bk|$, and $\LI_\bk < \MI_\bk < \UI_\bk$. 
	It's clear that $\LI_\bk \subseteq \PI_\bk$, $\UI_\bk \subseteq \QI_\bk$, and $\MI_\bk$ is contained in either $\PI_\bk$ or $\QI_\bk$, determined by whether $p_\bk > q_\bk$ or not.
	Let $n_\bk =n+ \sum_{\bh<\bk}p_\bh - q_\bh$, then  $i \mapsto n_\bk +1 -i$ gives rise to an order-reversing bijection $\LI_\bk \to \UI_\bk$.

\subsubsection{The split torus $A$ and its character $\nu$}	
	We first would like to choose a maximally split $\theta$-stable Cartan subalgebra $\h_0 = \t_0' \oplus \a_0$ of $\l'_0$,
	and choose a positive root system $\Delta^+(\l', \h)$ taking $\a_0$ before $\sqrt{-1} \t_0'$.
	Since $L' \cong U(p_1, q_1) \times \cdots \times U(p_r, q_r)$, we can make a choice for each $U(p_\bk, q_\bk)$ as in Subsection \ref{subsec: structure-of-U}, and then take direct product.
	Take the following spaces:
	\begin{itemize}
	\item $\t^\Delta = \t_1^\Delta \oplus \cdots \oplus \t_r^\Delta,$ with $\t_\bk^\Delta$ spanned by $E_i^i + E_{n_\bk+1-i}^{n_\bk+1-i}$, $i \in \LI_\bk$,
	\item $\t_\MI = \t_{\MI_1} \oplus \cdots \oplus \t_{\MI_r}$, with $\t_{\MI_\bk}$ spanned by $E_i^i$, $i \in \MI_\bk$, and
	\item $\t^\nabla = \t_1^\nabla \oplus \cdots \t_r^\nabla$, with $\t_\bk^\nabla$ spanned by $E_i^i - E_{n_\bk+1-i}^{n_\bk+1-i}$, $i \in \LI_\bk$;
	\item $\a = \a_1 \oplus \cdots \oplus \a_r$, with $\a_\bk$ spanned by $E_i^{n_\bk+1-i} + E_{n_\bk+1-i}^i$, $i \in \LI_\bk$.
	\end{itemize}
	Then $\t = \t^\Delta \oplus \t_\MI \oplus \t^\nabla$ is a compact Cartan subalgebra of $\l$. Let
	$\h := \t' \oplus \a$ with $\t' = \t^\Delta \oplus \t_\MI$ and
	$A = \exp \a_0$.

	Choose $\Delta^+(\l', \t) =\{
	e_i - e_j |  i<j \in \PI_\bk\sqcup \QI_\bk\}$.
	The composition of Cayley transforms, with respect to the non-compact roots $e_i - e_{n_\bk+1 -i}$, $i \in \LI_\bk$, would send $\t = \t'  \oplus \t^\nabla$ to $\h = \t' \oplus \a$. 
	Via this Cayley transform, the positive root system $\Delta^+(\l', \t) \subseteq \Delta(\l', \t)$ is sent to the positive root system $\Delta^+(\l', \h) \subseteq \Delta(\l', \h)$, 
	which takes $\a$ before $\t'$.
	Hence, $\Delta^+(\l', \a) \subseteq \Delta(\l', \a)$ coincides with the non-zero restriction from $\Delta^+(\l', \h) \subseteq \Delta(\l', \h)$, and
	the half sum $\nu$ of $\Delta^+(\l', \a)$ coincides with the restriction of the half sum $\delta_{L'}$ of $\Delta^+(\l', \h)$.	
	Following the calculation in subsection \ref{subsec: structure-of-U}, we know $\nu \in \a^*$ takes the form
	\begin{equation}\label{eq: nu}
	(m_1 - 1, \cdots, m_1+1 - 2l_1; \cdots; m_r - 1, \cdots, m_r+1 - 2l_r)
	\end{equation}
	 under $\a =  \a_1 \oplus \cdots \oplus \a_r \cong \C^{l_1+ \cdots + l_r}$.

\subsubsection{The group $M$}\label{subsubsec: M}
	The centralizer $Z_G(A)$ is equal to $MA$ with $M \cong U(p-l, q-l) \times (S^1)^{l}$.
	To be more explicit, let $\MI= \MI_1 \sqcup \cdots \sqcup \MI_r$ with $\MI_\bk$ as in Subsubsection \ref{subsubsec: partition} . 
	Then each $X \in M$ has
	\[X_{\LI_\bk}^{\LI_\bk} = \begin{pmatrix}
	e^{c_1\sqrt{-1} }	&	&	\\
		&\ddots 	&	\\
		&	&e^{c_{l_\bk}\sqrt{-1}}
	\end{pmatrix}, \quad
	X_{\UI_\bk}^{\UI_\bk} = \begin{pmatrix}
	e^{c_{l_\bk}\sqrt{-1} }	&	&	\\
		&\ddots 	&	\\
		&	&e^{c_1\sqrt{-1}}
	\end{pmatrix},
	\textrm{ with } c_i \in \R,\]
	$X_\MI^\MI \in U(p-l, q-l)$, and zero entries outside the $\LI_\bk \times \LI_\bk$, $\UI_\bk \times \UI_\bk$, and $\MI \times \MI$ blocks ($\bk =1, \cdots, r$).
	Moreover, $\t' = \t^\Delta \oplus \t_\MI$ is a compact Cartan subalgebra of $\m$, where
	\begin{itemize}
	\item $\t^\Delta = \t_1^\Delta \oplus \cdots \oplus \t_r^\Delta$ is identified with complex Lie algebra of $(S^1)^{l_1}\times \cdots \times (S^1)^{ l_r} \subseteq M$,
	\item $t_\MI = \t_{\MI_1} \oplus \cdots \oplus \t_{\MI_r}$ is identified with complex Lie algebra of a compact Cartan subalgebra of $U(p-l, q-l) \subseteq M$.
	\end{itemize}

\subsubsection{The discrete series $\xi$}
     Let $\delta'_G$ be the half sum of $\Delta^+(\g, \t)' := \Delta^+(\l', \t) \sqcup \Delta(\u', \t)$. Since the Cayley transform from $\t$ to $\h$ is identity on $\t'$, then the discrete series $\xi$ of $M$ has Harish-Chandra parameter $(\Lambda' + \delta'_G)|_{\t'} \in (\t')^*$.
	Let's write this parameter more explicitly.
    
	Under the identification $\t^* \cong \C^n$, the $\PI_\bk \sqcup \QI_\bk$ block of $\Lambda'$ has all entries equal to 
	\[a_\bk = \frac{k_\bk +m_\bk - n}2 + \sum_{\bh< \bk} m_\bh.\]
	Note
	\begin{itemize}
	\item $\Delta^+(\l',\t) =\bigsqcup_\bk\{e_i - e_j \mid  i<j \in \PI_\bk\sqcup \QI_\bk\}$,
	\item $\Delta(\u', \t) =\bigsqcup_{\bh<\bk}\{e_i- e_j \mid i \in \PI_\bh \sqcup \QI_\bh, j \in \PI_\bk \sqcup \QI_\bk\}$.
	\end{itemize}
	Then the $\PI_\bk \sqcup \QI_\bk$ block of $\delta'_G \in \t^* \cong \C^n$ is
	\[\left(\frac{n+1}2 - \sum_{c < \bk} m_\bh -1, \cdots,
	\frac{n+1}2 - \sum_{c < \bk} m_\bh -m_\bk \right).\]
	Take the restriction along
	\[\t' = \bigoplus_{\bk=1}^r \t_\bk^\Delta \oplus \t_{\MI_\bk} \hookrightarrow \t,\]
	and one can check that 
	\begin{itemize}
	\item on $\C^{l_\bk} \cong \t_\bk^\Delta$, $\Lambda' + \delta'_G$ has all entries equal to
	\begin{equation}\label{eq: chi}
	\left(a_\bk + \frac{n+1} 2 - \sum_{\bh<\bk } m_\bh - 1 \right)+
	\left(a_\bk + \frac{n+1} 2 - \sum_{\bh<\bk} m_\bh  - m_\bk \right) = k_\bk,
	\end{equation}
	\item and on $\C^{\MI_\bk} \cong \t_{\MI_\bk}$, $\Lambda' + \delta'_G$ takes the form
	\begin{equation}\label{eq: mu}
	\left( \frac{k_\bk + m_\bk - 2l_\bk +1} 2 -1, \cdots,
	 \frac{k_\bk + m_\bk - 2l_\bk +1} 2 -(m_\bk - 2l_\bk)\right).
	 \end{equation}
	\end{itemize}
	Dentote $(\Lambda' + \delta'_G)|_{\t^\Delta}$ by $\chi$, and $(\Lambda' + \delta'_G)|_{\t_\MI}$ by $\mu$.
	Since $M \cong U(p-l, q-l) \times (S^1)^{l}$, by the product property of cohomological induction,
	\[\xi = \pi_{(\Lambda' +\delta'_G)|_{\t'}} \cong \pi_{\mu} \boxtimes \chi,\]
	where $\pi_\mu$ is the discrete series of $U(p-l, q-l)$ with Harish-Chandra parameter $\mu$, and $\chi \in (\t^\Delta)^*$ is lifted to a character of $(S^1)^{l_1}\times \cdots \times (S^1)^{l_r}$, still denoted by $\chi$.
\end{subequations}

\subsection{Complete Langlands parameter for $\pi(\psi^\R; \underline p, \underline q)$} \label{subsec: complete-para}
	We have seen $\pi(\psi^\R; \underline p, \underline q) = A_{\LAq'}(\Lambda')$ is the Langlands quotient of $\Ind_{MAN}^G(\xi \boxtimes \nu)$,
	where $MA\cong U(p-l, q-l) \times (S^1)^l \times \R_{>0}^l$, and
	\[\xi \boxtimes \nu \cong \pi_\mu \boxtimes \chi \boxtimes \nu.\]
	It remains to read off the complete Langlands parameter $(\phi^\R, \epsilon)$ from the above standard module.

\subsubsection{Non-tempered part}
	According to \eqref{eq: nu} and \eqref{eq: chi}, the character $\chi \boxtimes \nu$ of $(S^1)^l \times \R^l_{>0} \cong (\C^{\times})^l = (\C^{\times})^{l_1} \times \cdots \times (\C^{\times})^{l_r}$ can be rewritten as
	\[\bigboxtimes_{\bk=1}^r \bigboxtimes_{j=0}^{l_\bk-1} (\frac z{\bar z})^{\frac {k_\bk} 2} (z \bar z)^{\frac{m_\bk -1} 2 - j}.\]
	Under the local Langlands correspondence, it determines the following $l$-dimensional representation of $W_\C = \C^\times$:
	\[\bigoplus_{\bk=1}^r \bigoplus_{j=0}^{l_\bk - 1}
		(\frac z{\bar z})^{\frac {k_\bk} 2} (z \bar z)^{\frac{m_\bk -1} 2 - j}.\]
	The conjugate dual of this representation is
	\[\bigoplus_{d=1}^r \bigoplus_{j=0}^{l_\bk - 1}
		(\frac z{\bar z})^{\frac {k_\bk} 2} (z \bar z)^{-\frac{m_\bk -1} 2 + j}.\]
	The above two representations provide the non-tempered part of ${BC}(\phi^\R)$, as presented in Theorem \ref{Langlands-sub}.

\subsubsection{Discrete series part}
	According to Subsection \ref{subsec: discrete-series}, to write out the complete Langlands parameter $(\phi_-^\R, \epsilon_-)$ of $\pi_\mu$, 
	we have to find $\mu_\Std \in \t_\MI^*$ dominant with respect to the standard Borel in $\gl_{|\MI|}$, and 
	$\sigma \in W(\gl_{|\MI|}, \t_\MI)$ such that $\mu = \sigma \mu_{\Std}$.
	Here $\gl_{|\MI|}$ is identified with endomorphisms over $\C^{|\MI|} \cong \t_\MI \subseteq \t \cong \C^n$ and is regarded as a subalgebra of $\gl_{n}$ in the natural way.
	
	Since $k_\bk- m_\bk \geqslant k_{\bk+1} + m_{\bk+1}$ for $1 \leqslant d \leqslant r-1$, then it follows from \eqref{eq: mu} that $\mu_{\Std} \in \t_\MI^* \cong \C^{|\MI|}$ takes the form
	\[\left( \frac{k_1}2 + \frac{m_1 - 2l_1 -1}2, \cdots, \frac{k_1}2 - \frac{m_1 - 2l_1-1}2,
	\cdots,
	\frac{k_r}2 + \frac{m_r - 2l_r -1}2, \cdots, \frac{k_r}2 - \frac{m_r - 2l_r -1} 2 \right).\]
	By \eqref{eq: Langlands-parameter}, $\phi_-^\R$ has base change
	\[{BC}(\phi_-^\R) = \bigoplus_{\bk=1}^r 
	\left(( \frac z{\bar z})^{\frac{ k_\bk+ (m_\bk-2l_\bk-1)} 2} 
	\oplus \cdots \oplus
	 ( \frac z{\bar z})^{\frac{ k_\bk- (m_\bk-2l_\bk-1)} 2} \right),\]
	as in Theorem \ref{Langlands-sub}.
	 
	Take the partition $\MI= \Delta_1 \sqcup \cdots \sqcup \Delta_r$ such that $|\Delta_\bk| = |\MI_\bk|$, and $\Delta_1< \cdots < \Delta_r$.
	The $\Delta_\bk$ block of $\mu_\Std$ is exactly $\MI_\bk$ block of $\mu$.
	Take the permutation $\sigma$ of $\MI$ which sends indices in $\Delta_\bk$ to indices in $\MI_\bk$, and preserves their order for each $\bk=1, \cdots, r$,
	then we have $\mu = \sigma \mu_\Std$:
	by our notation, $\forall i \in \MI$, the entry of $\sigma \mu_\Std$ at $i$ is $\mu_{\Std, \sigma^{-1}(i)}$, so
	$(\sigma\mu_{\Std})_{\MI_\bk} = \mu_{\Std, \sigma^{-1}(\MI_\bk)} = \mu_{\Std, \Delta_\bk} = \mu_{\MI_\bk}$.
	Under $\Pi_{\phi_-^\R}(U(p-l, q-l)) \to \D  A_{\phi_-^\R} \cong \{\pm 1 \}^{|\MI|}$ given by \eqref{component-character}, $\pi_{\sigma \mu_\Std}$ is sent to $\epsilon_- = (\epsilon_{-, i})_{i \in \MI}$ with
	\[\epsilon_{-, i} = \left\{\begin{aligned}
	&(-1)^{|\MI|- c_i},	&\sigma(i) \leqslant p,\\
	&(-1)^{|\MI|+1-c_i},	&\sigma(i)>p,
	\end{aligned}\right.\]
    where $1 \leqslant c_i \leqslant |\MI|$ is the position of $i$ in $\MI$. Since $\sigma(\Delta_\bk) = \MI_\bk$,  for $i \in \Delta_\bk$, whether $\sigma(i)\leqslant p$ or not is determined by $\MI_\bk \subseteq \PI_\bk$ or $\QI_\bk$.
	Thus, $\epsilon_- \in \{\pm 1\}^{|\MI}$ can be further described as follows:
	\begin{itemize}
	\item it alternates on each $\Delta_\bk$;
	\item it takes value $(-1)^{n - (m_1+\cdots+ m_{\bk-1})} \sgn(p_\bk - q_\bk)$ at the beginning of $\Delta_\bk$;
	\item it takes value $(-1)^{n+1 - (m_1+ \cdots + m_\bk)} \sgn(p_\bk - q_\bk) = \eta_\bk$ at the ends of $\Delta_\bk$.
	\end{itemize}
	This also coincides with the result of Theorem \ref{Langlands-sub};
	note that the indexing 
	$\Delta_\bk \to \mu_{\Std, \Delta_\bk} = [\frac {k_\bk}2 - \frac{m_\bk - 2l_\bk -1}2, \frac{k_\bk}2 + \frac{m_\bk - 2l_\bk -1} 2] = [B_\bk + l_\bk, A_\bk - l_\bk]$ 
	is order-reversing.

\section{Parameterization of orbits in $\D K \backslash \D G / \D P$}\label{Appendix: Characterization of orbtis}

We follow the notations in 2.5 of Subsection~\ref{subsec: LLC real}. Let $\D{G}=\GL_n(\BC)$. Identify $W^{\D{G}}$ with the symmetric group $S_n$ by requiring $e_{w(i)} = w(e_i)$ for the standard basis $e_1, \cdots, e_n$ of $\hat{\mathfrak{t}}$. Let $\lambda = (\lambda_1, \cdots, \lambda_n)$ be an integral dominant infinitesimal character of $G$. We first consider the good parity case, i.e. $\lambda \in \BZ^n$ if $n$ is odd and $\lambda \in(\frac12 \BZ \backslash \BZ)^n$ if $n$ is even.

We denote
\begin{equation}\label{eqn: I-lambda-good}
\mathfrak{I}_{\lambda}:=\{s \in W^{\D{G}} \, | \, s^{2} = 1, \text{ and } s(i) = i\text{ if } \lambda_i = \lambda_{s(i)} \}.
\end{equation}
Note that $W^{\D M(\lambda)}$ acts on $\mathfrak{I}_{\lambda}$
by conjugation. 
\begin{lemma}
\label{lemma: Weyl group parametrization-surjective}
There is a surjection 
\begin{equation}
\label{eq: Weyl group parametrization}
\mathfrak{I}_{\lambda}/W^{\D{M}(\lambda)} \surjects \operatorname{Sym}_n(\mathbb{C})_{reg} / \D{P}(\lambda)
\end{equation}
by sending $s$ to the permutation matrix $\dot s$ representing $s$. 

\end{lemma}

\begin{proof}
Let $\widehat{B}$ be the subgroup of upper-triangular matrices in $\widehat{G}$. It follows from \cite[\textsection 10.2]{RS:1990} that we have a bijection
\[
\{s \in W^{\D G} \, | \, s^{2} = 1\} \bijects \operatorname{Sym}_n(\mathbb{C})_{reg} / \D{B}, \quad s \mapsto \dot s
\]
which is order reversing where the left hand side is equipped the Bruhat order.
For $s \in W^{\D G}$ satisfying $s^{2} = 1$, if there exists $i$ such that $s(i) \neq i$ and $\lambda_i = \lambda_{s(i)}$, then $\dot s$ contains the block 
$\begin{pmatrix}
	0&1\\1&0
\end{pmatrix}$
on the rows and columns indexed by $\{i,s(i)\}$. This block is congruent, under the $\{i,s(i)\}$-block $\GL_2(\BC) \subset \D M(\lambda)$, to the identity matrix. In other words, there exists $g \in \D M(\lambda)$ such that ${}^t\!g\dot sg = \dot s'$, where $s' \in W^{\D G}$ satisfies $s'(j) = s(j)$ for $j \neq i$ or $s(i)$, and $s'(i) = i$. Repeating this process for the $i$'s as above, we see eventually that $\dot s$ is $\D P(\lambda)$-congruent to an element coming from $\mathfrak{I}_\lambda$. This shows that we have a surjection
\begin{equation*}
	\mathfrak{I}_\lambda \surjects \operatorname{Sym}_n(\mathbb{C})_{reg} / \D{P}(\lambda).
\end{equation*}
Finally, if $s \in \mathfrak{I}_\lambda$ is replaced by $wsw\inv$ with $w \in W^{\D M(\lambda)}$, then letting $\dot w$ be a permutation lift of $w$, we have ${}^t \dot w = \dot w\inv$. Hence $\dot w \dot s \dot w\inv = \dot w \dot s \,{}^t\! \dot w$ is in the same $\D P(\lambda)$-congruence class as $\dot s$, and the surjection above descends to (\ref{eq: Weyl group parametrization}).
\end{proof}

\begin{lemma}
\label{lemma:W-orbit-by-Auv}
Write the index set as a disjoint union
of blocks
\[
\{1,\dots,n\}=J_1\sqcup\cdots\sqcup J_r,
\qquad
\lambda \text{ is constant on each }J_u,
\]
so that
\[
W^{\D M(\lambda)}\;\cong\;\operatorname{Sym}(J_1)\times\cdots\times \operatorname{Sym}(J_r),
\]
acting on $S_n$ by conjugation.

For $s\in \mathfrak{I}_\lambda$, set
\[
A_{u,v}(s):=\#\{\, i\in J_u:\ s(i)\in J_v\,\}\qquad (1\le u,v\le r).
\]
Then two elements $s,t\in \mathfrak{I}_\lambda$ are conjugate under $W^{\D M(\lambda)}$ if and only if
\[
A_{u,v}(s)=A_{u,v}(t)\qquad \text{for all }u,v.
\]
\end{lemma}

\begin{proof}
Let $w\in W^{\D M(\lambda)}$. By definition, $w$ preserves each block $J_u$ setwise.
For any $u,v$ we have a bijection
\[
J_u \xrightarrow{\;w\;} J_u,\qquad i\longmapsto w(i),
\]
and
\[
(ws w^{-1})(w(i)) = w(s(i)).
\]
Thus $s(i)\in J_v$ if and only if $(wsw^{-1})(w(i))\in J_v$, and counting gives
$A_{u,v}(wsw^{-1})=A_{u,v}(s)$.

For $s\in \frak \mathfrak{I}_\lambda$,  it follows from the definition that 
if $s(i)\in J_u$ and $i\in J_u$, then $s(i)=i$.
Hence every $s\in \mathfrak{I}_\lambda$ is a product of:
(i) fixed points lying in each $J_u$, and
(ii) transpositions pairing elements of two distinct blocks.

Define the disjoint subsets (for fixed $s$)
\[
F_u(s):=\{\,i\in J_u:\ s(i)=i\,\},
\qquad
E_{u,v}(s):=\{\,i\in J_u:\ s(i)\in J_v\,\}\quad (u\neq v).
\]
Then
\[
J_u = F_u(s)\ \sqcup\ \bigsqcup_{v\neq u} E_{u,v}(s),
\qquad
|F_u(s)|=A_{u,u}(s),\quad |E_{u,v}(s)|=A_{u,v}(s).
\]
Moreover, since $s$ is an involution, the map $s$ restricts to a bijection
\[
s:\ E_{u,v}(s)\xrightarrow{\sim} E_{v,u}(s)\qquad (u\neq v).
\]

Assume $A_{u,v}(s)=A_{u,v}(t)$ for all $u,v$, where $s,t\in \frak \mathfrak{I}_\lambda$.
\begin{itemize}
\item[(a)]
For each $u$, choose an arbitrary bijection
\[
\beta_u:\ F_u(s)\xrightarrow{\sim} F_u(t)
\]
which exists because $|F_u(s)|=A_{u,u}(s)=A_{u,u}(t)=|F_u(t)|$.

\item[(b)]
For each ordered pair $(u,v)$ with $u<v$, choose an arbitrary bijection
\[
\alpha_{u,v}:\ E_{u,v}(s)\xrightarrow{\sim} E_{u,v}(t)
\]
which exists because $|E_{u,v}(s)|=A_{u,v}(s)=A_{u,v}(t)=|E_{u,v}(t)|$.
Now define a bijection
\[
\alpha_{v,u}:\ E_{v,u}(s)\xrightarrow{\sim} E_{v,u}(t)
\]
by the rule
\[
\alpha_{v,u}(s(i)) \ :=\ t(\alpha_{u,v}(i))\qquad (i\in E_{u,v}(s)).
\]
This is well-defined and bijective because $s:E_{u,v}(s)\to E_{v,u}(s)$ and
$t:E_{u,v}(t)\to E_{v,u}(t)$ are bijections.

\item[(c) ]
For each $u$, define a map $w_u:J_u\to J_u$ by:
\[
w_u|_{F_u(s)}:=\beta_u,
\qquad
w_u|_{E_{u,v}(s)}:=\alpha_{u,v}\ \ \ (v\neq u).
\]
Since the sets $F_u(s)$ and $E_{u,v}(s)$ (for $v\neq u$) partition $J_u$ and the target
sets $F_u(t)$ and $E_{u,v}(t)$ partition $J_u$ with matching cardinalities, $w_u$ is a
bijection. Hence $w_u\in \operatorname{Sym}(J_u)$.
\end{itemize}

Let
\[
w := w_1\cdots w_r \ \in\ \operatorname{Sym}(J_1)\times\cdots\times \operatorname{Sym}(J_r)\ \cong\ W^{\D M(\lambda)}.
\]

It suffices to check the intertwining identity
\[
w\circ s = t\circ w
\quad\text{on }\{1,\dots,n\}.
\]

- If $i\in F_u(s)$, then $s(i)=i$, and by construction $w(i)\in F_u(t)$, so
$t(w(i))=w(i)=w(s(i))$.

- If $i\in E_{u,v}(s)$ with $u\neq v$, then $s(i)\in E_{v,u}(s)$ and, by the definition
of $\alpha_{v,u}$,
\[
w(s(i)) = w_v(s(i)) = \alpha_{v,u}(s(i))
= t(\alpha_{u,v}(i)) = t(w_u(i)) = t(w(i)).
\]

Thus $w\circ s=t\circ w$, hence $w s w^{-1}=t$. This shows $s$ and $t$ are
$W^{\D M(\lambda)}$-conjugate.
This finishes the proof.
\end{proof}

\begin{proposition}\label{prop: Weyl group parametrization}
The map \eqref{eq: Weyl group parametrization} is a bijection.
\end{proposition}

\begin{proof}
Write the index set $\{1,\dots,n\}$ as a disjoint union of blocks
\[
\{1,\dots,n\}=J_1\sqcup\cdots\sqcup J_r
\]
so that $\lambda$ is constant on each $J_a$ and takes distinct values on distinct blocks.
For each $a=1,\dots,r$, set $m_a:=|J_1|+\cdots+|J_a|$ and define the standard
partial flag subspaces
\[
F_a:=\Span(e_1,\dots,e_{m_a})\subset \BC^n,
\]
so that $\D P(\lambda)$ is precisely the stabilizer of the partial flag
$0\subset F_1\subset\cdots\subset F_r=\BC^n$. Given $X\in\operatorname{Sym}_n(\mathbb{C})_{reg}$,
let $B_X(u,v):=u^tXv$ be the nondegenerate symmetric bilinear form attached to $X$.
For  $a,b\in\{1,\dots,r\}$, define the integers
\[
d_{a,b}(X):=\dim\bigl(F_a\cap F_b^{\perp_X}\bigr),
\qquad F_b^{\perp_X}:=\{v\in \BC^n: B_X(v,F_b)=0\}.
\]
\emph{\textbf{Claim:}} $d_{a,b}(X)$ is invariant under $P(\lambda)$-congruence.

Indeed, for $X'=p^tXp$ one has $B_{X'}(u,v)=B_X(pu,pv)$, hence
$F_b^{\perp_{X'}} = p^{-1}(F_b^{\perp_X})$ because $pF_b=F_b$ for $p\in P(\lambda)$.
Therefore
\[
d_{a,b}(X')=\dim(F_a\cap p^{-1}(F_b^{\perp_X}))
=\dim(pF_a\cap F_b^{\perp_X})=\dim(F_a\cap F_b^{\perp_X})=d_{a,b}(X).
\]

Now specialize to $X=\dot s$ for $s\in \mathfrak{I}_\lambda$. Since $B_{\dot s}(e_i,e_j)=1$
iff $j=s(i)$, one checks:
\[
e_i\in F_b^{\perp_{\dot s}}
\iff B_{\dot s}(e_i,F_b)=0
\iff s(i)>m_b.
\]
Hence
\begin{equation}\label{eq:dab-dot-s}
d_{a,b}(\dot s)=\#\{\, i\le m_a:\ s(i)>m_b\,\}.
\end{equation}

Define also
\[
c_{a,b}(s):=\#\{\, i\le m_a:\ s(i)\le m_b\,\}
= m_a-d_{a,b}(\dot s).
\]
From the matrix $(c_{a,b}(s))_{a,b}$ we can recover for each pair of blocks $u,v$,
the numbers $A_{u,v}(s)$ from Lemma \ref{lemma:W-orbit-by-Auv}
via the inclusion--exclusion formula
\[
A_{u,v}(s)
=\bigl(c_{u,v}(s)-c_{u,v-1}(s)\bigr)-\bigl(c_{u-1,v}(s)-c_{u-1,v-1}(s)\bigr),
\]
where we set $c_{0,b}=c_{a,0}=0$.

For $s\in \mathfrak{I}_\lambda$, we have $A_{u,u}(s)=\#\{\text{fixed points of }s\text{ in }J_u\}$
and for $u\ne v$ the integer $A_{u,v}(s)$ counts transpositions pairing $J_u$ with $J_v$,
and is symmetric: $A_{u,v}(s)=A_{v,u}(s)$.

Suppose $\dot s$ and $\dot t$ (with $s,t\in \mathfrak{I}_\lambda$) lie in the same $\D P(\lambda)$-orbit.
Then $d_{a,b}(\dot s)=d_{a,b}(\dot t)$ for all $a,b$ by Claim A, hence
$c_{a,b}(s)=c_{a,b}(t)$ and therefore $A_{u,v}(s)=A_{u,v}(t)$ for all $u,v$.
By  Lemma \ref{lemma:W-orbit-by-Auv}, $s$ and $t$ are conjugate by $W^{\D M(\lambda)}$. Thus the map $\mathfrak{I}_\lambda/W^{\D M(\lambda)}\to \operatorname{Sym}_n(\C)_{ reg}/\D P(\lambda)$ is injective.
Combined with surjectivity from Lemma \ref{lemma: Weyl group parametrization-surjective}, it is a bijection.

\end{proof}

We now turn to the bad parity case, i.e., $n$ is even and $\lambda \in \BZ^n$.
We say that $s\in S_n$ is  fixed point free if for any $i\in \{1, \cdots, n\}$, 
we have $s(i)\neq i$.
We denote by $\frak{J}^S_{\D{G}}$ the set of fixed point free involutions in $W^{\D G}$. Note that $W^{\D M(\lambda)}$ acts on $\frak{J}^S_{\D{G}}$
by conjugation.

\begin{lemma}
\label{lemma: Weyl group parametrization-skew-surjective}
There is a surjection 
\begin{equation}
\label{eq: Weyl group parametrization-sign-surjective}
\frak{J}^S_{\D{G}}/W^{\D M(\lambda)} \surjects \operatorname{Skew}_n(\mathbb{C})_{reg} / \D{P}(\lambda)
\end{equation}
by sending $s$ to the matrix $\dot s $, where 
$\dot s$ is a signed permutation matrix lifting $s$ satisfying 
\[
{}^t\!\dot s= -\dot s .
\]
\end{lemma}

\begin{proof}
Let $\widehat{B}$ be the subgroup of upper-triangular matrices in $\widehat{G}$. It follows from \cite[\textsection 10.2]{RS:1990} that we have a bijection
\[
\{s \in W^{\D G} \, | \, s^{2} = 1, s \text{ is fixed point free}\} \bijects \operatorname{Skew}_n(\mathbb{C})_{reg} / \D{B}, \quad s \mapsto \dot s 
\]
which is order reversing where the left hand side is equipped the Bruhat order. As in the proof of Lemma \ref{lemma: Weyl group parametrization-surjective}, $W^{\D M(\lambda)}$-conjugation on $s$ does not change the resulting $\D P(\lambda)$-congruence class of $\dot s$. Hence the above bijection descends to the surjection (\ref{eq: Weyl group parametrization-sign-surjective}).
\end{proof}

\begin{lemma}
\label{lemma:W-orbit-by-Auv-skew}
Write the index set as a disjoint union
of blocks
\[
\{1,\dots,n\}=J_1\sqcup\cdots\sqcup J_r,
\qquad
\lambda \text{ is constant on each }J_u,
\]
so that
\[
W^{\D M(\lambda)}\;\cong\;\operatorname{Sym}(J_1)\times\cdots\times \operatorname{Sym}(J_r),
\]
acting on $S_n$ by conjugation.

For $s\in \frak{J}^S_{\hat G}$, set
\[
A_{u,v}(s):=\#\{\, i\in J_u:\ s(i)\in J_v\,\}\qquad (1\le u,v\le r).
\]
Then two elements $s,t\in \frak{J}^S_{\hat G}$ are conjugate under $W^{\D M(\lambda)}$ if and only if
\[
A_{u,v}(s)=A_{u,v}(t)\qquad \text{for all }u,v.
\]
\end{lemma}

\begin{proof}
Write $E_{u,v}(s):=\{\,i\in J_u:\ s(i)\in J_v\,\}$. Then $|E_{u,v}(s)|=A_{u,v}(s)$ and $s$
restricts to a bijection $E_{u,v}(s)\xrightarrow{\sim}E_{v,u}(s)$ for all $u,v$.
The condition $s\in \frak{I}_G^S$ means there are no fixed points, but it allows $u=v$,
in which case $s|_{E_{u,u}(s)}$ is a fixed-point-free involution on that subset.

Assume $A_{u,v}(s)=A_{u,v}(t)$ for all $u,v$.

(i) For each $u<v$, choose any bijection $\alpha_{u,v}:E_{u,v}(s)\to E_{u,v}(t)$ and define
$\alpha_{v,u}:E_{v,u}(s)\to E_{v,u}(t)$ by $\alpha_{v,u}(s(i)):=t(\alpha_{u,v}(i))$.
This forces compatibility with the transpositions pairing $J_u$ and $J_v$.

(ii) For each $u$, the restrictions $s|_{E_{u,u}(s)}$ and $t|_{E_{u,u}(t)}$ are both fixed-point-free
involutions on sets of the same cardinality $A_{u,u}$. Hence they are conjugate by some permutation
$\beta_u\in \operatorname{Skew}(E_{u,u}(s))$.

Now define $w_u\in \operatorname{Sym}(J_u)$ by setting
\[
w_u|_{E_{u,v}(s)}:=\alpha_{u,v}\quad (v\neq u),
\qquad
w_u|_{E_{u,u}(s)}:=\beta_u.
\]
These pieces have disjoint domains and cover $J_u$, so $w_u$ is a permutation of $J_u$.
Let $w:=w_1\cdots w_r\in \operatorname{Sym}(J_1)\times\cdots\times \operatorname{Sym}(J_r)\cong W^{\D M(\lambda)}$.
By construction one checks on each $E_{u,v}(s)$ that $w\circ s=t\circ w$, hence
$wsw^{-1}=t$. This proves the lemma.
\end{proof}

\begin{proposition}\label{prop: Weyl group parametrization-skew}
The map \eqref{eq: Weyl group parametrization-sign-surjective} is a bijection.
\end{proposition}

\begin{proof}
Assume $\dot s$ and $\dot t$ lie in the same $\D P(\lambda)$-orbit, so there exists $p\in \D P(\lambda)$ with
\[
p^t\,\dot s\,p=\dot t.
\]
Let $J_1\sqcup\cdots\sqcup J_r$ be the block decomposition of $\{1,\dots,n\}$ into maximal
contiguous blocks on which $\lambda$ is constant, and set $m_a:=|J_1|+\cdots+|J_a|$.
Let $F_a:=\Span(e_1,\dots,e_{m_a})$, so $\D P(\lambda)$ is precisely the stabilizer of the partial flag
$0\subset F_1\subset\cdots\subset F_r=\C^n$.

For $X\in \operatorname{Skew}_n(\C)_{ reg}$ define
\[
d_{a,b}(X):=\dim\bigl(F_a\cap F_b^{\perp_X}\bigr),\qquad
F_b^{\perp_X}:=\{v\in \C^n:\ v^t X F_b=0\}.
\]
Exactly as in the symmetric case, $d_{a,b}(X)$ is invariant under $\D P(\lambda)$-congruence, hence
$d_{a,b}(\dot s)=d_{a,b}(\dot t)$ for all $a,b$.

For $X=\dot s$, one checks that $e_i\in F_b^{\perp_{\dot s}}$ if and only if
$s(i)>m_b$ (the sign in $\dot s$ does not affect whether the pairing is zero).
Hence
\[
d_{a,b}(\dot s)=\#\{\,i\le m_a:\ s(i)>m_b\,\},
\qquad
c_{a,b}(s):=\#\{\,i\le m_a:\ s(i)\le m_b\,\}=m_a-d_{a,b}(\dot s).
\]
Thus $c_{a,b}(s)=c_{a,b}(t)$ for all $a,b$.
As in the proof of Proposition \ref{prop: Weyl group parametrization}, these numbers determine the block-incidence counts
\[
A_{u,v}(s):=\#\{\,i\in J_u:\ s(i)\in J_v\,\},
\]
via inclusion--exclusion from the cumulative counts $c_{a,b}$. Therefore
$A_{u,v}(s)=A_{u,v}(t)$ for all $u,v$.

Suppose $\dot s$ and $\dot t$ with $s,t\in \frak{J}^S_{\hat G}$ lie in the same $\D P(\lambda)$-orbit.
Then $d_{a,b}(\dot s)=d_{a,b}(\dot t)$ for all $a,b$, hence
$c_{a,b}(s)=c_{a,b}(t)$ and therefore $A_{u,v}(s)=A_{u,v}(t)$ for all $u,v$.
By Lemma \ref{lemma:W-orbit-by-Auv-skew}, $s$ and $t$ are conjugate by $W^{\D M(\lambda)}$. Thus the map $\frak{J}^S_{\D G}/W^{\D M(\lambda)}\to \operatorname{Skew}_n(\C)_{ reg}/\D P(\lambda)$ is injective.

\end{proof}

\bibliographystyle{amsalpha}

\bibliography{reps}

\end{document}